\documentclass[final,hidelinks,onefignum,onetabnum]{template}
\usepackage{lipsum}
\usepackage{amsfonts}
\usepackage{amssymb}
\usepackage{bookmark}
\usepackage{bm}
\usepackage{graphicx}
\usepackage{epstopdf}
\usepackage{hyperref}
\usepackage{enumitem}
\usepackage{algorithmic}
\usepackage{booktabs}
\usepackage{multirow}
\usepackage{cleveref}
\usepackage{aliascnt}
\usepackage[caption=false]{subfig}
\usepackage{orcidlink}

\ifpdf
  \DeclareGraphicsExtensions{.eps,.pdf,.png,.jpg}
\else
  \DeclareGraphicsExtensions{.eps}
\fi

\newsiamremark{remark}{Remark}
\newsiamremark{hypothesis}{Hypothesis}
\crefname{hypothesis}{Hypothesis}{Hypotheses}
\newsiamthm{claim}{Claim}
\newsiamremark{fact}{Fact}
\crefname{fact}{Fact}{Facts}

\newaliascnt{example}{theorem}
\newtheorem{example}[example]{\textit{Example}}
\aliascntresetthe{example}
\crefname{example}{Example}{Examples}
\Crefname{example}{Example}{Examples}
\crefname{remark}{Remark}{Remarks}
\Crefname{remark}{Remark}{Remarks}

\crefname{enumi}{part}{parts}
\Crefname{enumi}{Part}{Parts}

\newcommand{\sd}{\mathord{\prime\mkern-2.5mu\reflectbox{$\scriptstyle\prime$}}}
\newcommand\innerprd{\mathbin{\vcenter{\hbox{\scalebox{0.5}{$\bullet$}}}}}

\Crefname{ALC@unique}{Line}{Lines}

\newcommand{\BREAK}{\STATE \textbf{break}}

\headers{Specular ellipse method}{K. Jung}

\title{The specular ellipse method for scalar ordinary differential equations: exactness and accuracy up to fourth order\thanks{Date: \today.
\funding{This work is supported in part by Michigan State University and the National Science Foundation Research Traineeship Program (DGE-2152014).The author received partial support from Jun Kitagawa's National Science Foundation grant (DMS-2246606).}}
}

\author{Kiyuob Jung\,\orcidlink{0009-0007-5075-4061}\thanks{Department of Mathematics, Michigan State University, East Lansing, MI 48824 USA (\email{kyjung@msu.edu}).}
}
\usepackage{amsopn}
\DeclareMathOperator{\sgn}{sgn}

\hypersetup{
  colorlinks=false, pdfborder={0 0 1}, linkbordercolor={1 0 0}, citebordercolor={0 1 0} 
}

\begin{document}

\maketitle

\begin{abstract}
    This paper introduces a family of one-step implicit methods for solving scalar ordinary differential equations.
    At each step, the update uses a scaled angular mean of two vector-field evaluations, and the positive scale may vary from step to step.
    The leading terms of the local truncation error can be expressed in terms of the derivative of the signed curvature of the scaled solution graph.
    We prove that the proposed method reproduces the exact solution at the mesh points when the solution graph has constant signed curvature under a fixed positive scaling and each implicit update is unique.
    When this special geometric condition is not satisfied, we establish second-order consistency and convergence for positive scale sequences satisfying suitable uniform conditions.
    Furthermore, third- and fourth-order consistency and convergence can be achieved by choosing the scale to cancel the relevant curvature terms in the local truncation error.
    Using only the given problem data, we classify when these improvements are possible and determine the corresponding scale choices.
    An example shows that the proposed fourth-order method can yield smaller errors than the classical fourth-order Runge--Kutta method at the same step size.
\end{abstract}



\section{Introduction}

Let $\Omega\subset\mathbb{R}^2$ be open, and let $F:\Omega\to\mathbb{R}$ be given.
Fix $t_0 \geq 0$, $T > t_0$, and $u_0\in\mathbb{R}$ such that $(t_0,u_0)\in\Omega$.
Consider the initial value problem
\begin{equation}
    \label{ODE:IVP}
    \begin{cases}
        u'(t)=F\bigl(t,u(t)\bigr),
        & t\in(t_0,T),\\
        u(t_0)=u_0.
    \end{cases}
\end{equation}
The unknown is a function $u:[t_0,T]\to\mathbb{R}$ satisfying $\bigl(t,u(t)\bigr)\in\Omega$ for every $t\in[t_0,T]$.

The numerical method studied in this paper can be written in the form
\begin{displaymath}
    u_{n+1}
    =
    u_n + h \, \mathcal{M}_n \left( F(t_{n+1},u_{n+1}), F(t_n,u_n) \right),
\end{displaymath}
where $h>0$ is the step size, $t_n=t_0+nh$, and $\mathcal{M}_n:\mathbb{R}^2\to\mathbb{R}$ combines the two vector-field evaluations.
The forward Euler, backward Euler, and Crank--Nicolson (or trapezoidal) methods arise from particular linear choices of $\mathcal M_n$.
For example, in the Crank--Nicolson method (CN), $\mathcal{M}_n$ is the arithmetic mean of the two vector-field values.
For standard definitions and notation concerning numerical methods for ordinary differential equations, we follow \cite{2016_Butcher_BOOK}.

In this paper, we introduce the \emph{specular ellipse} method (SE), a nonlinear one-step implicit method defined through the specular derivative introduced in \cite{2026a_Jung}.
The conclusion of this paper is that, for a given problem \eqref{ODE:IVP}, $F$ and its derivatives provide sufficient criteria for exact reproduction and for convergence of order two, three, or four under the corresponding assumptions.

The specular derivative is obtained by converting the forward and backward difference quotients into angles and then taking the slope corresponding to the arithmetic mean of these angles.
We incorporate this geometric construction into SE by taking $\mathcal{M}_n$ to be a $\sigma$-scaled angular mean of the consecutive vector-field values $F(t_n, u_n)$ and $F(t_{n+1}, u_{n+1})$, where $\sigma > 0$ is a scale parameter.
The scale may vary from step to step and is then denoted by $\sigma_n$.
Geometrically, $\sigma$ scales the vertical coordinate of the solution graph, while the angular mean chooses each numerical chord along the angular bisector of the corresponding tangent directions, yielding exact reproduction when the scaled graph has constant signed curvature and each implicit update is unique.

Nonlinear choices of $\mathcal{M}_n$ based on classical means have also been studied.
Evans and Sanugi introduced nonlinear one-step formulas based on the geometric mean \cite{1987_Evans} and later compared nonlinear trapezoidal formulas based on arithmetic, geometric, harmonic, logarithmic, and contra-harmonic means \cite{1991_Evans}.
A related nonlinear integration formula was proposed in \cite{1996_Sivakumar}.
Beyond these one-step formulas, nonlinear means were also incorporated into higher-order Runge--Kutta formulas \cite{1989a_Evans,1994_Sanugi,1995_Evans}.
In SE, the $\sigma$-scaled angular mean defining $\mathcal{M}_n$ differs from these classical means.

Later, Villatoro studied the exactness, local error, and stability of nonlinear one-step methods based on classical means.
Exactness conditions for these methods when applied to autonomous scalar ODEs were characterized in \cite{2008_Villatoro}.
A procedure for finding ODEs exactly solved by a given one-step method was developed in \cite{2009_Villatoro}.
Our exactness result instead gives a geometric sufficient condition in terms of constant signed curvature of the scaled solution graph.
Local truncation errors for Evans--Sanugi methods based on differentiable homogeneous means were analyzed in \cite{2010b_Villatoro}.
Stability of nonlinear one-step methods based on homogeneous means was studied in \cite{2010a_Villatoro}.
However, this stability analysis is not applicable to SE, since the $\sigma$-scaled angular mean is not homogeneous.
We therefore study the stability of SE for the Dahlquist test equation with a negative real coefficient.

Apart from the preceding constructions, another approach fixes the parametric form of a numerical method and then selects a problem-dependent parameter from the differential equation or its local truncation error.
A geometrically related use of parameter selection appears in phase-fitted variational integrators, where a local frequency is estimated from the curvature and speed of the trajectory at each step and used in frequency-dependent trigonometric interpolation \cite{2010_Kosmas,2019_Kosmas}.
In exponentially fitted two-step BDF methods, Ixaru et al. determined the fitting frequency from $F$ and its total derivatives to cancel the leading term of the local truncation error and improve the order from two to three \cite{2002_Ixaru}.
In exponentially fitted two-step Runge--Kutta methods, the fitting parameter was selected by annihilating or minimizing the leading local discretization error \cite{2012_DAmbrosio}.
A related cancellation mechanism appears in a nonlinear rational one-step family, where an appropriate parameter choice cancels the principal local truncation error and gives a third-order method \cite{2015_Ramos}.
Since the leading terms in the local truncation error involve the variation of the signed curvature of the $\sigma$-scaled solution graph, SE selects its scale from $F$ and its derivatives to cancel these terms and improve the convergence order from two to three or four.

\subsection{Definitions and notation}

We write $\left\| \, \cdot \, \right\|_{\mathbb{R}^n}$ for the Euclidean norm on $\mathbb{R}^n$.
If a constant $C$ depends on parameters $p_1, p_2, \ldots, p_m$, we indicate this dependence by writing $C \equiv C(p_1, p_2, \ldots, p_m)$.

We recall the definition of the specular derivative for a real-valued function of one real variable as given in \cite{2026a_Jung}.
Let $I$ be an open interval in $\mathbb{R}$, let $u:I\to\mathbb{R}$, and fix $t\in I$.
The specular derivative of $u$ at $t$ is defined by
\begin{displaymath}
    u^{\sd}(t)
    :=
    \lim_{h\searrow0}
    \mathcal{A}\bigl(u(t+h)-u(t),u(t)-u(t-h),h\bigr),
\end{displaymath}
provided that the limit exists as a finite real number, where
\begin{displaymath}
    \mathcal{A}(a,b,c)
    :=
    \frac{a\sqrt{b^2+c^2}+b\sqrt{a^2+c^2}}
    {c\sqrt{a^2+c^2}+c\sqrt{b^2+c^2}}, \qquad
    (a, b, c) \in \mathbb{R} \times \mathbb{R} \times (0, \infty).
\end{displaymath}
For $\alpha,\beta\in\mathbb{R}$, define
\begin{align*}
    \mathcal{B}(\alpha,\beta)
    &:= 
    \tan\left( \frac{1}{2}\arctan(\alpha) + \frac{1}{2}\arctan(\beta) \right),   \\
    \mathcal{C}(\alpha,\beta)
    &:=
    \left( \frac{\alpha}{\sqrt{1+\alpha^2}} + \frac{\beta}{\sqrt{1+\beta^2}} \right) \left( \frac{1}{\sqrt{1+\alpha^2}} + \frac{1}{\sqrt{1+\beta^2}} \right)^{-1}.
\end{align*}
Note that $\alpha$ and $\beta$ represent the right- and the left-hand derivatives, respectively.
The identity
\begin{displaymath}
    \mathcal{A}(a,b,c)
    =
    \mathcal{B}\left(\frac{a}{c},\frac{b}{c}\right)
    =
    \mathcal{C}\left(\frac{a}{c},\frac{b}{c}\right)
\end{displaymath}
holds for every $a,b\in\mathbb{R}$ and $c>0$; see \cite[Lem.~2.1]{2026a_Jung}.
Suppose that both one-sided derivatives exist as finite real numbers.
By \cite[Thm.~2.14]{2026a_Jung}, we have 
\begin{align}
    u^{\sd}(t)
    &= \mathcal{B}\left(\partial^+u(t),\partial^-u(t)\right),  \label{eq:repr_B} \\
    &= \mathcal{C}\left(\partial^+u(t),\partial^-u(t)\right),   \label{eq:repr_C}
\end{align}
where $\partial^+u(t)$ and $\partial^-u(t)$ are the right- and left-hand derivatives of $u$ at $t$.
We refer to \cite[App.~A]{2026a_Jung} for a detailed analysis of the functions $\mathcal{A}$, $\mathcal{B}$, and $\mathcal{C}$.

We use the following notation.
Let $t_0$, $T$ be fixed such that $0 \leq t_0 < T$.
For $h > 0$, we define the nonnegative integer $N_h := \lfloor h^{-1}(T-t_0)\rfloor$.
Given a numerical approximation $\{u_n\}_{n=0}^{N_h}$ to a solution
$u:[t_0,T]\to\mathbb{R}$, we write
\begin{displaymath}
    t_n:=t_0+nh,
    \qquad
    e_n:=u(t_n)-u_n,
    \qquad
    n=0,1,\ldots,N_h.
\end{displaymath}
For $k \in \mathbb{N}$ and $u\in C^k([t_0,T])$, we use the notation
\begin{displaymath}
    M_j
    :=
    \max_{t\in[t_0, T]}\left|u^{(j)}(t)\right|,
    \qquad
    j=1, 2, \ldots, k.
\end{displaymath}
For a solution $u:[t_0,T]\to\mathbb{R}$ of \eqref{ODE:IVP} and a positive real number $\rho > 0$, we denote the closed vertical $\rho$-tube about the graph
of $u$ by
\begin{displaymath}
    U_\rho(u)
    :=
    \left\{ (t,x) \in [t_0,T] \times \mathbb{R} \,\middle|\, |x-u(t)| \leq \rho \right\}.
\end{displaymath}

We use the following hypotheses.

\begin{enumerate}[label=\upshape(H\arabic*), ref=(H\arabic*)]
    \item
    \label{H1}
    \textup{(Unique solvability and Lipschitz continuity)}
    Let $u:[t_0,T]\to\mathbb{R}$ be the unique solution of \eqref{ODE:IVP}.
    Suppose that there exist constants $\rho>0$ and $L\geq0$
    such that $U_\rho(u)\subset\Omega$ and
    \begin{displaymath}
        |F(t,x)-F(t,y)|
        \leq
        L|x-y|
    \end{displaymath}
    whenever $(t,x),(t,y)\in U_\rho(u)$.

    \item 
    \label{H2}
    \textup{(Uniform nondegeneracy)}
    Let $u:[t_0,T]\to\mathbb{R}$ be a solution of \eqref{ODE:IVP}.
    Suppose that there exist constants $H>0$ and $\eta>0$ such that, for every $h\in(0,H]$, a sequence $\{\sigma_n\}_{n=0}^{N_h-1}$ of positive scales is given such that
    \begin{displaymath}
        \sigma_n^2+u'(t_n)^2 \geq \eta
    \end{displaymath}
    for every $n=0,1,\ldots,N_h-1$.
\end{enumerate}

\subsection{Main results}

For the specular ellipse method, we establish exactness with suitable fixed scales, convergence with prescribed scales, and convergence with numerically selected scales.

First, we show that the angle-bisecting geometry is retained by the discrete update.
If a fixed positive scaling makes the exact solution graph have constant signed curvature and each implicit update has a unique admissible solution, then the numerical values coincide with the exact solution at all mesh points.
This includes elliptic solution graphs after an appropriate scaling; see \cref{thm:SE_exact_constant_curvature}.

Second, under hypotheses~\ref{H1} and~\ref{H2} and the stated smoothness assumptions, SE generates a unique numerical sequence and converges with order two for all sufficiently small step sizes; see \cref{cor:consistency_of_SE,cor:convergence_orders_SE}.
The leading term in the local truncation error is determined by the derivative of the signed curvature of the scaled solution graph.
Prescribed scales cancelling the contribution at the current point or the combined contributions at the current and next points yield third- and fourth-order consistency, respectively, and the corresponding convergence orders under the required uniform bounds.

Third, we show that the scale choices used to achieve third- and fourth-order convergence are not merely theoretical: they can be implemented numerically without a closed-form expression for the exact solution.
Depending on the construction, the scales are determined from the step size or computed from $F$, its derivatives, and the numerical values.
Under the stated uniform nondegeneracy and regularity assumptions, the resulting numerical methods retain the corresponding convergence orders; see \cref{thm:third_order_SE,thm:G4_vanishing_scale_SE,thm:SE4}.

\subsection{Organization}

\Cref{sec:numerical_methods} introduces the specular Euler family and the specular ellipse method.
\Cref{sec:geometry_exactness_SE} develops its geometric interpretation and establishes exactness for scaled solution graphs of constant signed curvature.
\Cref{sec:convergence} proves second-order convergence, conditional third- and fourth-order convergence, and decay on the negative real Dahlquist equation.
\Cref{sec:scale_third_order,sec:scale_fourth_order} classify the field-based scale-selection problems and establish the corresponding third- and fourth-order convergence results.
\Cref{sec:numerical} illustrates the main results, while \cref{apx:C_sigma,apx:auxiliary_results} collect the estimates used in the analysis.

\section{Numerical methods}
\label{sec:numerical_methods}

This section introduces numerical methods for \eqref{ODE:IVP} based on discrete approximations of the specular derivative $u^{\sd}$.

\subsection{The specular Euler method}

From \eqref{eq:repr_C}, we suggest the following numerical procedure.

\begin{definition}
    Consider the problem \eqref{ODE:IVP}.
    The \emph{specular Euler} method generates a sequence $\left\{ u_n \right\}_{n=0}^{N_h}$ according to the formula
    \begin{equation} \label{mtd:specular_Euler}
    u_{n+1} = u_n + h \, \mathcal{C}( \alpha_n, \beta_n ),
    \end{equation}
    for $n=0,1,\ldots,N_h-1$, where $t_n := t_0+nh$ and $\alpha_n, \beta_n \in \mathbb{R}$ are chosen at each step.
The initial time $t_0 \geq 0$ and starting point $u_0 \in \mathbb{R}$ are given.
\end{definition}

In the constructions below, $\alpha_n$ and $\beta_n$ approximate the right- and left-hand derivatives, respectively.
There are two basic principles for choosing $\alpha_n$ and $\beta_n$.
First, one of $\alpha_n$ and $\beta_n$ should incorporate information from the given function $F$.
For example, the choice $\alpha_n = h^{-1}(u_{n+1} - u_n)$ and $\beta_n = h^{-1}(u_n - u_{n-1})$, which does not incorporate any information from $F(t_n, u_n)$, is not recommended.
Second, $\alpha_n$ and $\beta_n$ should not use information from $u_{n+2}$ or $F(t_{n+2}, u_{n+2})$, since such information is not directly tied to the change occurring at time $t_n$.
With these considerations in mind, \cref{tbl:choices_of_alpha_m_and_beta_m} summarizes eight possible choices for $\alpha_n$ and $\beta_n$ in the formula \eqref{mtd:specular_Euler}.
The one-step choices are applied from $n=0$.
The two-step choices are applied from $n=1$ after an external starting value $u_1$ has been supplied.
One may explore other combinations beyond those listed in the table.

\begin{table}[tbhp]
\footnotesize
\caption{Choices of $\alpha_n$ and $\beta_n$ in the formula \eqref{mtd:specular_Euler}.}
\label{tbl:choices_of_alpha_m_and_beta_m}
\centering
\renewcommand{\arraystretch}{1.7} 
\begin{tabular}{cccccc}
    \toprule
    $\alpha_n$ & $\beta_n$ & Formulation & $u_1$ assignment & Step & Name \\ 
    \midrule
    
    \multirow{3}{*}{$F(t_n, u_n)$} 
    & $F(t_{n-1}, u_{n-1})$                    & explicit & externally & two & Type~1 \\
    & $\displaystyle \frac{u_n - u_{n-1}}{h}$  & explicit & externally & two & Type~2 \\
    & $F(t_n, u_n)$                            & explicit & internally & one & Explicit Euler \\ 
    \midrule
    
    \multirow{2}{*}{$\displaystyle\frac{u_{n+1} - u_n}{h}$} 
    & $F(t_{n-1}, u_{n-1})$                 & explicit & externally & two & Type~3 \\
    & $F(t_n, u_n)$                         & explicit & internally & one & Type~4 \\
    \midrule
    
    \multirow{3}{*}{$F(t_{n+1}, u_{n+1})$} 
    & $F(t_n, u_n)$                           & implicit & internally & one & Type~5 \\
    & $\displaystyle\frac{u_{n+1} - u_n}{h}$  & implicit & internally & one & Type~6 \\
    & $F(t_{n+1}, u_{n+1})$                   & implicit & internally & one & Implicit Euler \\ 
    \bottomrule
\end{tabular}
\end{table}

The choices $\alpha_n=\beta_n=F(t_n,u_n)$ and $\alpha_n=\beta_n=F(t_{n+1},u_{n+1})$ yield the explicit and implicit Euler methods, respectively.
This follows directly from the identity $\mathcal{C}(\alpha, \alpha) = \alpha$ in \cite[Lem.~A.2~(a)]{2026a_Jung}.
Thus, the term ``Euler'' is deliberate: the proposed family contains both classical Euler methods as special cases.
We refer to all eight choices as the specular Euler methods.
We denote the methods of Types~1 through~6 by SET1 through SET6, respectively.

SET1 and SET2 are likewise two-step methods.
SET3 is also a two-step method and reduces to
\begin{displaymath}
    u_{n+1}
    =
    u_n+h\,F(t_{n-1},u_{n-1}),
\end{displaymath}
which uses the slope from the preceding step.
Numerical experiments not reported here indicate first-order convergence for all three methods, and we therefore do not consider them further.

The defining relation for SET4 is
\begin{displaymath}
    \frac{u_{n+1}-u_n}{h}
    =
    \mathcal{C} \left( \frac{u_{n+1}-u_n}{h}, F(t_n,u_n) \right).
\end{displaymath}
By \cite[Lem.~A.2~(e)]{2026a_Jung}, $\mathcal{C}(\alpha,\beta)=\alpha$ or $\mathcal{C}(\alpha,\beta)=\beta$ holds if and only if $\alpha=\beta$.
This implies 
\begin{displaymath}
    \frac{u_{n+1}-u_n}{h} = F(t_n,u_n),
\end{displaymath}
and hence SET4 coincides with the explicit Euler method.
Similarly, SET6 coincides with the implicit Euler method.

Thus, among Types~1 through~6, only SET5 remains as a genuinely nonclassical one-step method of interest.
We therefore focus on SET5 and its scaled extension throughout the remainder of the paper.

\subsection{The specular Euler method of Type 5}

In this subsection, we restrict our attention to SET5.

\begin{definition}      \label{def:SE5}
    Consider the initial value problem \eqref{ODE:IVP}.
    The \emph{specular Euler method of Type 5} (SET5) generates a sequence $\left\{ u_n \right\}_{n=0}^{N_h}$ according to the formula
    \begin{equation}    \label{mth:SE5}
        u_{n+1} = u_n + h \, \mathcal{C}(F(t_{n+1}, u_{n+1}), F(t_n, u_n)),
    \end{equation}
    for $n=0,1,\ldots,N_h-1$, where $t_n := t_0 + nh$ for $n=0,1,\ldots,N_h$, $h>0$ is the step size, and the initial time $t_0 \geq 0$ and starting point $u_0 \in \mathbb{R}$ are given.
\end{definition}

The following circular example motivates the main construction: SET5 reproduces the exact solution values at every mesh point.
Consider the problem \eqref{ODE:IVP} with $t_0 = 0$, $u_0 = 1$, and 
\begin{equation}    \label{ode:circle_eq}
    F(t, u) = -\frac{tu}{1-t^2}, \qquad (t,u)\in (-1,1)\times\mathbb{R}.
\end{equation}
For $t \in [0, 1)$, the exact solution is $u(t)=\sqrt{1-t^2}$, and the solution graph is $\gamma(t) = \left( t,\sqrt{1 - t^2} \right)$.

SET5 traces $\gamma$ exactly.
Indeed, let $0 = t_0 < t_1 < \cdots < t_N < 1$ be a uniform mesh.
For each $n=0,1,\ldots,N$, choose $\theta_n\in[0,\frac{\pi}{2})$ such that $t_n=\sin\theta_n$ and $u(t_n)=\cos\theta_n$.
Then, for every $n=0, 1, \ldots, N-1$,
\begin{displaymath}
    \frac{u(t_{n+1})-u(t_n)}{t_{n+1}-t_n}
    =
    -\tan\left(\frac{\theta_{n+1}+\theta_n}{2}\right)
    =
    \mathcal{C}\left( F(t_{n+1},u(t_{n+1})), F(t_n,u(t_n)) \right).
\end{displaymath}
Thus, the exact solution values satisfy the SET5 update at every step.
For this problem, the implicit update has at most one solution since $F(t_{n+1},\cdot)$ is strictly decreasing and $\mathcal{C}$ is strictly increasing in its first variable \cite[Lems.~2.1 and~A.1]{2026a_Jung}.
Hence, on this mesh, SET5 initialized by $u_0 = 1$ generates the unique sequence $\{u_n\}_{n=0}^{N}$ satisfying
\begin{displaymath}
    (t_n,u_n)
    =
    (t_n, u(t_n))
    =
    \gamma(t_n)
    =
    \left( t_n,\sqrt{1-t_n^2} \right), \qquad
    n = 0, 1, \ldots, N.
\end{displaymath}
The same conclusion is recovered geometrically from \cref{thm:SE_exact_constant_curvature}; see also \cref{ex:SE_trace}.

\subsection{The specular ellipse method}

The exact tracing of circular solution graphs by SET5 motivates replacing $\mathcal{C}$ with a scaled angular mean to extend the construction to elliptic solution graphs.

\begin{definition}  \label{def:scaled_angular_mean}
    For each scale $\sigma>0$, define the function $\mathcal{C}_{\sigma}: \mathbb{R}^2 \to \mathbb{R}$ by 
    \begin{equation}    \label{def:C_sigma}
        \mathcal{C}_{\sigma}(\alpha, \beta)
        :=
        \sigma \, \mathcal{C}\left( \frac{\alpha}{\sigma}, \frac{\beta}{\sigma} \right).
    \end{equation}
    When $\sigma = 1$, we simply write $\mathcal{C}(\alpha, \beta)$ for $\mathcal{C}_1(\alpha, \beta)$.
\end{definition}

The properties of the scaled angular mean $\mathcal{C}_{\sigma}$ needed below are established in \cref{apx:C_sigma}.
We now introduce the main numerical method studied in this paper, in which the scale $\sigma$ in $\mathcal{C}_\sigma$ may depend on the step index $n$.

\begin{definition} \label{def:SE}
    Consider the initial value problem \eqref{ODE:IVP}.
    Let $\{\sigma_n\}_{n=0}^{N_h-1}$ be a sequence of positive scales.
    The \emph{specular ellipse} (SE) method with the scale sequence $\{\sigma_n\}_{n=0}^{N_h-1}$ generates a sequence $\{u_n\}_{n=0}^{N_h}$ according to the implicit update
    \begin{equation} \label{mth:SE}
        u_{n+1} = u_n + h\,\mathcal{C}_{\sigma_n} \left( F(t_{n+1},u_{n+1}), F(t_n,u_n) \right),
    \end{equation}
    where $t_n:=t_0+nh$, $h>0$, and $u_0 = u(t_0)$.
    The update is applied for $n = 0, 1, \ldots, N_h-1$.
\end{definition}

If $\sigma_n=1$ for $n=0,1,\ldots,N_h-1$, then $\mathcal{C}_{\sigma_n}=\mathcal{C}$, and SE coincides with SET5.

For a prescribed positive scale sequence, the $n$-th SE update is obtained
by finding $v$ satisfying
\begin{equation}
    \label{eq:SE_implicit_step}
    v
    =
    u_n + h \, \mathcal{C}_{\sigma_n}
    \left(
        F(t_{n+1},v),
        F(t_n,u_n)
    \right)
\end{equation}
and setting $u_{n+1}:=v$.

\Cref{alg:SE} gives a fixed-point implementation of \eqref{eq:SE_implicit_step}.

\begin{algorithm}
  \caption{\textsc{Specular Ellipse Method (SE)}}
  \label{alg:SE}
  \begin{algorithmic}[1]
    \REQUIRE{
      $\Omega\subset\mathbb{R}^2$ open,
      $F:\Omega\to\mathbb{R}$,
      $(t_0,u_0)\in\Omega$,
      $T>t_0$,
      $h\in(0,T-t_0]$,
      $\{\sigma_n\}_{n=0}^{N_h - 1}\subset(0,\infty)$,
      $\zeta>0$,
      $M\in\mathbb{N}$
    }
    \STATE{$N\gets N_h$}
    \FOR{$n=0,1,\ldots,N-1$}
      \STATE{$t_n\gets t_0+nh$}
      \STATE{$t_{n+1}\gets t_n+h$}
      \STATE{$\beta\gets F(t_n,u_n)$}
      \STATE{$u_{\mathrm{temp}}\gets u_n+h\beta$}
      \STATE{$m\gets0$}
      \WHILE{$m<M$}
        \STATE{$\alpha\gets F(t_{n+1},u_{\mathrm{temp}})$}
        \STATE{
          $u_{\mathrm{guess}}
          \gets
          u_n + h \, \mathcal{C}_{\sigma_n}(\alpha,\beta)$
        }
        \IF{$|u_{\mathrm{guess}}-u_{\mathrm{temp}}| < \zeta$}
          \STATE{$u_{n+1}\gets u_{\mathrm{guess}}$}
          \BREAK
        \ENDIF
        \STATE{$u_{\mathrm{temp}}\gets u_{\mathrm{guess}}$}
        \STATE{$m\gets m+1$}
      \ENDWHILE
      \IF{$m=M$}
        \STATE{\textbf{stop with failure}}
      \ENDIF
    \ENDFOR
    \RETURN{$\{(t_n,u_n)\}_{n=0}^{N}$}
  \end{algorithmic}
\end{algorithm}

\section{Geometry and exactness of the specular ellipse method}
\label{sec:geometry_exactness_SE}

This section develops a geometric interpretation of SE and establishes an exactness result for scaled solution graphs of constant signed curvature.

\subsection{Scaled solution-graph and field quantities}

For \eqref{ODE:IVP}, we introduce the $\sigma$-\emph{scaled solution-graph quantities} associated with a solution $u$ and the $\sigma$-\emph{scaled field quantities} determined by $F$ and its partial derivatives.
The $\sigma$-scaled field quantities are defined independently of any particular solution.
When evaluated at $(t,u(t))$, they recover the tangent, unit tangent, and signed curvature of the $\sigma$-scaled solution graph and encode the variation of its curvature.

\begin{definition}[$\sigma$-scaled solution-graph quantities]
    Let $u\in C^2([t_0,T])$ be a solution of the initial value problem
    \eqref{ODE:IVP}, and fix $\sigma>0$.
    Define $\gamma_\sigma:[t_0,T]\to\mathbb{R}^2$ by
    \begin{displaymath}
        \gamma_\sigma(t)
        :=
        \left(t,\frac{u(t)}{\sigma}\right),
        \qquad
        t\in[t_0,T].
    \end{displaymath}
    Its image is called the \emph{$\sigma$-scaled solution graph} of $u$.
    The corresponding unit tangent is denoted by $\tau_\sigma$ and
    defined by
    \begin{displaymath}
        \tau_\sigma(t)
        :=
        \frac{\gamma_\sigma'(t)}
        {\left\|\gamma_\sigma'(t)\right\|_{\mathbb{R}^2}},
        \qquad
        t\in[t_0,T].
    \end{displaymath}
    We denote the signed curvature of $\gamma_\sigma$ by
    $\kappa_\sigma$; that is,
    \begin{equation}    \label{def:normalized_curvature}
        \kappa_\sigma(t)
        :=
        \frac{\det\left( \gamma_\sigma'(t), \gamma_\sigma''(t) \right)}{\left\|\gamma_\sigma'(t)\right\|_{\mathbb{R}^2}^{3}}
        =
        \frac{\sigma^2u''(t)}{\left(\sigma^2+u'(t)^2\right)^{\frac{3}{2}}}
        =
        \frac{u''(t)}{\sigma \left\|\gamma_\sigma'(t)\right\|_{\mathbb{R}^2}^{3}},
        \qquad
        t\in[t_0,T].
    \end{equation}
\end{definition}

If $u\in C^3([t_0,T])$ and $t \in [t_0, T]$, then differentiating \eqref{def:normalized_curvature} gives
\begin{equation}    \label{eq:normalized_curvature_derivative}
    \kappa_{\sigma}'(t)
    =
    \frac{u'''(t)}{\sigma \left\| \gamma_{\sigma}'(t) \right\|_{\mathbb{R}^2}^3} - \frac{3u'(t) u''(t)^2}{\sigma^3 \left\| \gamma_{\sigma}'(t) \right\|_{\mathbb{R}^2}^5},
    \qquad
    t\in[t_0,T].
\end{equation}

The solution-graph quantities above depend on the exact solution, which is generally unknown.
We therefore introduce $\sigma$-scaled field quantities determined for each fixed $\sigma > 0$ by the right-hand side $F$ and its partial derivatives.
The derivatives appearing in these quantities are defined using the following differential operator associated with $F$.

Let $\Omega\subset\mathbb{R}^2$ be open, and let $F\in C^2(\Omega)$.
For $G\in C^1(\Omega)$, define
\begin{displaymath}
    \mathcal{L}_FG(t,x)
    :=
    G_t(t,x)+F(t,x)G_x(t,x),
    \qquad
    (t,x)\in\Omega.
\end{displaymath}
If $G\in C^2(\Omega)$, define
\begin{displaymath}
    \mathcal{L}_F^2G(t,x)
    :=
    \mathcal{L}_F\bigl(\mathcal{L}_FG\bigr)(t,x),
    \qquad
    (t,x)\in\Omega.
\end{displaymath}

Together with $F$, these derivatives determine the $\sigma$-scaled field counterparts of the $\sigma$-scaled solution-graph quantities.

\begin{definition}[$\sigma$-scaled field quantities]
    Let $\Omega\subset\mathbb{R}^2$ be open, and let $F\in C^2(\Omega)$ be the right-hand side of \eqref{ODE:IVP}.
    Fix $\sigma>0$.
    Since $F$ is fixed, we suppress its dependence in the notation introduced below.
    Define the \emph{$\sigma$-scaled graph-tangent field} by
    \begin{displaymath}
        \mathcal{V}_\sigma(t,x)
        :=
        \left(1,\frac{F(t,x)}{\sigma}\right),
        \qquad
        (t,x)\in\Omega.
    \end{displaymath}
    The field $\mathcal{V}_\sigma$ never vanishes, since
    \begin{equation}    \label{eq:norm_V}
        \left\|\mathcal{V}_\sigma(t,x)\right\|_{\mathbb{R}^2}^2
        =
        1 + \frac{F(t,x)^2}{\sigma^2}.
    \end{equation}
    Define the corresponding \emph{unit graph-tangent field} by
    \begin{displaymath}
        \mathcal{T}_\sigma(t,x)
        :=
        \frac{\mathcal{V}_\sigma(t,x)}{\left\|\mathcal{V}_\sigma(t,x)\right\|_{\mathbb{R}^2}}
        =
        \frac{\left(\sigma,F(t,x)\right)}{\sqrt{\sigma^2+F(t,x)^2}}.
    \end{displaymath}
    Define the \emph{signed-curvature field associated with the
    $\sigma$-scaled solution graphs} by
    \begin{equation}
        \label{def:normalized_curvature_field}
        \mathcal{K}_\sigma(t,x)
        :=
        \frac{\mathcal{L}_FF(t,x)}{\sigma \left\|\mathcal{V}_\sigma(t,x)\right\|_{\mathbb{R}^2}^3}.
    \end{equation}
    Define the corresponding \emph{curvature-variation defect field} by
    \begin{equation}
        \label{def:normalized_curvature_defect}
        \mathcal{D}_\sigma(t,x)
        :=
        \mathcal{L}_F^2F(t,x) - \frac{3F(t,x)\bigl(\mathcal{L}_FF(t,x)\bigr)^2}{\sigma^2 \left\|\mathcal{V}_\sigma(t,x)\right\|_{\mathbb{R}^2}^2}.
    \end{equation}
    When no confusion can arise, we omit the arguments $(t,x)$ from $\mathcal{V}_\sigma$, $\mathcal{T}_\sigma$, $\mathcal{K}_\sigma$, and $\mathcal{D}_\sigma$.
\end{definition}

The scalar $\sigma$-scaled field quantities $\mathcal{K}_\sigma$ and
$\mathcal{D}_\sigma$ can be evaluated directly from the partial
derivatives of $F$, since
\begin{displaymath}
    \mathcal{L}_FF
    =
    F_t+FF_x
\end{displaymath}
and
\begin{displaymath}
    \mathcal{L}_F^2F
    =
    F_{tt}
    +2FF_{tx}
    +F^2F_{xx}
    +F_tF_x
    +FF_x^2.
\end{displaymath}
Here, all quantities on the right-hand sides are evaluated at $(t,x)$.
A direct calculation also gives
\begin{equation}
    \label{eq:transport_curvature_defect}
    \bigl(\mathcal{L}_F\mathcal{K}_\sigma\bigr)(t,x)
    =
    \frac{\mathcal{D}_\sigma(t,x) }{\sigma \left\|\mathcal{V}_\sigma(t,x)\right\|_{\mathbb{R}^2}^3},
    \qquad
    (t,x)\in\Omega.
\end{equation}

For convenience, the two families of quantities and their relationships are summarized in \cref{tbl:solution_graph_and_field_quantities}.

\begin{table}[tbhp]
    \centering
    \renewcommand{\arraystretch}{1.2}
    \caption{$\sigma$-scaled solution-graph quantities and the corresponding $\sigma$-scaled field quantities associated with \eqref{ODE:IVP}, induced by a solution $u$ and the right-hand side $F$, respectively.}
    \label{tbl:solution_graph_and_field_quantities}
    \begin{tabular}{@{}lcc@{}}
        \toprule
        \textbf{$\sigma$-scaled quantity}
        &
        \textbf{Solution graph}
        &
        \textbf{Field}
        \\
        \midrule

        Induced by
        &
        $u$
        &
        $F$
        \\

        Graph parametrization
        &
        $\gamma_\sigma$
        &
        --- 
        \\

        Graph tangent
        &
        $\displaystyle \gamma_\sigma'(t)$
        &
        $\displaystyle \mathcal{V}_\sigma(t,x)$
        \\

        Unit graph tangent
        &
        $\tau_\sigma(t)$
        &
        $\mathcal{T}_\sigma(t,x)$
        \\

        Signed curvature
        &
        $\displaystyle \kappa_\sigma(t)$
        &
        $\displaystyle \mathcal{K}_\sigma(t,x)$
        \\

        Curvature variation
        &
        $\displaystyle \kappa_\sigma'(t)$
        &
        $\displaystyle
        \mathcal{L}_F\mathcal{K}_\sigma(t,x)$
        \\

        Curvature defect
        &
        ---
        &
        $\displaystyle \mathcal{D}_\sigma(t,x)$
        \\

        \bottomrule
    \end{tabular}
\end{table}

We next examine the relationship between these two families of $\sigma$-scaled quantities.
Assume now that $u\in C^3([t_0,T])$ is a solution of \eqref{ODE:IVP} and that $(t, u(t)) \in \Omega$ for every $t \in [t_0, T]$.
For every $G\in C^1(\Omega)$, the chain rule gives
\begin{equation}    \label{eq:chain_rule}
    \frac{d}{dt}G(t, u(t))
    =
    \mathcal{L}_F G (t, u(t)),
    \qquad
    t\in(t_0,T).
\end{equation}
In particular, we have
\begin{displaymath} 
    u''(t)
    =
    \mathcal{L}_F F (t, u(t))
    \qquad\text{and}\qquad
    u'''(t)
    =
    \mathcal{L}_F^2 F (t, u(t)),
    \qquad
    t\in(t_0,T).
\end{displaymath}
Combining these equalities with \eqref{def:normalized_curvature_defect} and \eqref{eq:norm_V} gives
\begin{equation}    \label{eq:D_on_exact_sol}
    \mathcal{D}_\sigma(t, u(t))
    =
    u'''(t) - \frac{3 u'(t) u''(t)^2}{\sigma^2 + u'(t)^2},
    \qquad
    t\in(t_0,T).
\end{equation}
The graph-tangent and unit graph-tangent fields restrict to the corresponding solution-graph quantities:
\begin{equation}
    \label{eq:normalized_tangent_along_solution}
    \mathcal{V}_{\sigma} \bigl(t, u(t)\bigr)
    =
    \gamma_{\sigma}'(t)
    \qquad\text{and}\qquad
    \mathcal{T}_{\sigma} \bigl(t, u(t)\bigr)
    =
    \tau_{\sigma}(t),
    \qquad
    t\in(t_0,T).
\end{equation}
Likewise, the signed-curvature field restricts to the signed curvature of the $\sigma$-scaled solution graph:
\begin{equation}
    \label{eq:curvature_field_along_solution}
    \mathcal{K}_{\sigma} \bigl(t,u(t)\bigr)
    =
    \kappa_\sigma(t).
\end{equation}
Applying \eqref{eq:chain_rule} to $\mathcal{K}_\sigma$ and using \eqref{eq:transport_curvature_defect}, we obtain
\begin{equation}
    \label{eq:curvature_defect_relation-1}
    \kappa_\sigma'(t)
    =
    \bigl(\mathcal{L}_F\mathcal{K}_\sigma\bigr)
        \bigl(t,u(t)\bigr)
    =
    \frac{
        \mathcal{D}_\sigma\bigl(t,u(t)\bigr)
    }{
        \sigma
        \left\|
            \mathcal{V}_\sigma\bigl(t,u(t)\bigr)
        \right\|_{\mathbb{R}^2}^{3}
    }.
\end{equation}
Equivalently,
\begin{equation}   \label{eq:curvature_defect_relation-2}
    \mathcal{D}_\sigma\bigl(t,u(t)\bigr)
    =
    \sigma \left\| \gamma_\sigma'(t) \right\|_{\mathbb{R}^2}^{3} \kappa_\sigma'(t).
\end{equation}
Thus, $\mathcal{D}_\sigma$ encodes the variation of the signed curvature
along the $\sigma$-scaled solution graph.
Since $\sigma \left\| \mathcal{V}_\sigma\bigl(t,u(t)\bigr) \right\|_{\mathbb{R}^2}^{3} > 0$, $\mathcal{D}_\sigma\bigl(t,u(t)\bigr)$ and
$\kappa_\sigma'(t)$ have the same zeros and the same sign.

\subsection{Exactness}

We next introduce the discrete counterpart of the $\sigma$-scaled solution graph.

\begin{definition}
    \label{def:scaled_numerical_solution_graph}
    Let $N\in\mathbb{N}$ and $h>0$, and set $t_n:=t_0+nh$, $n=0,1,\ldots,N$, where $t_N\leq T$.
    Let $\{u_n\}_{n=0}^{N}\subset\mathbb{R}$ be a candidate numerical sequence such that $(t_n,u_n)\in\Omega$ for all $n=0,1,\ldots,N$.
    Fix $\sigma>0$ and define the \emph{$\sigma$-scaled numerical solution-graph points} by
    \begin{displaymath}
        \Gamma_\sigma(n)
        :=
        \left(t_n,\frac{u_n}{\sigma}\right),
        \qquad
        n=0,1,\ldots,N.
    \end{displaymath}
    The polygonal curve joining the points $\{\Gamma_\sigma(n)\}_{n=0}^{N}$ is called the \emph{$\sigma$-scaled numerical solution graph}.
\end{definition}

If $u_n=u(t_n)$ for a solution $u$ of \eqref{ODE:IVP}, then $\Gamma_\sigma(n)=\gamma_\sigma(t_n)$.

For a sequence of positive scales $\{\sigma_n\}$, the $n$-th SE step is represented in the $\sigma_n$-scaled plane by the chord joining $\Gamma_{\sigma_n}(n)$ and $\Gamma_{\sigma_n}(n+1)$.
The following lemma characterizes the SE relation by comparing the direction of this chord with that of the sum of the corresponding unit field tangents.

\begin{lemma}
    \label{lem:SE_chord_tangent}
    Fix $n\in\mathbb{N} \cup \{0\}$ and $\sigma > 0$ and suppose that $(t_n,u_n),(t_{n+1},u_{n+1})\in\Omega$.
    Then the following statements are equivalent.
    \begin{enumerate}[label=\upshape(\roman*)]
        \item \label{lem:SE_chord_tangent-1} The pair $u_n,u_{n+1}$ satisfies the SE relation \eqref{mth:SE} with $\sigma_n=\sigma$.
        \item \label{lem:SE_chord_tangent-2} The following unit directions coincide:
        \begin{equation}
            \label{eq:SE_chord_tangent}
            \frac{\Gamma_{\sigma}(n+1)-\Gamma_{\sigma}(n)}{\left\|\Gamma_{\sigma}(n+1)-\Gamma_{\sigma}(n)\right\|_{\mathbb{R}^2}}
            =
            \frac{\mathcal{T}_{\sigma}(t_{n+1},u_{n+1}) + \mathcal{T}_{\sigma}(t_n,u_n)}{\left\|\mathcal{T}_{\sigma}(t_{n+1},u_{n+1}) + \mathcal{T}_{\sigma}(t_n,u_n) \right\|_{\mathbb{R}^2}}.
        \end{equation}
        \item \label{lem:SE_chord_tangent-3} There exists $\lambda \equiv \lambda(n, \sigma, u_n, u_{n+1}) > 0$ such that
        \begin{displaymath}
            \Gamma_{\sigma}(n+1)-\Gamma_{\sigma}(n)
            =
            \lambda \left( \mathcal{T}_{\sigma}(t_{n+1},u_{n+1}) + \mathcal{T}_{\sigma}(t_n,u_n) \right).
        \end{displaymath}
    \end{enumerate}
\end{lemma}
    
\begin{proof}
    Note that \cref{lem:SE_chord_tangent-1} holds if and only if
    \begin{equation}    \label{eq:SE_chord_tangent-1}
        \frac{u_{n+1}-u_n}{h\sigma}
        =
        \frac{\mathcal{C}_\sigma \left( \alpha_n, \beta_n \right)}{\sigma},
    \end{equation}
    where $\alpha_n := F(t_{n+1},u_{n+1})$ and $\beta_n := F(t_n,u_n)$.

    We claim that \cref{lem:SE_chord_tangent-2} holds if and only if \eqref{eq:SE_chord_tangent-1} holds.
    On the one hand, by the definition of $\Gamma_\sigma$,
    \begin{displaymath}
        \frac{\Gamma_\sigma(n+1)-\Gamma_\sigma(n)}{\left\| \Gamma_\sigma(n+1)-\Gamma_\sigma(n) \right\|_{\mathbb{R}^2}}
        =
        \left( 1 + \left( \frac{u_{n+1} - u_n}{h \sigma} \right)^2 \right)^{-\frac{1}{2}} \left( 1, \frac{u_{n+1} - u_n}{h \sigma} \right).
    \end{displaymath}
    On the other hand, by the definitions of $\mathcal{T}_\sigma$ and $\mathcal{C}_\sigma$, 
    \begin{displaymath}
        \mathcal{T}_\sigma(t_{n+1},u_{n+1}) + \mathcal{T}_\sigma(t_n,u_n)
        =
        \sigma
        \left( \frac{1}{\sqrt{\sigma^2+\alpha_n^2}} + \frac{1}{\sqrt{\sigma^2+\beta_n^2}} \right)
        \left( 1, \frac{\mathcal{C}_\sigma(\alpha_n,\beta_n)}{\sigma} \right),
    \end{displaymath}
    and hence 
    \begin{displaymath}
        \frac{ \mathcal{T}_\sigma(t_{n+1},u_{n+1}) + \mathcal{T}_\sigma(t_n,u_n) }{ \left\| \mathcal{T}_\sigma(t_{n+1},u_{n+1}) + \mathcal{T}_\sigma(t_n,u_n) \right\|_{\mathbb{R}^2} }
        =
        \left( 1 + \left( \dfrac{\mathcal{C}_\sigma \left( \alpha_n, \beta_n \right)}{\sigma} \right)^2 \right)^{-\frac{1}{2}} \left( 1, \dfrac{\mathcal{C}_\sigma \left( \alpha_n, \beta_n \right)}{\sigma}  \right) .
    \end{displaymath}
    Consequently, the claim follows.
    Therefore, \cref{lem:SE_chord_tangent-1} and \cref{lem:SE_chord_tangent-2} are equivalent.

    Finally, define
    \begin{displaymath}
        \lambda
        :=
        \frac{\left\|\Gamma_{\sigma}(n+1)-\Gamma_{\sigma}(n)\right\|_{\mathbb{R}^2}}{\left\|\mathcal{T}_{\sigma}(t_{n+1},u_{n+1}) + \mathcal{T}_{\sigma}(t_n,u_n) \right\|_{\mathbb{R}^2}}
        >
        0.
    \end{displaymath}
    Then \cref{lem:SE_chord_tangent-2} and \cref{lem:SE_chord_tangent-3} are equivalent.
\end{proof}

The preceding geometric characterization allows us to identify a class of problems for which SE reproduces the exact solution at every mesh point.

\begin{theorem}[Exactness]
    \label{thm:SE_exact_constant_curvature}
    Let $u\in C^3([t_0,T])$ be a solution of \eqref{ODE:IVP}.
    Suppose that there exists $\sigma>0$ such that
    \begin{equation}
        \label{eq:exact_assumption}
        \kappa_\sigma'(t)=0
    \end{equation}
    for every $t\in(t_0,T)$.
    Let $h>0$.
    Suppose further that \eqref{eq:SE_implicit_step} with $\sigma_n=\sigma$ has a unique admissible solution at each step, where admissibility means $(t_{n+1},v)\in\Omega$.
    Then SE with $u_0 = u(t_0)$ and $\sigma_n = \sigma$ satisfies
    \begin{displaymath}
        u_n=u(t_n)
    \end{displaymath}
    for every $n = 0, 1, \ldots, N_h$.
\end{theorem}

\begin{proof}
    Fix $s, t\in[t_0,T]$ with $s < t$.
    We claim that
    \begin{equation} \label{eq:SE_exact_constant_curvature-0}
        \gamma_{\sigma}(t)-\gamma_{\sigma}(s)
        \qquad\text{and}\qquad
        \tau_{\sigma}(t)+\tau_{\sigma}(s)
    \end{equation}
    have the same direction.
    Since $\kappa_\sigma'=0$ on $(t_0, T)$, the function $\kappa_\sigma$ is constant on $(t_0, T)$.
    By continuity, there exists $c\in\mathbb{R}$ such that $\kappa_\sigma(t)=c$ for all $t \in [t_0, T]$.

    If $c=0$, then \eqref{def:normalized_curvature} gives $u''=0$.
    Hence, $\gamma_\sigma'$ is constant, and
    \begin{displaymath}
        \gamma_{\sigma}(t) - \gamma_{\sigma}(s)
        =
        (t-s) \gamma_{\sigma}'(s) 
        \qquad\text{and}\qquad
        \tau_{\sigma}(t)+\tau_\sigma(s)
        =
        \frac{2}{\left\lVert\gamma_{\sigma}'(s)\right\rVert_{\mathbb{R}^2}} \gamma_{\sigma}'(s).
    \end{displaymath}
    Since $t-s>0$, both vectors are positive scalar multiples of $\gamma_\sigma'(s)$ and therefore have the same direction.

    Suppose that $c\neq 0$.
    Define the function $\nu_\sigma : [t_0, T] \to \mathbb{R}^2$ by
    \begin{displaymath}
        \nu_\sigma(z)
        :=
        \frac{(-u'(z),\sigma)}
        {\sqrt{\sigma^2+u'(z)^2}}.
    \end{displaymath}
    A direct differentiation gives
    \begin{displaymath}
        \frac{d}{dz} \left( \gamma_\sigma(z) + \frac{1}{c}\nu_\sigma(z) \right)
        =
        \gamma_\sigma'(z) - \frac{1}{c} \, \kappa_\sigma(z) \, \gamma_\sigma'(z)
        =
        0
    \end{displaymath}
    for every $z \in (t_0, T)$.
    Hence, by continuity, there exists a point $(p, q) \in \mathbb{R}^2$ such that $\gamma_\sigma(z) + \frac{1}{c}\nu_\sigma(z) = (p, q)$ for every $z \in [t_0, T]$.
    Then
    \begin{equation}   \label{eq:SE_exact_constant_curvature-1}
        \nu_{\sigma}(t) + \nu_{\sigma}(s) 
        =
        -c \left( \gamma_\sigma(t) + \gamma_\sigma(s) - 2(p, q) \right)
        \neq
        (0, 0),
    \end{equation}
    where the last inequality follows since the second component of the left-hand side is strictly positive.
    Note that 
    \begin{equation} \label{eq:SE_exact_constant_curvature-2}
        \left\lVert \gamma_\sigma(z)-(p,q) \right\rVert_{\mathbb{R}^2}
        =
        \frac{1}{|c|} \left\lVert\nu_\sigma(z)\right\rVert_{\mathbb{R}^2}
        =
        \frac{1}{|c|}
    \end{equation}
    for all $z \in [t_0, T]$.
    Hence, $\gamma_\sigma$ lies on the circle with center $(p, q)$ and radius $|c|^{-1}$.

    On the one hand, \eqref{eq:SE_exact_constant_curvature-1} and \eqref{eq:SE_exact_constant_curvature-2} yield
    \begin{align*}
        (\nu_{\sigma}(t) + \nu_{\sigma}(s)) \innerprd (\gamma_{\sigma}(t) - \gamma_{\sigma}(s))
        &= -c \left\| \gamma_\sigma(t) - (p, q) \right\|_{\mathbb{R}^2}^2 + c \left\| \gamma_\sigma(s) - (p, q) \right\|_{\mathbb{R}^2}^2    \\
        &= 0.
    \end{align*}
    On the other hand, since $\tau_\sigma(z) \innerprd \nu_\sigma(z) = 0$ for every $z \in [t_0,T]$, 
    \begin{align*}
        (\nu_{\sigma}(t) + \nu_{\sigma}(s)) \innerprd (\tau_{\sigma}(t) + \tau_{\sigma}(s))
        &= \nu_{\sigma}(s) \innerprd \tau_{\sigma}(t) + \nu_{\sigma}(t) \innerprd \tau_{\sigma}(s) \\
        &= \frac{ \sigma\bigl(u'(t)-u'(s)\bigr) + \sigma\bigl(u'(s)-u'(t)\bigr) }{ \sqrt{\sigma^2+u'(t)^2} \sqrt{\sigma^2+u'(s)^2}}   \\
        &= 0.
    \end{align*}
    Therefore, both vectors in \eqref{eq:SE_exact_constant_curvature-0} are orthogonal to the same vector $\nu_{\sigma}(t) + \nu_{\sigma}(s)$.
    By \eqref{eq:SE_exact_constant_curvature-1}, the common orthogonal vector $\nu_{\sigma}(t)+\nu_{\sigma}(s)$ is nonzero.
    Its orthogonal complement in $\mathbb{R}^2$ is one-dimensional.
    Hence, the two vectors in \eqref{eq:SE_exact_constant_curvature-0} are parallel.
    Since both have strictly positive first components, the two vectors have the same direction.
    This proves the claim.

    Fix $n\in\{ 0,1,\ldots, N_h-1 \}$ and apply \cref{lem:SE_chord_tangent} to the pair $u(t_n), u(t_{n+1})$.
    For this pair, the definitions of $\Gamma_\sigma$ and $\gamma_\sigma$ give
    \begin{displaymath}
        \Gamma_{\sigma}(n+1)-\Gamma_{\sigma}(n)
        =
        \gamma_{\sigma}(t_{n+1})-\gamma_{\sigma}(t_n).
    \end{displaymath}
    Moreover, \eqref{eq:normalized_tangent_along_solution} gives
    \begin{displaymath}
        \mathcal{T}_{\sigma}(t_{n+1},u(t_{n+1})) + \mathcal{T}_{\sigma}(t_n,u(t_n))
        =
        \tau_{\sigma}(t_{n+1}) + \tau_{\sigma}(t_n).
    \end{displaymath}
    Taking $s = t_n$ and $t = t_{n+1}$ in the claim, we conclude that \cref{lem:SE_chord_tangent} \ref{lem:SE_chord_tangent-2} holds.
    The equivalent condition \cref{lem:SE_chord_tangent} \ref{lem:SE_chord_tangent-1} therefore gives
    \begin{equation}    \label{eq:SE_exact_constant_curvature-3}
        u(t_{n+1})
        =
        u(t_n) + h \, \mathcal{C}_{\sigma} \left( F(t_{n+1},u(t_{n+1})), F(t_n,u(t_n)) \right).
    \end{equation}
    Since $n$ was arbitrary, this identity holds for every $n = 0, 1, \ldots, N_h-1$.

    It remains to show that the numerical values coincide with the exact values at the mesh points.
    The initial condition gives $u_0 = u(t_0)$.
    Assume inductively that $u_n = u(t_n)$.
    By the induction hypothesis and \eqref{eq:SE_exact_constant_curvature-3}, $u(t_{n+1})$ is an admissible solution of \eqref{eq:SE_implicit_step} with $\sigma_n=\sigma$.
    By the assumed uniqueness, $u_{n+1}=u(t_{n+1})$.
    The conclusion follows by induction.
\end{proof}

\begin{example} \label{ex:SE_trace}
    Fix $a,b>0$, $p\geq0$, and $q\in\mathbb{R}$.
    Fix $t_0,T\in(p-a,p+a)$ with $0\leq t_0<T$.
    Consider the problem \eqref{ODE:IVP} with
    \begin{equation} \label{ODE:ellipse}
        F(t, u)
        =
        -\frac{b^2(t-p)}{a^2(u-q)},
        \qquad
        (t, u)\in (p-a,p+a)\times(q,\infty).
    \end{equation}
    The exact solution on the upper elliptic branch is
    \begin{equation} \label{eq:ellipse_sol}
        u(t)
        =
        q+b \left( 1 - \left(\frac{t-p}{a}\right)^2 \right)^{\frac{1}{2}}, \qquad t\in[t_0, T],
    \end{equation}
    where the initial value is chosen as $u_0 := u(t_0)$.
    The graph of this solution lies on the ellipse
    \begin{displaymath}
        \left(\frac{t-p}{a}\right)^2
        +
        \left(\frac{u-q}{b}\right)^2
        =1.
    \end{displaymath}

    We now determine the scale from the constant-curvature condition.
    Let $\sigma>0$ be arbitrary.
    A direct calculation from \eqref{def:normalized_curvature} gives
    \begin{displaymath}
        \kappa_\sigma'(t)
        =
        3ab\sigma^2
        \left(
            \frac{b^2}{a^2}-\sigma^2
        \right)
        (t-p)
        \left(
            \sigma^2\bigl(a^2-(t-p)^2\bigr)
            +
            \frac{b^2}{a^2}(t-p)^2
        \right)^{-\frac{5}{2}}.
    \end{displaymath}
    Since $t_0 < T$ and $a,b,\sigma > 0$, it follows that $\kappa_\sigma' \equiv 0$ on $(t_0, T)$ if and only if $\sigma=\frac{b}{a}$.
    Thus, if \eqref{eq:SE_implicit_step} with $\sigma_n=\frac{b}{a}$ has a unique admissible solution at each step, \cref{thm:SE_exact_constant_curvature} shows that the numerical values coincide with the exact solution values.
    A numerical illustration of this exact-tracing property on an upper-branch segment is given in \cref{fig:ellipse_exactness}.

    When $p=q=0$ and $a=b=1$, the elliptic solution reduces to the unit-circle solution of \eqref{ode:circle_eq}, and the exact-tracing scale is $\sigma=1$.
\end{example}

\section{Convergence analysis for the specular ellipse method}
\label{sec:convergence}

In the preceding section, exact tracing was obtained under \eqref{eq:exact_assumption} when \eqref{eq:SE_implicit_step} with $\sigma_n=\sigma$ has a unique admissible solution at each step.
Since this condition need not hold for a general ODE, exact tracing is an exceptional property.
In this section, we instead regard SE as a numerical method for general smooth ODEs and study its consistency and convergence.

\subsection{Consistency} 

We begin by defining the local truncation error in the standard way.

\begin{definition}
    Let $u:[t_0,T]\to\mathbb{R}$ be a solution of \eqref{ODE:IVP}.
    For each $h>0$, let $\{\sigma_n\}_{n=0}^{N_h-1}$ be a sequence of positive scales.
    Define the \emph{local truncation error} of SE with the scale sequence $\{\sigma_n\}_{n=0}^{N_h-1}$ by
    \begin{equation}    \label{def:lte}
        \ell_{n+1}
        :=
        \frac{u(t_{n+1})-u(t_n)}{h}
        -
        \mathcal{C}_{\sigma_n} \left( F(t_{n+1}, u(t_{n+1})), F(t_n, u(t_n)) \right)
    \end{equation}
    for $n=0,1,\ldots,N_h-1$.
\end{definition}

The following theorem gives the local truncation error expansions used in the consistency analysis and relates their leading terms to the signed-curvature derivatives of the corresponding $\sigma_n$-scaled solution graphs.

\begin{theorem}
    \label{thm:lte}
    Suppose that hypothesis~\ref{H2} holds and that $F\in C^2(\Omega)$.
    \begin{enumerate}[label=\upshape(\alph*)]
        \item 
        \label{thm:lte-1}
        If $u \in C^3([t_0,T])$, then
        \begin{equation}
            \label{eq:lte-2nd}
            \ell_{n+1}
            =
            - \frac{h^2}{12} \sigma_n \left\lVert\gamma_{\sigma_n}'(t_n)\right\rVert^3 \kappa_{\sigma_n}'(t_n)
            + R_{n+1}(h),
        \end{equation}
        where there exists a function $\omega:(0, H] \to [0, \infty)$ such that $\omega(h)\to0$ as $h\searrow0$ and
        \begin{displaymath}
            |R_{n+1}(h)|
            \leq
            \omega(h)h^2
        \end{displaymath}
        whenever $0<h\leq H$ and $t_{n+1}\leq T$.

        \item
        \label{thm:lte-2}
        If $u\in C^4([t_0,T])$, then
        \begin{equation}
            \label{eq:lte-3rd}
            \ell_{n+1}
            =
            -\frac{h^2}{12} \sigma_n \left\lVert\gamma_{\sigma_n}'(t_n)\right\rVert^3 \kappa_{\sigma_n}'(t_n)
            + R_{n+1}(h),
        \end{equation}
        where there exists a constant $C>0$, independent of $h$ and $n$, such that
        \begin{displaymath}
            |R_{n+1}(h)|
            \leq
            Ch^3
        \end{displaymath}
        whenever $h>0$ is sufficiently small and $t_{n+1}\leq T$.

        \item
        \label{thm:lte-3}
        If $u\in C^5([t_0,T])$, then 
        \begin{displaymath} 
            \ell_{n+1}
            =
            -\frac{h^2}{24} \sigma_n \left[  \left\lVert\gamma_{\sigma_n}'(t_n)\right\rVert^3 \kappa_{\sigma_n}'(t_n) + \left\lVert\gamma_{\sigma_n}'(t_{n+1})\right\rVert^3 \kappa_{\sigma_n}'(t_{n+1}) \right]
            + R_{n+1}(h),
        \end{displaymath}
        where there exists a constant $C>0$, independent of $h$ and $n$, such that
        \begin{displaymath}
            |R_{n+1}(h)|
            \leq
            Ch^4
        \end{displaymath}
        whenever $h>0$ is sufficiently small and $t_{n+1}\leq T$.
    \end{enumerate}
\end{theorem}

\begin{proof}
    We apply \cref{lem:uniform_diagonal_C_sigma}.
    For each $n$, define 
    \begin{displaymath}
        \delta_n
        :=
        u'(t_{n+1}) - u'(t_n).
    \end{displaymath}
    For $K=[-M_1,M_1]$, the radius defined in
    \eqref{def:uniform_diagonal_radius} is
    \begin{displaymath}
        r
        =
        \left( M_1^2+\frac{\eta}{2} \right)^{\frac{1}{2}} - M_1
        >
        0.
    \end{displaymath}
    Since 
    \begin{displaymath}
        \left|\delta_n\right| 
        \leq
        \int_{t_n}^{t_{n+1}}\left|u^{\prime \prime}(s)\right| \, d s 
        \leq 
        M_2\left(t_{n+1}-t_n\right)
        =
        M_2 h,
    \end{displaymath}
    we may, if necessary, decrease $H$ so that $|\delta_n|\leq r$ whenever $0<h\leq H$ and $t_{n+1}\leq T$.
    Moreover,
    \begin{displaymath}
        |u'(t_n)|\leq M_1
        \qquad\text{and}\qquad
        \sigma_n^2+u'(t_n)^2\geq\eta.
    \end{displaymath}
    Therefore, \cref{lem:uniform_diagonal_C_sigma} applies with $K=[-M_1, M_1]$, $\sigma=\sigma_n$, $\alpha=u'(t_n)$, and $\delta=\delta_n$.
    Hence, there exists a constant $\lambda\equiv\lambda(M_1,\eta)>0$ such that
    \begin{align*}
        & \, \mathcal{C}_{\sigma_n} \left( u'(t_{n+1}), u'(t_n) \right) \\
        =& \, \mathcal{C}_{\sigma_n} \left( u'(t_n), u'(t_n)+\delta_n \right)\\
        =&
        \, u'(t_n) 
        + \frac{1}{2}\delta_n 
        - \frac{u'(t_n)} {4\left(\sigma_n^2+u'(t_n)^2\right)} \delta_n^2
        + \frac{u'(t_n)^2-\sigma_n^2} {8\left(\sigma_n^2+u'(t_n)^2\right)^2} \delta_n^3
        + R_{\sigma_n}\bigl(u'(t_n),\delta_n\bigr),
    \end{align*}
    where
    \begin{displaymath}
        \left|
            R_{\sigma_n}\bigl(u'(t_n),\delta_n\bigr)
        \right|
        \leq
        \lambda|\delta_n|^4
    \end{displaymath}
    uniformly in $n$ and $h$.
    Substituting this equality into \eqref{def:lte} gives
    \begin{equation}    \label{eq:lte_expanded}
        \begin{aligned}
            \ell_{n+1}
            =& \, \frac{u(t_{n+1})-u(t_n)}{h} - u'(t_n) - \frac{1}{2}\delta_n + \frac{u'(t_n)} {4\left(\sigma_n^2+u'(t_n)^2\right)} \delta_n^2 \\
            & - \frac{u'(t_n)^2-\sigma_n^2} {8\left(\sigma_n^2+u'(t_n)^2\right)^2} \delta_n^3 - R_{\sigma_n}\bigl(u'(t_n),\delta_n\bigr).
        \end{aligned}
    \end{equation}
    
    Suppose first that $u\in C^3([t_0,T])$.
    Taylor's theorem gives, uniformly in $n$,
    \begin{align*}
        \frac{u(t_{n+1})-u(t_n)}{h}
        &=
        u'(t_n)
        +
        \frac{h}{2}u''(t_n)
        +
        \frac{h^2}{6}u'''(t_n)
        +
        o(h^2),\\
        \delta_n
        &=
        hu''(t_n)
        +
        \frac{h^2}{2}u'''(t_n)
        +
        o(h^2).
    \end{align*}
    In particular,
    \begin{displaymath}
        \delta_n^2
        =
        h^2u''(t_n)^2+o(h^2)
        \qquad\text{and}\qquad
        \delta_n^3
        =
        O(h^3)
    \end{displaymath}
    uniformly in $n$.
    Moreover, hypothesis~\ref{H2} gives
    \begin{displaymath}
        \left| \frac{u'(t_n)}{4\left(\sigma_n^2+u'(t_n)^2\right)}\right|
        \leq
        \frac{M_1}{4\eta} 
        \qquad\text{and}\qquad
        \left| \frac{u'(t_n)^2-\sigma_n^2}{8\left(\sigma_n^2+u'(t_n)^2\right)^2} \right|
        \leq
        \frac{1}{8\eta}.
    \end{displaymath}
    Also,
    \begin{displaymath}
        \left|
            R_{\sigma_n}\bigl(u'(t_n),\delta_n\bigr)
        \right|
        \leq
        \lambda M_2^4h^4.
    \end{displaymath}
    Substituting these expansions and estimates into \eqref{eq:lte_expanded}, we obtain
    \begin{displaymath}
        \ell_{n+1}
        =
        - \frac{h^2}{12}u'''(t_n)
        + \frac{h^2u'(t_n)u''(t_n)^2}{4\left(\sigma_n^2+u'(t_n)^2\right)} + o(h^2)
        =
        -\frac{h^2}{12} \mathcal{D}_{\sigma_n}\bigl(t_n,u(t_n)\bigr) + o(h^2)
    \end{displaymath}
    uniformly in $n$.
    Together with \eqref{eq:D_on_exact_sol} and \eqref{eq:curvature_defect_relation-2}, the preceding expansion proves \cref{thm:lte-1}.

    If $u\in C^4([t_0,T])$, then the $o(h^2)$ remainders in the preceding expansions of the difference quotient and $\delta_n$ can both be replaced by $O(h^3)$, uniformly in $n$.
    Hence,
    \begin{displaymath}
        \delta_n^2
        =
        h^2u''(t_n)^2+O(h^3),
        \qquad
        \delta_n^3
        =
        O(h^3),
        \qquad
        R_{\sigma_n}\bigl(u'(t_n),\delta_n\bigr)
        =
        O(h^4)
    \end{displaymath}
    uniformly in $n$.
    Substitution into \eqref{eq:lte_expanded}, together with \eqref{eq:D_on_exact_sol}, gives
    \begin{displaymath}
        \ell_{n+1}
        =
        -\frac{h^2}{12} \mathcal{D}_{\sigma_n}\bigl(t_n,u(t_n)\bigr)
        + O(h^3)
    \end{displaymath}
    uniformly in $n$.
    The uniformity of the $O(h^3)$ remainder, together with \eqref{eq:curvature_defect_relation-2}, proves \cref{thm:lte-2}.

    Finally, suppose that $u\in C^5([t_0,T])$.
    Taylor's theorem gives, uniformly in $n$,
    \begin{align*}
        \frac{u(t_{n+1})-u(t_n)}{h}
        &=
        u'(t_n)
        +
        \frac{h}{2}u''(t_n)
        +
        \frac{h^2}{6}u'''(t_n)
        +
        \frac{h^3}{24}u^{(4)}(t_n)
        +
        O(h^4),\\
        \delta_n
        &=
        hu''(t_n)
        +
        \frac{h^2}{2}u'''(t_n)
        +
        \frac{h^3}{6}u^{(4)}(t_n)
        +
        O(h^4).
    \end{align*}
    It follows that
    \begin{align*}
        \delta_n^2
        &=
        h^2u''(t_n)^2
        +
        h^3u''(t_n)u'''(t_n)
        +
        O(h^4),\\
        \delta_n^3
        &=
        h^3u''(t_n)^3
        +
        O(h^4),
    \end{align*}
    uniformly in $n$.
    Substituting these expansions into
    \eqref{eq:lte_expanded}, we obtain
    \begin{align*}
        \ell_{n+1}
        =&
        -\frac{h^2}{12} \mathcal{D}_{\sigma_n}\bigl(t_n, u(t_n)\bigr) \\
        & \,
        - \frac{h^3}{24} \left( u^{(4)}(t_n) 
        - \frac{ 6u'(t_n)u''(t_n)u'''(t_n) }{ \sigma_n^2 + u'(t_n)^2 } + \frac{ 3\left(u'(t_n)^2 - \sigma_n^2\right)u''(t_n)^3 }{ \left(\sigma_n^2 + u'(t_n)^2\right)^2 } \right)
        + O(h^4)
    \end{align*}
    uniformly in $n$.
    Differentiating \eqref{eq:D_on_exact_sol} with respect to $t$ for $\sigma=\sigma_n$ gives
    \begin{displaymath}
        \frac{d}{dt} \mathcal{D}_{\sigma_n} (t, u(t)) 
        =
        u^{(4)}(t) 
        - \frac{6u'(t)u''(t)u'''(t)}{\sigma_n^2+u'(t)^2} 
        + \frac{3\left(u'(t)^2-\sigma_n^2\right)u''(t)^3}{\left(\sigma_n^2+u'(t)^2\right)^2}.
    \end{displaymath}
    Therefore,
    \begin{equation}
        \label{eq:lte-2}
        \ell_{n+1}
        =
        - \frac{h^2}{12} \mathcal{D}_{\sigma_n}\bigl(t_n, u(t_n)\bigr)
        - \frac{h^3}{24} \left. \frac{d}{dt} \mathcal{D}_{\sigma_n} (t, u(t))  \right|_{t = t_n} 
        + O(h^4)
    \end{equation}
    uniformly in $n$.

    We claim that 
    \begin{equation}   \label{eq:lte-3}
        h \left. \frac{d}{dt} \mathcal{D}_{\sigma_n} (t, u(t))  \right|_{t = t_n} 
        = 
        - \mathcal{D}_{\sigma_n}(t_n, u(t_n)) 
        + \mathcal{D}_{\sigma_n}(t_{n+1}, u(t_{n+1}))
        + O(h^2)
    \end{equation}    
    uniformly in $n$.
    For every $s\in[t_n, t_{n+1}]$,
    \begin{align*}
        \sigma_n^2+u'(s)^2
        &\geq
        \sigma_n^2+u'(t_n)^2 - \left|u'(s)^2-u'(t_n)^2\right|
        \\
        &\geq
        \eta - |u'(s) - u'(t_n)| |u'(s) + u'(t_n)|
        \\
        &\geq
        \eta - 2 M_1 M_2 h.
    \end{align*}
    Hence, for all sufficiently small $h>0$,
    \begin{equation}
        \label{eq:stepwise_nondegeneracy}
        \sigma_n^2 + u'(s)^2
        \geq
        \frac{\eta}{2},
        \qquad
        s\in[t_n, t_{n+1}].
    \end{equation}
    By \eqref{eq:stepwise_nondegeneracy},
    \begin{displaymath}
        \frac{1}{\sigma_n^2+u'(s)^2}
        \leq
        \frac{2}{\eta} 
        \qquad\text{and}\qquad
        0
        \leq
        \frac{\sigma_n^2}{\sigma_n^2+u'(s)^2}
        \leq
        1,
        \qquad
        s\in[t_n,t_{n+1}].
    \end{displaymath}
    Since $u\in C^5([t_0,T])$, differentiating \eqref{eq:D_on_exact_sol} twice and using these bounds shows that the mapping $t\mapsto\mathcal{D}_{\sigma_n}\bigl(t,u(t)\bigr)$ has a uniformly bounded second derivative on $[t_n, t_{n+1}]$.
    Taylor's theorem therefore gives \eqref{eq:lte-3}.

    Combining \eqref{eq:lte-2} with \eqref{eq:lte-3} and applying \eqref{eq:curvature_defect_relation-2}, we obtain
    \begin{displaymath}
        \ell_{n+1}
        =
        -\frac{h^2}{24} \sigma_n \left[  \left\lVert \gamma_{\sigma_n}'(t_n) \right\rVert^3 \kappa_{\sigma_n}'(t_n)  + \left\lVert \gamma_{\sigma_n}'(t_{n+1}) \right\rVert^3 \kappa_{\sigma_n}'(t_{n+1}) \right]
        + O(h^4)
    \end{displaymath}
    uniformly in $n$.
    The uniformity of the $O(h^4)$ remainder proves \cref{thm:lte-3}.
\end{proof}

\begin{corollary}[Consistency]
    \label{cor:consistency_of_SE}
    Suppose that hypothesis~\ref{H2} holds and that $F\in C^2(\Omega)$.
    Then the following statements hold.
    \begin{enumerate}[label=\upshape(\alph*)]
        \item \label{cor:consistency_of_SE-2nd}   
        Suppose that $u\in C^3([t_0,T])$.
        Then SE is consistent of order~$2$.
        More precisely, there exist constants $h_1 \in (0, H]$ and $C>0$, independent of $h$ and $n$, such that
        \begin{displaymath}
            |\ell_{n+1}|
            \leq
            Ch^2
        \end{displaymath}
        whenever $0<h\leq h_1$ and $t_{n+1}\leq T$.

        \item \label{cor:consistency_of_SE-3rd} 
        Suppose that $u\in C^4([t_0,T])$ and that the positive scales $\sigma_n$ are chosen such that%
        \begin{equation}
            \label{eq:sol_graph_3rd}
            \kappa_{\sigma_n}'(t_n)=0
        \end{equation}
        for every $n$ such that $t_{n+1}\leq T$.
        Then SE is consistent of order~$3$.
        More precisely, there exist constants $h_2 \in (0, H]$ and $C>0$, independent of $h$ and $n$, such that
        \begin{displaymath}
            |\ell_{n+1}|
            \leq
            Ch^3
        \end{displaymath}
        whenever $0<h\leq h_2$ and $t_{n+1}\leq T$.

        \item \label{cor:consistency_of_SE-4th}
        Suppose that $u\in C^5([t_0,T])$ and that the positive scales $\sigma_n$ are chosen such that%
        \begin{equation} \label{eq:sol_graph_4th}
            \left\lVert\gamma_{\sigma_n}'(t_n)\right\rVert^3 \kappa_{\sigma_n}'(t_n) + \left\lVert\gamma_{\sigma_n}'(t_{n+1})\right\rVert^3
            \kappa_{\sigma_n}'(t_{n+1})
            =
            0
        \end{equation}
        for every $n$ such that $t_{n+1}\leq T$.
        Then SE is consistent of order~$4$.
        More precisely, there exist constants $h_3 \in (0, H]$ and $C>0$, independent of $h$ and $n$, such that
        \begin{displaymath}
            |\ell_{n+1}|
            \leq
            Ch^4
        \end{displaymath}
        whenever $0<h\leq h_3$ and $t_{n+1}\leq T$.
    \end{enumerate}
\end{corollary}

\begin{proof}
    For \cref{cor:consistency_of_SE-2nd}, \eqref{eq:normalized_curvature_derivative} and hypothesis~\ref{H2} give
    \begin{displaymath}
        \left| \sigma_n \left\lVert \gamma_{\sigma_n}'(t_n) \right\rVert^3 \kappa_{\sigma_n}'(t_n) \right|
        =
        \left| u'''(t_n) - \frac{ 3u'(t_n)u''(t_n)^2 }{ \sigma_n^2+u'(t_n)^2 } \right|
        \leq
        M_3 + \frac{3M_1M_2^2}{\eta}.
    \end{displaymath}
    Since $\omega(h)\to0$ as $h\searrow0$, there exists $h_1 \in (0, H]$ such that $\omega(h)\leq1$ whenever $0 < h \leq h_1$.
    Hence, \cref{thm:lte}~\ref{thm:lte-1} gives
    \begin{displaymath}
        |\ell_{n+1}|
        \leq
        \left( \frac{M_3}{12} + \frac{M_1M_2^2}{4\eta} + 1 \right) h^2.
    \end{displaymath}
    This proves \cref{cor:consistency_of_SE-2nd}.

    Under the condition in \cref{cor:consistency_of_SE-3rd}, the leading term in \cref{thm:lte}~\ref{thm:lte-2} vanishes.
    Therefore,
    \begin{displaymath}
        |\ell_{n+1}|
        \leq
        Ch^3
    \end{displaymath}
    for all sufficiently small $h>0$.
    This proves \cref{cor:consistency_of_SE-3rd}.
    Similarly, \eqref{eq:sol_graph_4th} and \cref{thm:lte}~\ref{thm:lte-3} imply \cref{cor:consistency_of_SE-4th}.
\end{proof}

For comparison, the local truncation error of CN is defined by
\begin{displaymath}
    \ell_{n+1}^{\mathrm{CN}}
    :=
    \frac{u(t_{n+1})-u(t_n)}{h} - \frac{u'(t_{n+1})+u'(t_n)}{2}.
\end{displaymath}
If $u\in C^4([t_0,T])$, Taylor's theorem gives
\begin{equation}    \label{eq:local_truncation_error_CN}
    \ell_{n+1}^{\mathrm{CN}}
    =
    -\frac{h^2}{12} u'''(t_n)
    + O(h^3)
\end{equation}
uniformly in $n$.
The leading $O(h^2)$ terms in the two local truncation errors vanish under different conditions.
For CN, the condition is $u'''(t_n) = 0$; for SE, it is $\kappa_{\sigma_n}'(t_n) = 0$.
This local comparison does not imply a global error inequality or an efficiency advantage.

\subsection{Convergence}

The following theorem converts a uniform local truncation error bound into a global convergence estimate for SE.

\begin{theorem}
    \label{thm:convergence_of_SE}
    Suppose that hypothesis~\ref{H1} holds.
    Suppose additionally that there exist $p\geq2$, $C_{\ell}>0$, and $h_{\ell}>0$ such that, for every $h\in(0,h_{\ell}]$, a sequence $\{\sigma_n\}_{n=0}^{N_h-1}$ of positive scales is given and the corresponding local truncation errors satisfy
    \begin{equation} \label{ineq:convergence_SE-1}
        \max_{0\leq n\leq N_h-1}|\ell_{n+1}|
        \leq
        C_{\ell}h^p.
    \end{equation}
    Then there exists a constant $H\equiv H(t_0, T, \rho, L, p, C_{\ell}, h_{\ell})>0$ such that, for every $h\in(0,H]$, there exists a unique sequence $\{u_n\}_{n=0}^{N_h}$ satisfying \eqref{mth:SE} with $\{(t_n,u_n)\}_{n=0}^{N_h}\subseteq U_\rho(u)$, and the estimate
    \begin{equation}
        \label{ineq:convergence_SE-2}
        \max_{0\leq n\leq N_h}|u(t_n)-u_n|
        \leq
        2C_{\ell}(T-t_0) \exp\left(8L(T-t_0)\right)h^p.
    \end{equation}
\end{theorem}

\begin{proof}
    For every $h\in(0,h_\ell]$ and $n=0,1,\ldots,N_h-1$, define $\Sigma_n(t,x;s,y) := \sigma_n$.
    We apply \cref{lem:lte_to_global_error_SE} with $P=C_\ell$, $\varepsilon=h_\ell$, $Q=0$, and $\Lambda=2L$.
    Then \cref{lem:Lipschitz_C_sigma} and hypothesis~\ref{H1} give \eqref{ineq:lte_to_global_error_SE-0}, while \eqref{ineq:convergence_SE-1} gives \eqref{ineq:lte_to_global_error_SE-1}.
    Therefore, \cref{lem:lte_to_global_error_SE} gives the desired sequence and \eqref{ineq:convergence_SE-2}.
\end{proof}

\begin{corollary}[Convergence]
    \label{cor:convergence_orders_SE}
    Suppose that hypotheses~\ref{H1} and~\ref{H2} hold and that $F\in C^2(\Omega)$.
    Then the following statements hold.
    \begin{enumerate}[label=\upshape(\alph*)]
        \item
        \label{cor:convergence_orders_SE-2nd}
        If $u\in C^3([t_0,T])$, then SE has convergence order~$2$.

        \item
        \label{cor:convergence_orders_SE-3rd}
        If $u\in C^4([t_0,T])$ and
        \begin{displaymath}
            \kappa_{\sigma_n}'(t_n)=0
        \end{displaymath}
        whenever $t_{n+1}\leq T$, then SE has convergence order~$3$.

        \item
        \label{cor:convergence_orders_SE-4th}
        If $u\in C^5([t_0,T])$ and
        \begin{displaymath}
            \left\lVert\gamma_{\sigma_n}'(t_n)\right\rVert^3 \kappa_{\sigma_n}'(t_n)
            +
            \left\lVert\gamma_{\sigma_n}'(t_{n+1})\right\rVert^3 \kappa_{\sigma_n}'(t_{n+1})
            =
            0
        \end{displaymath}
        whenever $t_{n+1}\leq T$, then SE has convergence order~$4$.
    \end{enumerate}
\end{corollary}

\begin{proof}
    If $u\in C^3([t_0,T])$, then \cref{cor:consistency_of_SE}~\ref{cor:consistency_of_SE-2nd} and \cref{thm:convergence_of_SE}, applied with $p=2$, imply \cref{cor:convergence_orders_SE-2nd}.
    The remaining parts follow similarly, with $p=3$ and $p=4$, respectively.
\end{proof}

\subsection{Stability on the negative real axis}

Consider the test equation
\begin{equation}
    \label{ODE:Dahlquist}
    \begin{cases}
        u'=\lambda u, & t>0,\\
        u(0)=u_0,
    \end{cases}
\end{equation}
where $\lambda,u_0\in\mathbb{R}$.

Since $\mathcal{C}_{\sigma}$ is not homogeneous, applying SE to the test equation does not in general produce a stability function depending only on $h\lambda$.
Nevertheless, SE satisfies an unconditional decay property on the negative real test equation.

\begin{theorem} \label{thm:stability_SE}
    Consider \eqref{ODE:Dahlquist} with $\lambda<0$, and let $h>0$.
    Let $\{\sigma_n\}_{n=0}^{\infty}$ be an arbitrary sequence of positive scales.
    Then \eqref{eq:SE_implicit_step} has a unique solution $v_n^{\ast}\in\mathbb{R}$ for every $n\in\mathbb{N}\cup\{0\}$.
    Setting $u_{n+1} := v_n^{\ast}$, we have
    \begin{equation}    \label{ineq:stability_SE-1}
        |u_{n+1}|
        \leq
        |u_n|.
    \end{equation}
    The inequality is strict whenever $u_n\neq0$.
    Moreover, for every such choice of positive scales,
    \begin{equation}    \label{eq:stability_SE-2}
        \lim_{n\to\infty}u_n=0.
    \end{equation}
\end{theorem}

\begin{proof}
    Fix $n\in\mathbb{N}\cup\{0\}$.
    Define the functions $\varphi, \psi: \mathbb{R} \to \mathbb{R}$ by
    \begin{align*}
        \varphi(v)
        &:= v - u_n - h \, \mathcal{C}_{\sigma_n}(\lambda v, \lambda u_n),    \\
        \psi(v)
        &:= \mathcal{C} \left( \frac{\lambda v}{\sigma_n}, \frac{\lambda u_n}{\sigma_n} \right).
    \end{align*}
    By the definition of $\mathcal{C}_{\sigma_n}$, we have
    \begin{displaymath}
        \varphi(v)
        =
        v - u_n - h \sigma_n \, \psi(v).
    \end{displaymath}
    By \cite[Lems.~2.1 and~A.1]{2026a_Jung}, $\mathcal{C}$ is continuous and strictly increasing in its first variable.
    Since $\lambda<0$, the function $\psi$ is continuous and strictly decreasing.
    Consequently, $\varphi$ is continuous and strictly increasing.

    We claim that
    \begin{displaymath}
        \lim_{v\to-\infty}\varphi(v)=-\infty
        \qquad\text{and}\qquad
        \lim_{v\to\infty}\varphi(v)=\infty.
    \end{displaymath}
    Indeed, by \cite[Lem.~A.2~(m)]{2026a_Jung}, the function $\psi$ is bounded as $|v| \to \infty$.
    Hence, there exist constants $M>0$ and $V>0$ such that $|\varphi(v)-v| \leq M$ whenever $|v|\geq V$.
    Therefore, $v-M \leq \varphi(v) \leq v+M$ whenever $|v|\geq V$, proving the claim.

    By the claim and the continuity of $\varphi$, the Intermediate Value Theorem asserts that there exists a point $v_n^{\ast} \in \mathbb{R}$ such that $\varphi(v_n^{\ast}) = 0$.
    Since $\varphi$ is strictly increasing, this point is unique.
    Setting $u_{n+1} := v_n^{\ast}$, we obtain the unique solution of \eqref{eq:SE_implicit_step}.

    Define
    \begin{displaymath}
        \alpha_n
        :=
        \frac{\lambda u_{n+1}}{\sigma_n}
        \qquad\text{and}\qquad
        \beta_n
        :=
        \frac{\lambda u_n}{\sigma_n}.
    \end{displaymath}
    Since $\varphi(u_{n+1})=0$, the definition of $\mathcal{C}_{\sigma_n}$ gives
    \begin{equation}
        \label{eq:stability_SE-1}
        \alpha_n - \beta_n
        =
        h\lambda \, \mathcal{C}(\alpha_n, \beta_n).
    \end{equation}
    Multiplying \eqref{eq:stability_SE-1} by $\alpha_n + \beta_n$, we obtain
    \begin{displaymath}
        \alpha_n^2 - \beta_n^2
        =
        h\lambda \, \mathcal{C}(\alpha_n, \beta_n) (\alpha_n + \beta_n)
        \leq
        0,
    \end{displaymath}
    where the inequality follows from the fact that $\mathcal{C}(\alpha_n, \beta_n)$ and $\alpha_n + \beta_n$ have the same sign \cite[Lem.~A.2~(g)]{2026a_Jung}.
    Thus, $\alpha_n^2 \leq \beta_n^2$, which yields \eqref{ineq:stability_SE-1}.

    Suppose that equality holds in \eqref{ineq:stability_SE-1}.
    Then $\mathcal{C}(\alpha_n, \beta_n) (\alpha_n + \beta_n) = 0$.
    By \cite[Lem.~A.2~(f)]{2026a_Jung}, $\mathcal{C}(\alpha_n, \beta_n) = 0$ if and only if $\alpha_n + \beta_n = 0$.
    Hence, both factors vanish.
    Equation \eqref{eq:stability_SE-1} then gives $\alpha_n = \beta_n$.
    Consequently, $\alpha_n = \beta_n = 0$, and therefore $u_n = 0$.
    Thus, the inequality is strict whenever $u_n \neq 0$.

    Now we show \eqref{eq:stability_SE-2}.
    Iterating \eqref{ineq:stability_SE-1}, we obtain
    \begin{displaymath}
        0
        \leq
        |u_{n+1}|
        \leq
        |u_n|
        \leq
        \cdots
        \leq
        |u_0|.
    \end{displaymath}
    Since the sequence $\{|u_n|\}_{n=0}^{\infty}$ is nonincreasing and bounded below by zero, there exists $\beta \geq 0$ such that $|u_n| \to \beta$ as $n \to \infty$.

    We claim that $\beta = 0$.
    Suppose, to the contrary, that $\beta > 0$.
    Since $|u_n|\leq|u_0|$, the sequence $\bigl\{(u_n,u_{n+1})\bigr\}_{n=0}^{\infty}$ is bounded in $\mathbb{R}^2$.
    Hence, there exist a strictly increasing sequence of indices $\{n_j\}_{j=1}^{\infty}$ and $\mu,\nu\in\mathbb{R}$ such that $(u_{n_j}, u_{n_j+1}) \to (\mu,\nu)$ as $j\to\infty$.
    Since $|u_n|\to\beta$, we have $|\mu|=|\nu|=\beta$, and therefore either $\nu=\mu$ or $\nu=-\mu$.

    Suppose first that $\nu=\mu$.
    Since $\mathcal{C}_{\sigma_{n_j}}(\lambda\mu,\lambda\mu) = \lambda\mu$, \cref{lem:Lipschitz_C_sigma} gives
    \begin{displaymath}
        \left| \mathcal{C}_{\sigma_{n_j}} \left( \lambda u_{n_j+1}, \lambda u_{n_j} \right) - \lambda\mu \right|
        \leq
        2 |\lambda| \left( |u_{n_j+1}-\mu| + |u_{n_j}-\mu| \right)
        \to
        0.
    \end{displaymath}
    Passing to the limit in the SE relation, we obtain
    \begin{displaymath}
        0
        =
        \nu - \mu
        =
        \lim_{j\to\infty} \left(u_{n_j+1}-u_{n_j}\right)
        =
        h \lim_{j\to\infty} \mathcal{C}_{\sigma_{n_j}} \left( \lambda u_{n_j+1}, \lambda u_{n_j} \right)
        =
        h \lambda \mu.
    \end{displaymath}
    Hence, $\mu=0$, contradicting $|\mu|=\beta>0$.

    Suppose now that $\nu=-\mu$.
    Since $\mathcal{C}_{\sigma_{n_j}}(-\lambda\mu,\lambda\mu) = 0$, \cref{lem:Lipschitz_C_sigma} gives
    \begin{displaymath}
        \left| \mathcal{C}_{\sigma_{n_j}}  \left( \lambda u_{n_j+1}, \lambda u_{n_j} \right) \right|
        \leq
        2|\lambda| \left( |u_{n_j+1}+\mu| + |u_{n_j}-\mu| \right)
        \to
        0.
    \end{displaymath}
    As in the preceding case, passing to the limit in the SE relation gives $-2\mu = 0$, which again contradicts $|\mu| = \beta > 0$.

    Therefore, $\beta = 0$, and hence \eqref{eq:stability_SE-2} follows.
\end{proof}

To the best of our knowledge, specular differentiation over $\mathbb{C}$ has not yet been developed.
We therefore restrict the stability analysis to real $\lambda < 0$.

\section{Choice of the scale for third-order accuracy}
\label{sec:scale_third_order}

The scale $\sigma_n$ affects the leading terms in the local truncation error of SE.
The solution-graph conditions \eqref{eq:sol_graph_3rd} and \eqref{eq:sol_graph_4th} describe how it may be chosen to obtain third- and fourth-order accuracy.
However, these conditions involve the exact solution and cannot be evaluated directly in a numerical computation.

While $\gamma_\sigma$ depends on the exact solution, $\mathcal{D}_\sigma$ can be evaluated directly from $F$ and its derivatives at numerical points $(t_n,u_n)$.
We seek third-order accuracy by choosing $\sigma$ to cancel $\mathcal{D}_\sigma$, while the next section considers the corresponding fourth-order problem.

\Cref{cor:consistency_of_SE-3rd} of \cref{cor:consistency_of_SE} shows that \eqref{eq:sol_graph_3rd} is sufficient for third-order consistency.
\Cref{eq:curvature_defect_relation-2} shows that condition~\eqref{eq:sol_graph_3rd} is equivalent along the exact solution to
\begin{equation}
    \label{eq:field_3rd}
    \mathcal{D}_{\sigma_n}\bigl(t_n,u(t_n)\bigr)
    =
    0.
\end{equation}
At the beginning of the $n$-th step, we replace the unavailable point $\bigl(t_n,u(t_n)\bigr)$ by its numerical counterpart $(t_n,u_n)$.
The resulting value $\mathcal{D}_\sigma(t_n,u_n)$ can be computed directly from $F$ and its derivatives.
A positive scale giving exact field cancellation need not exist.
We therefore consider, for each fixed point $(t,x)$ in the domain of $F$, the one-dimensional problem
\begin{equation}
    \label{opt:field_3rd}
    \inf_{\sigma>0}
    \left|
        \mathcal{D}_\sigma(t,x)
    \right|.
\end{equation}

\begin{definition}
    A minimizer $\sigma>0$ of the problem \eqref{opt:field_3rd} is said to be
    \emph{optimal} if
    \begin{displaymath}
        \mathcal{D}_\sigma(t,x)
        =
        0.
    \end{displaymath}
\end{definition}

Thus, an optimal minimizer gives exact field cancellation, whereas a nonoptimal minimizer realizes the smallest positive residual.
If the infimum is not attained, no positive minimizing scale exists and the boundary behavior as $\sigma \searrow 0$ or  $\sigma \to \infty$ must be considered.

\subsection{Classification based on scaled field quantities}

The following theorem classifies the positive defect-cancelling scales of \eqref{opt:field_3rd} and determines whether the residual infimum is zero or positive in the remaining cases.

\begin{theorem} \label{thm:sigma_selection_3rd_order}
    Let $\Omega\subset\mathbb{R}^2$ be open and let $F\in C^2(\Omega)$.
    Fix $(t, x) \in \Omega$.
    In the following statements, all quantities involving $F$ are evaluated at $(t,x)$.
    \begin{enumerate}[label=\upshape(G\arabic*),ref=(G\arabic*),itemsep=1ex]
        \item \label{G1}
        If $F\mathcal{L}_F F=0$, then every $\sigma>0$ is a minimizer of \eqref{opt:field_3rd} and
        \begin{displaymath}
            \min_{\sigma > 0} \left|\mathcal{D}_{\sigma}\right| = \left\vert \mathcal{L}^2_F F \right\vert .
        \end{displaymath}
        These minimizers are optimal if and only if $\mathcal{L}_F^2F=0$.

        \item \label{G2}
        Suppose that $F\mathcal{L}_F F\neq0$ and $F\mathcal{L}_F^2F \leq 0$.
        Then the problem \eqref{opt:field_3rd} has no minimizer and
        \begin{displaymath}
            \inf_{\sigma>0}
            \left|\mathcal{D}_{\sigma}\right|
            =
            \lim_{\sigma\to\infty}
            \left|\mathcal{D}_{\sigma}\right|
            =
            \left|\mathcal{L}_F^2F\right|.
        \end{displaymath}
        
        \item \label{G3}
        Suppose that $F\mathcal{L}_F F\neq0$, $F\mathcal{L}_F^2F>0$, and $3\left(\mathcal{L}_F F\right)^2 > F\mathcal{L}_F^2F$.
        Then the problem \eqref{opt:field_3rd} has a unique optimal minimizer $\sigma^{\ast} (t,x) > 0$, which is given by%
        \begin{equation} \label{eq:pointwise_scale_selection_F}
            \sigma^{\ast}
            =
            \left( \frac{3F \left(\mathcal{L}_F F\right)^2}{\mathcal{L}_F^2 F} - F^2 \right)^{\frac{1}{2}},
        \end{equation}
        and for which 
        \begin{displaymath}
            \min_{\sigma>0} \left|\mathcal{D}_\sigma\right|
            =
            \left| \mathcal{D}_{\sigma^{\ast}} \right|
            =
            0.
        \end{displaymath}

        \item \label{G4}
        Suppose that $F\mathcal{L}_F F\neq0$, $F\mathcal{L}_F^2F>0$, and $3\left(\mathcal{L}_F F\right)^2 \leq F\mathcal{L}_F^2F$.
        Then the problem \eqref{opt:field_3rd} has no minimizer and
        \begin{equation}    \label{eq:sigma_selection_3rd_order_G4}
            \inf_{\sigma > 0} \left|\mathcal{D}_{\sigma}\right|
            =
            \lim_{\sigma \searrow 0} \left|\mathcal{D}_{\sigma}\right|
            =
            \left| \mathcal{L}_F^2F - \frac{3\left(\mathcal{L}_F F\right)^2}{F} \right|.
        \end{equation}
    \end{enumerate}
\end{theorem}

\begin{proof}
    By \eqref{def:normalized_curvature_defect}, we have
    \begin{displaymath}
        \mathcal{D}_\sigma
        =
        \mathcal{L}_F^2F - \frac{ 3F(\mathcal{L}_FF)^2 }{ \sigma^2 + F^2}.
    \end{displaymath}
    First, suppose that $F\mathcal{L}_F F = 0$.
    Then $\mathcal{D}_\sigma = \mathcal{L}_F^2 F$ for every $\sigma > 0$.
    Thus, every $\sigma > 0$ is a minimizer, and these minimizers are optimal if and only if $\mathcal{L}_F^2F=0$.
    This proves case~\ref{G1}.

    Now, suppose that $F\mathcal{L}_F F \neq 0$.
    Then $F \neq 0$ and $\mathcal{L}_F F \neq 0$.
    Define the function $\varphi : (0, \infty) \to (0, 1)$ by 
    \begin{displaymath}
        \varphi(\sigma) := \frac{F^2}{\sigma^2 + F^2}.
    \end{displaymath}
    This function is a strictly decreasing bijection from $(0,\infty)$ onto $(0, 1)$.
    Since $F \neq 0$, we have 
    \begin{equation} \label{eq:pointwise_defect-1}
        |\mathcal{D}_{\sigma}| 
        = 
        \frac{1}{|F|} \left\vert F \mathcal{L}_F^2 F - 3 (\mathcal{L}_F F)^2 \varphi(\sigma) \right\vert.
    \end{equation}
    
    If $F\mathcal{L}_F^2F \leq 0$, then $F\mathcal{L}_F^2F < 3 (\mathcal{L}_F F)^2 \varphi(\sigma)$ for every $\sigma > 0$.
    Therefore, $\mathcal{D}_\sigma\neq 0$ for every $\sigma>0$.
    Moreover, \eqref{eq:pointwise_defect-1} yields
    \begin{displaymath}
        |\mathcal{D}_\sigma|
        =
        |\mathcal{L}_F^2F| + \frac{3(\mathcal{L}_FF)^2}{|F|} \varphi(\sigma)
        >
        |\mathcal{L}_F^2F|
    \end{displaymath}
    for every $\sigma>0$.
    Since $\varphi(\sigma)\to0$ as $\sigma\to\infty$, we obtain
    \begin{displaymath}
        \inf_{\sigma>0}|\mathcal{D}_\sigma|
        =
        \lim_{\sigma\to\infty}|\mathcal{D}_\sigma|
        =
        |\mathcal{L}_F^2F|.
    \end{displaymath}
    The infimum is not attained since $\varphi(\sigma) > 0$ for every $\sigma>0$.
    This proves \ref{G2}.

    It remains to consider the case $F\mathcal{L}_F^2 F > 0$. 
    First, suppose that $F\mathcal{L}_F^2F < 3(\mathcal{L}_FF)^2$.
    Then
    \begin{displaymath}
        0
        <
        \frac{F\mathcal{L}_F^2F}{3(\mathcal{L}_FF)^2}
        <
        1.
    \end{displaymath}
    Since $\varphi$ is a bijection from $(0,\infty)$ onto $(0,1)$, there exists a unique $\sigma^\ast > 0$ such that
    \begin{equation}    \label{eq:pointwise_defect-2}
        \varphi(\sigma^\ast)
        =
        \frac{F\mathcal{L}_F^2F}{3(\mathcal{L}_FF)^2}.
    \end{equation}
    By \eqref{eq:pointwise_defect-1}, $\mathcal{D}_{\sigma^\ast}=0$.
    Solving the equation \eqref{eq:pointwise_defect-2} for $\sigma^\ast$ gives
    \begin{displaymath}
        \sigma^\ast
        =
        \left(\frac{3F(\mathcal{L}_FF)^2}{\mathcal{L}_F^2F} - F^2\right)^{\frac{1}{2}}.
    \end{displaymath}
    Thus, $\sigma^\ast$ is the unique optimal minimizer.
    This proves case~\ref{G3}.

    On the other hand, suppose that $3(\mathcal{L}_F F)^2 \leq F\mathcal{L}_F^2 F$.
    Then, for every $\sigma>0$,
    \begin{displaymath}
        3(\mathcal{L}_FF)^2\varphi(\sigma)
        <
        3(\mathcal{L}_FF)^2
        \leq
        F\mathcal{L}_F^2F.
    \end{displaymath}
    Therefore, $\mathcal{D}_\sigma\neq0$ for every $\sigma>0$.
    Moreover, \eqref{eq:pointwise_defect-1} gives
    \begin{displaymath}
        |\mathcal{D}_\sigma|
        =
        \frac{F\mathcal{L}_F^2F - 3(\mathcal{L}_FF)^2\varphi(\sigma)}{|F|}
        >
        \frac{F\mathcal{L}_F^2F - 3(\mathcal{L}_FF)^2 }{|F|}
        =
        \left|\mathcal{L}_F^2F - \frac{3(\mathcal{L}_FF)^2}{F}\right|.
    \end{displaymath}
    Since $\varphi(\sigma)\to1$ as $\sigma\searrow0$, we obtain
    \begin{displaymath}
        \inf_{\sigma>0}|\mathcal{D}_\sigma|
        =
        \lim_{\sigma\searrow0}|\mathcal{D}_\sigma|
        =
        \left| \mathcal{L}_F^2F - \frac{3(\mathcal{L}_FF)^2}{F}\right|.
    \end{displaymath}
    The infimum is not attained since $\varphi(\sigma) < 1$ for every $\sigma > 0$.
    This proves case~\ref{G4}.
\end{proof}

In case~\ref{G1}, every positive scale is a minimizer, and these minimizers are optimal if and only if $\mathcal{L}_F^2F=0$; otherwise, the criterion does not favor any positive scale.
In case~\ref{G2}, the criterion does not select a finite preferred scale; the infimum is approached as $\sigma \to \infty$.
In case~\ref{G4}, the defect is reduced by taking $\sigma$ as small as practicable.
In case~\ref{G3}, the unique optimal minimizer $\sigma^{\ast}$ makes the field-based coefficient of the leading $O(h^2)$ term in the local truncation error vanish at the selected point.
When this point lies on the exact trajectory, this $O(h^2)$ term is cancelled.

The following example revisits the ellipse equation in \cref{ex:SE_trace}.
Away from $t=p$, the right-hand side satisfies the hypotheses of case~\ref{G3}, and the selected scale $\sigma^{\ast}$ given by \eqref{eq:pointwise_scale_selection_F} recovers the exact-tracing value $\sigma=\frac{b}{a}$.

\begin{example} \label{ex:ellipse_F}
    Let $u$ be the upper-branch solution given by \eqref{eq:ellipse_sol} on $[p,T]$, where $p<T<p+a$.
    At $t=p$, we have $F=0$ and $\mathcal{L}_F^2F=0$.
    Thus, case~\ref{G1} applies, and every $\sigma>0$ is an optimal minimizer.

    Now suppose that $t\neq p$.
    Since $u\neq q$ on the domain of $F$, we have
    \begin{displaymath}
        F\mathcal{L}_F F
        =
        \left(\frac{b}{a}\right)^4
        \frac{
            (t-p)
            \left(
                (u-q)^2+\left(\frac{b}{a}\right)^2(t-p)^2
            \right)
        }{
            (u-q)^4
        }
        \neq 0.
    \end{displaymath}
    Moreover,
    \begin{displaymath}
        F\mathcal{L}_F^2F
        =
        3\left(\frac{b}{a}\right)^6
        \frac{
            (t-p)^2
            \left(
                (u-q)^2+\left(\frac{b}{a}\right)^2(t-p)^2
            \right)
        }{
            (u-q)^6
        }
        >
        0
    \end{displaymath}
    and
    \begin{displaymath}
        3\left(\mathcal{L}_F F\right)^2
        -
        F\mathcal{L}_F^2F
        =
        3\left(\frac{b}{a}\right)^4
        \frac{
            (u-q)^2+\left(\frac{b}{a}\right)^2(t-p)^2
        }{
            (u-q)^4
        }
        >
        0.
    \end{displaymath}
    Hence, case~\ref{G3} applies.
    The preceding formulas give the unique optimal minimizer $\sigma^{\ast} = \frac{b}{a}$, which is the value for which SE has zero local truncation error along the corresponding elliptic branch.

    For the unit-circle solution of \eqref{ode:circle_eq}, corresponding to $p=q=0$ and $a=b=1$, this formula again gives $\sigma^{\ast} = 1$.
\end{example}

\Cref{thm:sigma_selection_3rd_order} classifies the possible scale choices at each point; it does not guarantee either a finite preferred scale at every step or a single fixed scale throughout the computation.

\begin{example} \label{ex:Dahlquist}
    Consider the Dahlquist equation \eqref{ODE:Dahlquist} with $\lambda \in \mathbb{R}$.
    For $F(u) = \lambda u$, a direct calculation gives $\mathcal{L}_F F = \lambda^2u$ and $\mathcal{L}_F^2 F = \lambda^3u$.
    If $\lambda=0$ or $u=0$, then $F=\mathcal{L}_F F=\mathcal{L}_F^2 F=0$, and hence case~\ref{G1} applies and every $\sigma>0$ is an optimal minimizer.
    Now suppose that $\lambda \neq 0$ and $u \neq 0$.
    Then
    \begin{displaymath}
        F\mathcal{L}_F F
        =
        \lambda^3 u^2
        \neq0 
        \qquad\text{and}\qquad
        3\left(\mathcal{L}_F F\right)^2
        =
        3\lambda^4u^2
        >
        \lambda^4u^2
        =
        F\mathcal{L}_F^2F
        >
        0.
    \end{displaymath}
    Hence, case~\ref{G3} applies, and its unique optimal minimizer is $\sigma^{\ast}(t,u) = \sqrt{2}\,|\lambda u|$.
    Unlike the elliptic problem in \cref{ex:SE_trace}, this scale depends on $u$ and varies along every nonzero solution.

    If $\lambda<0$, choose
    \begin{displaymath}
        \sigma_n
        :=
        \begin{cases}
            \sqrt{2}\,|\lambda u_n|, & u_n\neq0,\\
            1, & u_n=0.
        \end{cases}
    \end{displaymath}
    By \cref{thm:stability_SE}, \eqref{eq:SE_implicit_step} has a unique solution at each step, the magnitude strictly decreases at every nonzero iterate, and $u_n=0$ implies $u_{n+1}=0$.
    Hence, either the sequence becomes identically zero after finitely many steps, or it remains nonzero and converges to zero.
\end{example}

The following example shows that the applicable case among \ref{G1}--\ref{G4} can vary along a single solution trajectory.

\begin{example}
    \label{ex:autonomous_most_cases}
    Consider the equation \eqref{ODE:IVP} with $t_0=0$, $u_0\in\mathbb{R}$, and
    \begin{equation}
        \label{ODE:autonomous_most_cases}
        F(t,u) := 1-u^2.
    \end{equation}
    A direct calculation gives $\mathcal{L}_F F = -2u(1-u^2)$ and $\mathcal{L}_F^2 F = 2(1-u^2)(3u^2-1)$.

    At $u=\pm1$, case~\ref{G1} applies and every $\sigma>0$ is an optimal minimizer, since $\mathcal{D}_\sigma(u)=0$.
    At $u=0$, case~\ref{G1} also applies, but every $\sigma>0$ is a nonoptimal minimizer, since $\mathcal{D}_\sigma(0)=\mathcal{L}_F^2 F(0)=-2$.

    Now suppose that $u \notin \left\{ -1, 0, 1 \right\}$.
    Then
    \begin{displaymath}
        F \mathcal{L}_F F
        =
        -2u(1-u^2)^2
        \neq
        0 
        \qquad\text{and}\qquad
        F\mathcal{L}_F^2 F
        =
        2(1-u^2)^2(3u^2-1).
    \end{displaymath}
    Hence, if $0<|u|\leq\frac{1}{\sqrt{3}}$, then case~\ref{G2} applies.
    If $|u|>\frac{1}{\sqrt3}$, then case~\ref{G3} applies.

    In particular, the solution $u(t)=\tanh t$ with $u(0)=0$ passes from case~\ref{G1} through case~\ref{G2} and then case~\ref{G3}, while the equilibrium $u=1$, which again belongs to case~\ref{G1}, is approached only as $t\to\infty$.
\end{example}

\begin{remark}  \label{rmk:extreme_sigma}
    As $\sigma$ varies, $\mathcal{C}_\sigma$ transitions between harmonic-type and arithmetic averaging.
    Indeed, by \eqref{eq:C_infty}, $\mathcal{C}_\sigma(\alpha,\beta)$ approaches the arithmetic mean of $\alpha$ and $\beta$ as $\sigma\to\infty$.
    Consequently, the SE relation approaches the CN relation.
    On the other hand, by \eqref{eq:C_0}, $\mathcal{C}_\sigma(\alpha,\beta)$ approaches the harmonic mean of $\alpha$ and $\beta$ as $\sigma\searrow0$ when they have the same sign.
\end{remark}

Accordingly, our main interest lies in cases~\ref{G3} and~\ref{G4}.
However, the pointwise classification does not by itself provide the local truncation-error estimates required for third-order convergence.
We therefore analyze the two cases separately below.

\subsection{Optimal scale}

In case~\ref{G3}, the positive optimal scale is explicit.
The following theorem establishes third-order convergence when this scale is evaluated at $(t_n,u_n)$.

\begin{theorem}[Third-order convergence in case~\ref{G3}]
    \label{thm:third_order_SE}
    Suppose that hypothesis~\ref{H1} holds, with $F\in C^3(\Omega)$. Suppose further that
    \begin{equation}
        \label{ineq:uniform_G3_condition}
        0
        <
        F(t,x)\mathcal{L}_F^2F(t,x)
        <
        3\bigl(\mathcal{L}_FF(t,x)\bigr)^2
    \end{equation}
    for every $(t,x)\in U_\rho(u)$. Define the function $\Sigma:U_\rho(u)\to(0,\infty)$ by
    \begin{equation}
        \label{eq:G3_scale_function}
        \Sigma(t,x)
        :=
        \left(
            \frac{
                3F(t,x)\bigl(\mathcal{L}_FF(t,x)\bigr)^2
            }{
                \mathcal{L}_F^2F(t,x)
            }
            -
            F(t,x)^2
        \right)^{\frac{1}{2}}.
    \end{equation}
    Consider SE with
    \begin{equation}
        \label{eq:SE3}
        \sigma_n
        :=
        \Sigma(t_n,u_n),
        \qquad
        n=0,1,\ldots,N_h-1.
    \end{equation}
    Then there exist constants $H>0$ and $C_g>0$ such that, for every $h \in (0, H]$, there exists a unique sequence $\{u_n\}_{n=0}^{N_h}$ satisfying \eqref{mth:SE} with $\{(t_n,u_n)\}_{n=0}^{N_h} \subseteq U_\rho(u)$, and the estimate
    \begin{displaymath}
        \max_{0\leq n\leq N_h}|u(t_n)-u_n|
        \leq
        C_gh^3.
    \end{displaymath}
\end{theorem}

\begin{proof}
    By \eqref{ineq:uniform_G3_condition} and $F\in C^3(\Omega)$, $\Sigma$ is well defined, strictly positive, and continuously differentiable on $U_\rho(u)$.
    Therefore, there exist constants $\underline{\sigma}>0$ and $L_\Sigma\geq0$ such that
    \begin{equation} \label{ineq:SE3-1}
        \underline{\sigma} \leq \Sigma(t,x) 
        \qquad\text{and}\qquad
        |\Sigma(t,x)-\Sigma(t,y)| \leq L_\Sigma|x-y|        
    \end{equation}
    whenever $(t,x), (t,y)\in U_\rho(u)$.
    Moreover, case~\ref{G3} of \cref{thm:sigma_selection_3rd_order} gives
    \begin{equation}
        \label{eq:third_order_SE-2}
        \mathcal{D}_{\Sigma(t,x)}(t,x)
        =
        0
    \end{equation}
    for all $(t,x)\in U_\rho(u)$.

    Let $\{\ell_{n+1}^{\ast}\}_{n=0}^{N_h-1}$ denote the local truncation errors of SE with the exact scale sequence $\{\sigma_n^{\ast}\}_{n=0}^{N_h-1}$, where
    \begin{align*}
        \sigma_n^{\ast}
        &:=
        \Sigma\bigl(t_n,u(t_n)\bigr),\\
        \ell_{n+1}^{\ast}
        &:=
        \frac{u(t_{n+1})-u(t_n)}{h} - \mathcal{C}_{\sigma_n^{\ast}} \left(u'(t_{n+1}), u'(t_n)\right),
    \end{align*}
    for $n=0,1,\ldots,N_h-1$.
    Since $\sigma_n^\ast\geq\underline{\sigma}$, the exact scales $\{\sigma_n^{\ast}\}_{n=0}^{N_h-1}$ satisfy hypothesis~\ref{H2} with $\eta=\underline{\sigma}^2$.
    Since $u'=F(t,u)$ and $F\in C^3(\Omega)$, we have $u\in C^4([t_0,T])$.
    Therefore, \cref{thm:lte}~\ref{thm:lte-2} implies that there exist constants $C^{\ast} > 0$ and $h^{\ast} > 0$, independent of $h$ and $n$, such that 
    \begin{equation}
        \label{ineq:SE3-3}
        |\ell_{n+1}^{\ast}| 
        \leq \frac{h^2}{12} \left\vert \sigma_n^{\ast} \left\| \gamma_{\sigma_n^{\ast}}' (t_n) \right\|^3 \kappa_{\sigma_n^{\ast}}' (t_n) \right\vert + C^{\ast} h^3
        =
        C^{\ast} h^3,
    \end{equation}
    whenever $0 < h \leq h^{\ast}$ and $n = 0, 1, \ldots, N_h - 1$.
    Here, the last equality follows from \eqref{eq:curvature_defect_relation-2} and \eqref{eq:third_order_SE-2}.
    
    Fix $h\in(0,h^\ast]$ and $n\in\{0,1,\ldots,N_h-1\}$ for which $(t_n,u_n)\in U_\rho(u)$.
    By \eqref{eq:SE3} and \eqref{ineq:SE3-1}, we have
    \begin{equation}    \label{ineq:SE3-4}
        |\sigma_n^\ast - \sigma_n|
        =
        |\Sigma(t_n, u(t_n))- \Sigma(t_n, u_n)|
        \leq
        L_{\Sigma}|e_n|
    \end{equation}
    and $\sigma_n,\sigma_n^\ast\geq\underline{\sigma}$.
    Applying \cref{lem:scale_sensitivity_along_solution} with $t = t_{n+1}$, $s = t_n$, $\sigma = \underline{\sigma}$, $\mu = \sigma_n$, and $\nu = \sigma_n^{\ast}$ gives
    \begin{equation}    \label{ineq:SE3-5}
        \left| \mathcal{C}_{\sigma_n}\bigl(u'(t_{n+1}),u'(t_n)\bigr) - \mathcal{C}_{\sigma_n^{\ast}}\bigl(u'(t_{n+1}),u'(t_n)\bigr) \right|
        \leq
        \frac{M_1 M_2^2}{2 \underline{\sigma}^3} h^2 |\sigma_n - \sigma_n^{\ast}|.
    \end{equation}
    By the definitions of $\ell_{n+1}$ and $\ell_{n+1}^\ast$, we have 
    \begin{equation}   \label{eq:third_order_SE-6}
        \ell_{n+1}
        =
        \ell_{n+1}^\ast + \mathcal{C}_{\sigma_n^{\ast}}\bigl(u'(t_{n+1}),u'(t_n)\bigr) - \mathcal{C}_{\sigma_n}\bigl(u'(t_{n+1}),u'(t_n)\bigr) .
    \end{equation}
    Combining \eqref{eq:third_order_SE-6}, \eqref{ineq:SE3-3}, \eqref{ineq:SE3-5}, and \eqref{ineq:SE3-4} gives
    \begin{equation}    \label{eq:third_order_SE-7}
        |\ell_{n+1}|
        \leq
        C^{\ast} h^3 + Q h^2 |e_n|,
    \end{equation}
    where $Q := 2^{-1} \underline{\sigma}^{-3} M_1 M_2^2 L_\Sigma$.
    For every $n$, define $\Sigma_n(t, x; s, y) := \Sigma(t, x)$.
    By \eqref{eq:SE3}, $\Sigma_n(t_n, u_n; t_{n+1}, v) = \sigma_n$ is independent of $v$.
    Therefore, \cref{lem:Lipschitz_C_sigma} and hypothesis~\ref{H1} give \eqref{ineq:lte_to_global_error_SE-0} with $\Lambda=2L$.
    Since $|v-u(t_{n+1})|\geq0$, \eqref{eq:third_order_SE-7} gives \eqref{ineq:lte_to_global_error_SE-1} with $p=3$, $P=C^\ast$, and $\varepsilon = h^\ast$.
    Therefore, \cref{lem:lte_to_global_error_SE} completes the proof.
\end{proof}

The exact scale $\sigma_n^\ast=\Sigma(t_n,u(t_n))$ cancels the defect at $(t_n,u(t_n))$.
The estimate \eqref{ineq:SE3-4} and the factor $|\alpha-\beta|^2$ in \eqref{ineq:scale_sensitivity_C_sigma} show that replacing $\sigma_n^\ast$ by the computable scale $\sigma_n=\Sigma(t_n,u_n)$ changes the local truncation error by $O(h^2|e_n|)$.
Since the local truncation error enters the one-step error relation multiplied by $h$, the resulting contribution to the error recurrence is $O(h^3|e_n|)$.

\begin{example} \label{ex:SE_vs_CN}
    Consider the problem \eqref{ODE:IVP} with $t_0=0$, $u_0=1$, and
    \begin{displaymath}
        F(t,u) := -u^2.
    \end{displaymath}
    The exact solution on $[0, 1]$ is $u(t)=(1+t)^{-1}$.
    Direct calculations give $\mathcal{L}_F F = 2 u^3$ and $\mathcal{L}_F^2 F = -6 u^4$.
    Substitution into \eqref{def:normalized_curvature_defect} gives
    \begin{equation}
        \label{eq:quadratic_decay_defect}
        \mathcal{D}_\sigma(u)
        =
        -6u^4
        +
        \frac{12u^8}{\sigma^2+u^4}
        =
        \frac{6u^4(u^4-\sigma^2)}{\sigma^2+u^4}.
    \end{equation}
    Therefore, for every $u>0$, the conditions of case~\ref{G3} are satisfied, and the unique optimal minimizer is $\sigma^{\ast}(t, u) = u^2$.
    Choosing $\sigma_n=u(t_n)^2$ gives $\ell_{n+1}=O(h^3)$.
    The corresponding pointwise defect cancellation is illustrated in the left panel of \cref{fig:quadratic_decay_defect_cancellation}.
    For the numerical computation, we use $\sigma_n=u_n^2$.
    The assumptions of \cref{thm:third_order_SE} hold on $U_\rho(u)$ for every $\rho\in(0,\frac{1}{2})$, and hence this numerical scale gives third-order convergence.

    By contrast, the local truncation error of CN satisfies
    \begin{displaymath}
        \ell_{n+1}^{\mathrm{CN}}
        =
        \frac{h^2}{2(1+t_n)^4}
        +
        O(h^3),
    \end{displaymath}
    whose leading $O(h^2)$ term is nonzero.
    The numerical results in \cref{fig:SE_vs_CN} confirm third-order convergence for this scale choice and second-order convergence for CN.
\end{example}

\subsection{Scales vanishing with the step size}

Within case~\ref{G4}, if $3\left(\mathcal{L}_FF\right)^2 < F\mathcal{L}_F^2F$, then the infimum in \eqref{eq:sigma_selection_3rd_order_G4} is strictly positive.
Hence, sending $\sigma$ to zero does not cancel the defect, and third-order convergence cannot be expected in general.
In the equality subcase, however, the limiting defect is zero.
The following theorem shows that choosing positive scales satisfying $\sigma_n = O\bigl(\sqrt{h}\bigr)$ uniformly in $n$ improves the convergence order from two to three.

\begin{theorem}[Third-order convergence in the equality case of~\ref{G4}]
    \label{thm:G4_vanishing_scale_SE}
    Suppose that hypothesis~\ref{H1} holds, with $u\in C^4([t_0,T])$ and $F\in C^2(\Omega)$.
    Suppose further that
    \begin{equation}
        \label{eq:G4_nondegeneracy_along_solution}
        F\bigl(t,u(t)\bigr)
        \mathcal{L}_FF\bigl(t,u(t)\bigr)
        \neq
        0
    \end{equation}
    and
    \begin{equation}
        \label{eq:G4_boundary_equality}
        3\left(
            \mathcal{L}_FF\bigl(t,u(t)\bigr)
        \right)^2
        =
        F\bigl(t,u(t)\bigr)
        \mathcal{L}_F^2F\bigl(t,u(t)\bigr)
    \end{equation}
    for every $t\in[t_0, T]$.
    Suppose that there exist constants $C>0$ and $\varepsilon>0$ such that, for every $h\in(0,\varepsilon]$, a sequence of positive scales $\{\sigma_n\}_{n=0}^{N_h-1}$ is chosen and satisfies
    \begin{equation}
        \label{ineq:G4_vanishing_scale_rate}
        \sigma_n\leq C \sqrt{h},
        \qquad
        n=0,1,\ldots,N_h-1.
    \end{equation}
    Then there exist constants $H>0$ and $C_g>0$ such that, for every $h\in(0,H]$, there exists a unique sequence $\{u_n\}_{n=0}^{N_h}$ satisfying \eqref{mth:SE} with $\{(t_n,u_n)\}_{n=0}^{N_h}\subseteq U_\rho(u)$, and the estimate
    \begin{displaymath}
        \max_{0\leq n\leq N_h}|u(t_n)-u_n|
        \leq
        C_gh^3.
    \end{displaymath}
\end{theorem}

\begin{proof}
    By \eqref{eq:G4_nondegeneracy_along_solution} and compactness, there exists $\eta>0$ such that
    \begin{displaymath}
        u'(t)^2 = F\bigl(t,u(t)\bigr)^2 \geq \eta
    \end{displaymath}
    for all $t \in [t_0, T]$.
    Consequently, the chosen scales $\{\sigma_n\}_{n=0}^{N_h - 1}$ satisfy hypothesis~\ref{H2}.
    By \eqref{def:normalized_curvature_defect} and \eqref{eq:G4_boundary_equality}, we have 
    \begin{displaymath}
        \left| \mathcal{D}_{\sigma_n}\bigl(t_n,u(t_n)\bigr) \right|
        =
        \left| \frac{\sigma_n^2u'''(t_n)}{\sigma_n^2+u'(t_n)^2} \right|
        \leq
        \frac{M_3 C^2}{\eta}h.
    \end{displaymath}
    It follows from \cref{thm:lte}~\ref{thm:lte-2} and \eqref{eq:curvature_defect_relation-2} that there exist constants $C_\ell>0$ and $h_\ell\in(0,\varepsilon]$ such that $|\ell_{n+1}| \leq C_\ell h^3$ whenever $0<h\leq h_\ell$ and $n=0,1,\ldots,N_h-1$.
    Applying \cref{thm:convergence_of_SE} with $p=3$ completes the proof.
\end{proof}

The following example shows that scales vanishing with the step size can yield not only third-order convergence but also higher convergence orders.

\begin{example} \label{ex:autonomous_small_sigma}
    Consider the equation \eqref{ODE:IVP} with $t_0 = 0$, $u_0 > 0$, and
    \begin{equation} \label{ODE:autonomous_small_sigma}
        F(t, u) := u^{-1}, \qquad u> 0.
    \end{equation}
    The exact solution is $u(t)= (u_0^2+2t)^{\frac{1}{2}}$ for $t \geq 0$.
    Then, 
    \begin{displaymath}
        F \mathcal{L}_F F
        =
        -\frac{1}{u^4}
        \neq0
        \qquad\text{and}\qquad
        3\left(\mathcal{L}_F F\right)^2
        =
        F\mathcal{L}_F^2 F
        =
        \frac{3}{u^6}
        >
        0.
    \end{displaymath}
    Thus, case~\ref{G4} applies at every $u>0$.  Hence, the defect has no minimizing positive scale, and its infimum is approached as
    $\sigma \searrow 0$.
    Indeed,
    \begin{displaymath}
        \inf_{\sigma > 0}
        \left|\mathcal{D}_\sigma(u)\right|
        =
        \inf_{\sigma > 0}
        \left| \frac{3\sigma^2}{u^3\left(1+\sigma^2u^2\right)}\right|
        =
        \lim_{\sigma \searrow 0} \frac{3\sigma^2}{u^3\left(1+\sigma^2u^2\right)}
        =
        0.
    \end{displaymath}
    For $\sigma_n=h^p$ with $p\geq\frac{1}{2}$, \cref{thm:G4_vanishing_scale_SE} already gives third-order convergence.  The direct argument below uses the special structure of this equation to prove the stronger order $2p+2$ for every $p>0$.
        
    Fix $T > 0$ and $p >0$.
    For every $N \in \mathbb{N}$, set $h := \frac{T}{N}$ and $\sigma_n := h^p$ for $n = 0, 1, \ldots, N-1$.
    Then SE with $\left\{\sigma_n\right\}_{n=0}^{N-1}$ satisfies the global error estimate
    \begin{equation}    \label{ineq:autonomous_small_sigma-1}
        \max_{0\leq n\leq N} | u(t_n)-u_n |
        \leq
        \frac{T}{4u_0^3}h^{2p+2}.
    \end{equation}
    Consequently, this family of SE approximations has convergence order~$2p+2$.

    We prove \eqref{ineq:autonomous_small_sigma-1}.
    Suppose that $u_n>0$.
    The function
    \begin{displaymath}
        x
        \mapsto
        x-u_n
        -
        h \, \mathcal{C}_{\sigma_n}
        \left(
            \frac{1}{x},
            \frac{1}{u_n}
        \right)
    \end{displaymath}
    is strictly increasing on $(0,\infty)$, is negative at $x=u_n$, and
    tends to infinity as $x\to\infty$.
    Hence, \eqref{eq:SE_implicit_step} uniquely determines $u_{n+1}>u_n$.
    Starting from $u_0>0$, induction shows that $u_n\geq u_0$ for every $n=0,1,\ldots,N$.
    Moreover, \cite[Lem.~A.2~(d)]{2026a_Jung} and the SE relation give
    \begin{equation} \label{ineq:autonomous_small_sigma-4}
        0
        <
        u_{n+1}-u_n
        =
        h \, \mathcal{C}_{\sigma_n} \left( \frac{1}{u_{n+1}}, \frac{1}{u_n} \right)
        \leq
        \frac{h}{u_n}
        \leq
        \frac{h}{u_0}.
    \end{equation}
    Since $u_n\geq u_0$ and $u(t_n) = \sqrt{u_0^2+2t_n} \geq u_0$, we have $u_n + u(t_n) \geq 2u_0$.
    Consequently,
    \begin{equation}
        \label{ineq:autonomous_small_sigma-2}
        |u(t_n)-u_n|
        =
        \frac{\left|u(t_n)^2-u_n^2\right|}{u(t_n)+u_n}
        \leq
        \frac{\left|u(t_n)^2-u_n^2\right|}{2u_0}.
    \end{equation}
    It therefore suffices to estimate the difference between $u(t_n)^2$ and $u_n^2$.
    The exact solution gives 
    \begin{equation}    \label{ineq:autonomous_small_sigma-3}
        u(t_{n+1})^2-u(t_n)^2 = 2h.
    \end{equation}
    Writing $A_n := \sqrt{1 + \sigma_n^2 u_{n+1}^2}$ and $B_n := \sqrt{1 + \sigma_n^2 u_n^2}$, we find that 
    \begin{align*}
        u_{n+1}^2 - u_n^2 - 2h
        &= (u_{n+1} + u_n) \left[ (u_{n+1} - u_n) - \frac{2h}{u_{n+1} + u_n} \right]    \\
        &= h (u_{n+1} + u_n) \left[ \frac{ A_n + B_n }{ u_n A_n + u_{n+1} B_n} - \frac{2}{u_{n+1} + u_n} \right] \\
        &= h (u_{n+1} + u_n) \, \frac{\sigma_n^2 ( u_{n+1} - u_n)^2}{(u_n A_n + u_{n+1} B_n)(A_n + B_n)}.
    \end{align*}
    Since $A_n, B_n \geq 1$ and $\sigma_n = h^p$, the above equality implies 
    \begin{equation}     \label{ineq:autonomous_small_sigma-5}
        0
        \leq
        u_{n+1}^2 - u_n^2 - 2h
        \leq
        \frac{\sigma_n^2 ( u_{n+1} - u_n)^2}{2} h
        = \frac{( u_{n+1} - u_n)^2}{2} h^{2p + 1}
    \end{equation}
    Combining \eqref{ineq:autonomous_small_sigma-3}, \eqref{ineq:autonomous_small_sigma-4}, and \eqref{ineq:autonomous_small_sigma-5} yields
    \begin{displaymath}
        0 \leq \left( u_{n+1}^2 - u(t_{n+1})^2 \right) - \left( u_n^2 - u(t_n)^2 \right) \leq \frac{1}{2u_0^2} h^{2p+3}
    \end{displaymath}
    for each $n$.
    Summing over $j=0,1,\ldots,n-1$ and using $u_0 = u(t_0)$ yields
    \begin{equation} \label{ineq:autonomous_small_sigma-6}
        0
        \leq
        u_n^2-u(t_n)^2
        \leq
        \frac{T}{2u_0^2}h^{2p+2}.
    \end{equation}
    Combining \eqref{ineq:autonomous_small_sigma-2} and \eqref{ineq:autonomous_small_sigma-6} implies \eqref{ineq:autonomous_small_sigma-1}.
\end{example}

\section{Choice of the scale for fourth-order accuracy}
\label{sec:scale_fourth_order}

\Cref{cor:consistency_of_SE-4th} of \cref{cor:consistency_of_SE} shows that \eqref{eq:sol_graph_4th} is sufficient for fourth-order consistency.
By \eqref{eq:curvature_defect_relation-2}, this condition is equivalent, along the exact solution, to
\begin{equation}    \label{eq:field_4th}
    \mathcal{D}_{\sigma_n}\bigl(t_n, u(t_n)\bigr) + \mathcal{D}_{\sigma_n}\bigl(t_{n+1}, u(t_{n+1})\bigr)
    =
    0.
\end{equation}
The numerical counterpart of the third-order condition \eqref{eq:field_3rd} depends only on the known pair $(t_n, u_n)$.
Hence, the scale can be chosen before solving for $u_{n+1}$.
In contrast, the numerical counterpart of \eqref{eq:field_4th} also involves the unknown value at $t_{n+1}$.
After replacing the exact solution values by numerical ones, $v$ and $\sigma$ must therefore be determined simultaneously.
This leads to the following system.
\begin{definition}
    \label{def:defect_balanced_SE}
    Given $(t_n, u_n)\in\Omega$ and $h>0$, find $v\in\mathbb{R}$ satisfying $(t_{n+1},v)\in\Omega$ and a scale $\sigma>0$ such that
    \begin{equation} \label{ODS:SE_4th_order}
        \begin{cases}
            \displaystyle
            v
            =
            u_n + h \, \mathcal{C}_\sigma \left( F(t_{n+1},v), F(t_n,u_n) \right),
            \\[1.2ex]
            \mathcal{D}_\sigma(t_n,u_n) + \mathcal{D}_\sigma(t_{n+1},v)
            =
            0.
        \end{cases}
    \end{equation}
    For any solution pair $(v,\sigma)$, set $u_{n+1}:=v$ and $\sigma_n:=\sigma$.
\end{definition}

For each fixed trial value $v$ at $t_{n+1}$, the second equation in \eqref{ODS:SE_4th_order} is a scalar equation for $\sigma$.
This equation need not admit a positive solution.
We therefore relax exact defect balance by minimizing its absolute residual.

More generally, fix points $(t,x), (s,y)$ in the domain of $F$ and consider the one-dimensional optimization problem
\begin{equation}    \label{opt:field_4th}
    \inf_{\sigma > 0} \left| \mathcal{D}_\sigma(t, x) + \mathcal{D}_\sigma(s, y) \right|.
\end{equation} 

\begin{definition}
    A minimizer $\sigma>0$ of the problem \eqref{opt:field_4th} is said to be
    \emph{optimal} if
    \begin{displaymath}
        \mathcal{D}_\sigma(t,x)
        +
        \mathcal{D}_\sigma(s,y)
        =
        0.
    \end{displaymath}
\end{definition}

For each trial value $v$ in the implicit solve of \eqref{ODS:SE_4th_order}, we evaluate the minimization problem \eqref{opt:field_4th} with $(t,x) = (t_n, u_n)$ and $(s, y) = (t_{n+1}, v)$.
If this problem has a positive minimizer, a chosen minimizer is inserted into the SE update.
If the chosen minimizer gives zero residual, it also satisfies the defect-balance equation exactly; otherwise, it provides a relaxed approximation of that equation.

\subsection{Classification based on scaled field quantities}

The following theorem classifies the positive defect-cancelling scales of \eqref{opt:field_4th} and determines whether the residual infimum is zero or positive in the remaining cases.

\begin{theorem}
    \label{thm:sigma_selection_4th_order}
    Let $\Omega\subset\mathbb{R}^2$ be open and let $F\in C^2(\Omega)$.
    Fix $(t,x),(s,y)\in\Omega$.
    Define
    \begin{align*}
        a&:=\mathcal{L}_F^2F(t,x),
        &
        b&:=3F(t,x)\bigl(\mathcal{L}_F F(t,x)\bigr)^2,
        &
        c&:=F(t,x)^2,\\
        d&:=\mathcal{L}_F^2F(s,y),
        &
        e&:=3F(s,y)\bigl(\mathcal{L}_F F(s,y)\bigr)^2,
        &
        f&:=F(s,y)^2,
    \end{align*}
    and
    \begin{align*}
        A&:= a+d,
        &
        B&:= (a+d)(c+f)-b-e,   \\
        C&:= (a+d)cf-bf-ec,
        &
        D&:= B^2 - 4AC.
    \end{align*}
    The boundary residuals are
    \begin{align*}
        R_0
        &:=
        \lim_{\sigma\searrow0} \left| \mathcal{D}_\sigma(t,x)+\mathcal{D}_\sigma(s,y) \right|, \\
        R_{\infty}
        &:= 
        \lim_{\sigma\to \infty} \left| \mathcal{D}_\sigma(t,x)+\mathcal{D}_\sigma(s,y) \right|
        =
        |A|.
    \end{align*}
    Exactly one of the following six cases occurs.
    \begin{enumerate}[label=\upshape(E\arabic*),ref=(E\arabic*),itemsep=1ex]
        \item \label{E1}
        Suppose that $A = B = C = 0$.
        Then every $\sigma>0$ is an optimal minimizer of the problem \eqref{opt:field_4th}.

        \item \label{E2}
        Suppose that $A = 0$ and $BC < 0$.
        Then the problem \eqref{opt:field_4th} has the unique optimal minimizer
        \begin{equation} \label{eq:defect_balance_linear_scale}
            \sigma^{\ast}
            =
            \left( -\frac{bf + ec}{b+e} \right)^{\frac{1}{2}}.
        \end{equation}

        \item \label{E3}
        Suppose that $A = 0$, $BC \geq 0$, and $(B, C) \neq (0, 0)$.
        Then the problem \eqref{opt:field_4th} has no positive minimizer.
        More precisely, exactly one of the following cases occurs.
        \begin{enumerate}[label=\upshape(E3\alph*), ref=(E3\alph*), itemsep=1ex]
            \item \label{E3a}
            If $cf > 0$ and $C=0$, then 
            \begin{displaymath}
                \inf_{\sigma>0} \left| \mathcal{D}_\sigma(t,x) + \mathcal{D}_\sigma(s,y) \right|
                =
                R_{\infty}
                =
                R_0
                =
                0.
            \end{displaymath}

            \item \label{E3b}
            Suppose that one of the following conditions holds:
            \begin{enumerate}[label=\upshape(E3b\arabic*),ref=(E3b\arabic*),itemsep=1ex]
                \item \label{E3b1} $cf > 0$ and $C \neq 0$;
                \item \label{E3b2} $cf = 0$.
            \end{enumerate}
            Then
            \begin{displaymath}
                \inf_{\sigma>0} \left| \mathcal{D}_\sigma(t, x) + \mathcal{D}_\sigma(s,y) \right|
                =
                R_{\infty}
                =
                0
                <
                R_0.
            \end{displaymath}
        \end{enumerate}
        
        \item \label{E4}
        Suppose that $A \neq 0$, $D > 0$, $AB < 0$, and $AC > 0$.
        Then the problem \eqref{opt:field_4th} has exactly two optimal minimizers,
        \begin{equation} \label{eq:defect_balance_two_cancelling_scales}
            \sigma_{\pm}^{\ast}
            =
            \left( \frac{-B\pm\sqrt{D}}{2A} \right)^{\frac{1}{2}}.
        \end{equation}

        \item \label{E5}
        Suppose that $A \neq 0$ and one of the following conditions holds:
        \begin{enumerate}[label=\upshape(E5\alph*),ref=(E5\alph*),itemsep=1ex]
            \item \label{E5a} $D > 0$, $AB < 0$, and $AC = 0$;
            \item \label{E5b} $D > 0$ and $AC < 0$;
            \item \label{E5c} $D = 0$, $AB < 0$.
        \end{enumerate}
        Then the problem \eqref{opt:field_4th} has the unique optimal minimizer 
        \begin{equation}
            \label{eq:defect_balance_unique_cancelling_scale}
            \sigma^{\ast}
            =
            \left( \frac{-B + \sgn(A) \sqrt{D}}{2A} \right)^{\frac{1}{2}}.
        \end{equation}

        \item \label{E6}
        Suppose that $A \neq 0$ and one of the following conditions holds:
        \begin{enumerate}[label=\upshape(\roman*)]
            \item $D > 0$, $AB > 0$, and $AC \geq 0$;
            \item $D = 0$ and $AB \geq 0$;
            \item $D < 0$.
        \end{enumerate}
        Then no positive scale makes the defect residual vanish.
        Consequently, the problem \eqref{opt:field_4th} has no optimal minimizer.
        More precisely, exactly one of the following cases occurs.

        \begin{enumerate}[label=\upshape(E6\alph*), ref=(E6\alph*), itemsep=1ex]
            \item \label{E6a}
            If $B = A(c+f)$ and $C = Acf$, then every positive scale is a minimizer of \eqref{opt:field_4th}, but none is optimal.

            \item \label{E6b}
            Suppose that one of the following conditions holds:
            \begin{enumerate}[label=\upshape(\roman*)]
                \item $cf>0$ and $C=0$;
                \item $cf=0$, $c+f>0$, and $B=0$.
            \end{enumerate}
            Then the problem \eqref{opt:field_4th} has no positive minimizer, and
            \begin{displaymath}
                \inf_{\sigma>0} \left| \mathcal{D}_\sigma(t,x) + \mathcal{D}_\sigma(s,y) \right|
                =
                R_0
                =
                0.
            \end{displaymath}
            
            \item \label{E6c}
            If neither case~\ref{E6a} nor case~\ref{E6b} applies, then
            \begin{displaymath}                
                \inf_{\sigma>0} \left| \mathcal{D}_\sigma(t,x) + \mathcal{D}_\sigma(s,y) \right|
                >
                0 
                \qquad\text{and}\qquad
                R_0 > 0.
            \end{displaymath}
        \end{enumerate}
    \end{enumerate}
\end{theorem}

\begin{proof}
    Define the function $\varphi : (0, \infty) \to \mathbb{R}$ by 
    \begin{displaymath}
        \varphi(\sigma)
        :=
        \mathcal{D}_\sigma(t,x)+\mathcal{D}_\sigma(s,y)
        =
        A-\frac{b}{\sigma^2+c}-\frac{e}{\sigma^2+f}.
    \end{displaymath}
    Observe that the positive zeros of $\varphi$ correspond to the positive solutions of the equation%
    \begin{equation}    \label{eq:quadratic}
        A \sigma^4 + B \sigma^2 + C = 0.
    \end{equation}

    Suppose first that $A = 0$.
    If $B = C = 0$, then $b + e = bf + ec = 0$.
    Therefore, every $\sigma>0$ is a solution of \eqref{eq:quadratic}, proving case~\ref{E1}.

    If $BC < 0$, then $B \neq 0$, and hence the equation \eqref{eq:quadratic} has the unique positive solution
    \begin{displaymath}
        \sigma = \left( -\frac{C}{B} \right)^{\frac{1}{2}} = \left( -\frac{bf + ec}{b + e} \right)^{\frac{1}{2}}.
    \end{displaymath}
    This proves case~\ref{E2}.

    To show case~\ref{E3}, assume that $BC \geq 0$ and $(B, C) \neq (0, 0)$.
    Then $B\sigma^2+C \neq 0$ for every $\sigma > 0$.
    Thus, \eqref{eq:quadratic} has no positive solution, and hence $|\varphi(\sigma)| > 0$ for every $\sigma>0$.
    Letting $\sigma\to\infty$ gives
    \begin{displaymath}
        0
        \leq
        \inf_{\sigma>0}|\varphi(\sigma)|
        \leq
        \lim_{\sigma\to\infty}|\varphi(\sigma)|
        =
        |A|
        =
        0,
    \end{displaymath}
    and hence
    \begin{displaymath}
        \inf_{\sigma>0}|\varphi(\sigma)| = R_{\infty} = 0.
    \end{displaymath}
    It remains to determine $R_0$.
    If $cf>0$, then
    \begin{displaymath}
        R_0
        =
        \lim_{\sigma \searrow 0} \left\vert -\frac{b}{\sigma^2 + c} - \frac{e}{\sigma^2 + f} \right\vert 
        =
        \left| -\frac{b}{c} - \frac{e}{f} \right|
        =
        \frac{|bf+ec|}{cf}
        =
        \frac{|C|}{cf}.
    \end{displaymath}
    Therefore, $C=0$ gives case~\ref{E3a}, while $C\neq0$
    gives case~\ref{E3b1}.

    Suppose that $cf=0$.
    The case $c=f=0$ would imply $b=e=0$ and hence $B=C=0$, which is impossible.
    Thus, $c+f>0$.
    Moreover, $c=0$ implies $b=0$, while $f=0$ implies $e=0$.
    Consequently, $C=0$ and $B\neq0$.
    Since $A = C = cf = 0$, we have
    \begin{displaymath}
        R_0
        =
        \lim_{\sigma \searrow 0} \left\vert \frac{B\sigma^2}{(\sigma^2+c)(\sigma^2+f)} \right\vert 
        =
        \lim_{\sigma \searrow 0} \left\vert \frac{B}{\sigma^2+c+f} \right\vert 
        =
        \frac{|B|}{c+f}
        >
        0.
    \end{displaymath}
    This gives case~\ref{E3b2}.
    
    Assume now that $A \neq 0$.
    We first determine the number of optimal minimizers of \eqref{opt:field_4th} from the positive solutions of \eqref{eq:quadratic}.
    Regarding \eqref{eq:quadratic} as a quadratic equation in $z := \sigma^2$, we consider the three cases: $D > 0$, $D = 0$, and $D < 0$.

    Suppose first that $D > 0$, and hence
    \begin{displaymath}
        z_{\pm} = \frac{-B \pm \sqrt{D}}{2A}, 
        \qquad
        z_+ + z_- = - \frac{B}{A},
        \qquad
        z_+ z_- = \frac{C}{A}.
    \end{displaymath}
    
    If $AB < 0$, then $z_+ + z_- > 0$.
    We now consider the three cases determined by the sign of $AC$.
    First, if $AC>0$, then $z_+z_->0$.
    Therefore, $z_+,z_->0$, and $\sigma_{\pm} = \sqrt{z_{\pm}} > 0$ are two positive solutions of \eqref{eq:quadratic}.
    This gives case~\ref{E4}.
    Second, if $AC = 0$, then $C = 0$ and $D = B^2$.
    Since $AB<0$, we have $\sgn(A)|B|=-B$.
    The quadratic equation has a unique positive root, say $z_{\ast}$, in $z$, given by
    \begin{displaymath}
        z^{\ast} = - \frac{B}{A} = \frac{-B - B}{2A} = \frac{-B + \sgn(A) |B|}{2A} = \frac{-B + \sgn(A) \sqrt{D}}{2A} > 0.
    \end{displaymath}
    Thus, $\sigma^\ast:=\sqrt{z^\ast}$ is the unique positive solution of \eqref{eq:quadratic}, proving case~\ref{E5a}.
    Third, if $AC<0$, then $z_+z_- = \frac{C}{A} < 0$.
    Hence, exactly one of $z_+$ and $z_-$ is positive.
    Since $AC<0$, $D > B^2$, which yields $-B+\sqrt D > 0$ and $-B-\sqrt D < 0$.
    Therefore, $z_+>0>z_-$ when $A>0$, while $z_->0>z_+$ when $A<0$.
    In either case, the unique positive root is
    \begin{displaymath}
        z^\ast
        =
        \frac{-B+\sgn(A)\sqrt D}{2A}.
    \end{displaymath}
    Thus, $\sigma^\ast:=\sqrt{z^\ast}$ is the unique positive solution of \eqref{eq:quadratic}, proving case~\ref{E5b}.

    If $AB = 0$, then $B = 0$, and the identity $D = -4AC > 0$ implies that $AC < 0$.
    Hence, this case is included in case~\ref{E5b}.
    In fact, in this case, the quadratic equation in $z$ has the unique positive root
    \begin{displaymath}
        z^\ast
        =
        \frac{\sqrt{-AC}}{|A|}
        =
        \sqrt{-\frac CA}
        =
        \frac{\sgn(A)\sqrt D}{2A}
        >
        0.
    \end{displaymath}
    Thus, $\sigma^\ast:=\sqrt{z^\ast}$ is the unique positive solution of \eqref{eq:quadratic}.

    Suppose next that $AB > 0$.
    If $AC < 0$, then $D > 0$, and hence this case is included in case~\ref{E5b}.
    Hence, it remains to consider $AC \geq 0$.
    Since $AB > 0$, $z_+ + z_- < 0$.
    If $AC > 0$, then $z_+z_- > 0$, and therefore $z_+,z_- < 0$.
    If $AC = 0$, then one root is zero and the other is $-\frac{B}{A}<0$.
    Thus, neither case admits a positive root in $z$.
    This gives case~\ref{E6}\textnormal{(i)}.

    Now, suppose that $D=0$.
    Then the quadratic equation in $z$ has the unique root $z = -\frac{B}{2A}$.
    If $AB < 0$, then $z > 0$, and $\sigma^\ast := \sqrt z$ is the unique positive solution of \eqref{eq:quadratic}.
    This gives case~\ref{E5c}.
    If $AB \geq 0$, then $z \leq 0$, and \eqref{eq:quadratic} has no positive solution.
    This gives case~\ref{E6}\textnormal{(ii)}.

    Finally, if $D<0$, the quadratic equation in $z$ has no real root.
    Hence, \eqref{eq:quadratic} has no positive solution.
    This gives case~\ref{E6}\textnormal{(iii)}.

   We now distinguish cases~\ref{E6a}, \ref{E6b}, and~\ref{E6c}.
    Since $\varphi$ has no zero on $(0,\infty)$, it has a fixed sign on this interval.
    
    We first consider the case in which $\varphi$ is constant.
    Since $\varphi(\sigma) \to A$ as $\sigma \to \infty$, comparison of the coefficients shows that $\varphi$ is constant precisely when $B = A(c+f)$ and $ C = Acf$.
    In this case, $\varphi(\sigma) = A$ for every $\sigma > 0$.
    Since $A \neq 0$,
    \begin{displaymath}
        \min_{\sigma>0} \left| \varphi(\sigma) \right|
        =
        |A|
        >
        0.
    \end{displaymath}
    Therefore, every positive scale is a minimizer, but none is optimal.
    This proves case~\ref{E6a}.

    It remains to consider the case in which $\varphi$ is nonconstant.
    Note that 
    \begin{displaymath}
        0
        \leq
        \inf_{\sigma > 0} |\varphi(\sigma)|
        \leq
        \lim_{\sigma \searrow 0} |\varphi(\sigma)|
        =
        R_0.
    \end{displaymath}
    It remains to determine when $R_0 = 0$.
    First, suppose that $cf>0$.
    Then
    \begin{displaymath}
        R_0
        =
        \left| A-\frac{b}{c}-\frac{e}{f} \right|
        =
        \frac{|Acf-bf-ec|}{cf}
        =
        \frac{|C|}{cf}.
    \end{displaymath}
    Hence, $R_0 = 0$ precisely when $C = 0$.
    Second, suppose that $cf = 0$ and $c + f > 0$.
    Then either $c = 0$ or $f = 0$.
    Since $c = 0$ implies $b = 0$ and $f = 0$ implies $e = 0$, we have $C = 0$.
    Therefore,
    \begin{displaymath}
        R_0
        =
        \lim_{\sigma \searrow 0} \left| \frac{A\sigma^4 + B\sigma^2}{(\sigma^2+c)(\sigma^2+f)} \right|
        =
        \lim_{\sigma \searrow 0} \left| \frac{A\sigma^2 + B} {\sigma^2+c+f} \right|
        =
        \frac{|B|}{c + f}.
    \end{displaymath}
    Hence, $R_0 = 0$ precisely when $B  = 0$.
    Thus, the two alternatives in case~\ref{E6b} are exactly the cases in which $R_0 = 0$.
    The preceding inequality gives the stated value of the infimum.
    Since $\varphi$ has no zero on $(0,\infty)$, this infimum is not attained at any positive scale.
    This proves case~\ref{E6b}.

    Suppose finally that neither case~\ref{E6a} nor case~\ref{E6b} applies.
    Then the preceding formulas give $R_0>0$.
    Moreover, $R_\infty>0$ since $A\neq0$.
    Suppose, to the contrary, that
    \begin{displaymath}
        \inf_{\sigma > 0} |\varphi(\sigma)|
        =
        0.
    \end{displaymath}
    Then there exists a sequence $\{\sigma_k\}_{k=1}^{\infty}$ of positive scales such that $|\varphi(\sigma_k)|\to0$ as $k\to\infty$.
    The positivity of $R_0$ and $R_\infty$ implies that, after discarding finitely many terms, the sequence is bounded above and bounded away from zero.
    Passing to a subsequence, we may assume that $\sigma_k\to\sigma^{\ast}$ as $k\to\infty$ for some $\sigma^{\ast}>0$.
    By the continuity of $\varphi$, $\varphi(\sigma^\ast)=0$.
    Since $\sigma^\ast>0$, this contradicts the fact established above that $\varphi$ has no zero on $(0,\infty)$.
    Therefore,
    \begin{displaymath}
        \inf_{\sigma>0} \left| \varphi(\sigma) \right|
        >
        0.
    \end{displaymath}
    This proves case~\ref{E6c} and completes the proof.
\end{proof}

In cases~\ref{E1}, \ref{E6a}, and~\ref{E6c}, we use an arbitrary fixed positive scale.
Under the assumptions of \cref{cor:convergence_orders_SE}, this choice gives second-order convergence.
In case~\ref{E3}, the residual infimum is approached as $\sigma\to\infty$.
By \eqref{eq:C_infty}, the limiting method is CN.
We therefore do not consider these cases further. 

As in the third-order analysis, we focus on nontrivial positive optimal scales and scales vanishing with the step size.
Nontrivial positive optimal scales are obtained in cases~\ref{E2}, \ref{E4}, and~\ref{E5}.
After setting aside case~\ref{E3} as above, case~\ref{E6b} is the only remaining case without a positive optimal scale in which the small-scale residual vanishes.
Before considering the main cases, we identify how the third- and fourth-order classifications match.

\begin{remark}  \label{rmk:3rd_4th_order_classification_matchings}
    The fourth- and third-order classifications can be matched by setting $(s,y)=(t,x)$, which gives $A = 2a$, $B = 2(2ac - b)$, $C = 2c(ac - b)$, and $D = 4b^2$.
    Using these identities, we obtain the following matching.
    Case~\ref{G1} gives case~\ref{E1} when $a=0$ and case~\ref{E6a} when $a \neq 0$.
    Case~\ref{G2} gives case~\ref{E3b1} when $F\mathcal{L}_F^2F=0$ and case~\ref{E6c} when $F\mathcal{L}_F^2F < 0$.
    Case~\ref{G3} gives case~\ref{E5b}.
    Case~\ref{G4} gives case~\ref{E6b}\textnormal{(i)} in the equality case and case~\ref{E6c} in the strict inequality case.
    In particular, cases~\ref{E2}, \ref{E3a}, \ref{E3b2}, \ref{E4}, \ref{E5a}, \ref{E5c}, and~\ref{E6b}\textnormal{(ii)} do not occur when $(s,y) = (t,x)$.
    
    We prove only the matching between case~\ref{G3} and case~\ref{E5b}.
    Suppose that case~\ref{G3} applies.
    Then $A \neq 0$, $D > 0$, and
    \begin{displaymath}
        AC
        =
        4F^2 \bigl(F\mathcal{L}_F^2 F\bigr) \left[ F\mathcal{L}_F^2 F - 3(\mathcal{L}_F F)^2 \right]
        <
        0.
    \end{displaymath}
    Thus, case~\ref{E5b} applies.
    The converse follows from the same calculation.

    Moreover, $A$ and $b$ have the same sign, and hence $\sgn(A)\sqrt D=2b$.
    Therefore,
    \begin{displaymath}
        (\sigma^\ast)^2
        =
        \frac{-B + \sgn(A)\sqrt D}{2A}
        =
        \frac{b}{a}-c
        =
        \frac{3F(\mathcal{L}_F F)^2}{\mathcal{L}_F^2 F} - F^2.
    \end{displaymath}
    Hence, the minimizing scales in cases~\ref{G3} and~\ref{E5b} agree.
    The remaining matchings follow from similar calculations.
\end{remark}

\subsection{Optimal scale}

We now focus on cases~\ref{E2}, \ref{E4}, and~\ref{E5} in \cref{thm:sigma_selection_4th_order}.
Among these cases, we establish convergence only for cases~\ref{E5a} and~\ref{E5b}.

In case~\ref{E2}, $A=0$, and hence $B = -(b + e)$ and $C=-(bf + ec)$.
If $F$ has the same nonzero sign at the two points, then $b$ and $e$ have the same sign.
Hence, $BC \geq 0$, which contradicts $BC < 0$.
Thus, case~\ref{E2} can occur only on a step across which $F$ changes sign.

In case~\ref{E4}, the sum of the two defects has two positive zeros.
By Rolle's theorem, there exists $\sigma>0$ such that
\begin{displaymath}
    \frac{b}{(\sigma^2+c)^2}
    +
    \frac{e}{(\sigma^2+f)^2}
    =
    0.
\end{displaymath}
Hence, $be<0$.
Therefore, case~\ref{E4} can also occur only on a step across which $F$ changes sign.
Consequently, neither case can hold for all sufficiently close pairs in a neighborhood of the exact solution.

Moreover, the condition $A=0$ in case~\ref{E2} need not be preserved when the second point varies during the implicit solve, while case~\ref{E4} gives two positive optimal scales and requires an additional branch selection.
The convergence argument used in this paper does not apply to these cases, and establishing convergence would require different assumptions.
We therefore do not establish separate convergence results for these two cases.

In case~\ref{E5c}, $D=0$, and hence the two roots in $\sigma^2$ coincide at the same positive value.
Consequently, the scale selection need not be locally Lipschitz, and a small perturbation may produce either two positive roots or no positive root.
Therefore, the convergence argument used for cases~\ref{E5a} and~\ref{E5b} does not apply to this case.

We now establish convergence in cases~\ref{E5a} and~\ref{E5b}.

\begin{theorem}[Fourth-order convergence in cases~\ref{E5a} and~\ref{E5b}]
    \label{thm:SE4}
    Suppose that hypothesis~\ref{H1} holds, with $u\in C^5([t_0,T])$ and $F\in C^3(\Omega)$.
    Suppose further that there exists $\varepsilon_0>0$ such that, whenever $(t,x),(s,y)\in U_\rho(u)$ and $0\leq s-t\leq \varepsilon_0$, either case~\ref{E5a} or case~\ref{E5b} of \cref{thm:sigma_selection_4th_order} applies.
    Define
    \begin{equation}
        \label{eq:E5_selected_scale}
        \Sigma(t, x; s, y)
        :=
        \left( \frac{-B+\sgn(A)\sqrt D}{2A} \right)^{\frac{1}{2}}
    \end{equation}
    for every such pair of points, where $A$, $B$, and $D$ are the corresponding coefficients in \cref{thm:sigma_selection_4th_order}.
    Consider SE with
    \begin{equation}
        \label{eq:SE4}
        \sigma_n
        :=
        \Sigma(t_n,u_n;t_{n+1},u_{n+1}),
        \qquad
        n=0,1,\ldots,N_h-1.
    \end{equation}
    Then there exist constants $H>0$ and $C_g>0$ such that, for every $h\in(0,H]$, there exists a unique sequence $\{u_n\}_{n=0}^{N_h}$ satisfying \eqref{mth:SE} with \eqref{eq:SE4} and $\{(t_n,u_n)\}_{n=0}^{N_h}\subseteq U_\rho(u)$, and this sequence satisfies
    \begin{equation}
        \label{ineq:SE4}
        \max_{0\leq n\leq N_h}|e_n|
        \leq
        C_gh^4.
    \end{equation}
\end{theorem}

\begin{proof}
    The set $W_{\rho,\varepsilon_0}(u)$ defined by \eqref{eq:W_rho_epsilon} is compact.
    By the assumption and $F\in C^3(\Omega)$, $\Sigma$ defined in \eqref{eq:E5_selected_scale} is well defined, strictly positive, and $C^1$ on an open neighborhood of this set.
    Compactness therefore yields constants $\underline{\sigma}>0$ and $L_\Sigma\geq0$ such that
    \begin{equation}
        \label{ineq:SE4-1}
        \underline{\sigma}
        \leq
        \Sigma(t,x;s,y)
        \qquad\text{and}\qquad
        |\Sigma(t,x;s,y)-\Sigma(t,\widetilde x;s,\widetilde y)|
        \leq
        L_\Sigma\bigl(|x-\widetilde x|+|y-\widetilde y|\bigr)
    \end{equation}
    whenever $(t, x),(t, \widetilde x), (s,y), (s,\widetilde y) \in U_\rho(u)$ and $0\leq s-t\leq \varepsilon_0$.
    Moreover, since the minimizer in each of cases~\ref{E5a} and~\ref{E5b} is optimal, 
    \begin{equation}
        \label{eq:SE4-2}
        \mathcal{D}_{\Sigma(t, x; s, y)}(t, x)
        +
        \mathcal{D}_{\Sigma(t, x; s, y)}(s, y)
        =
        0
    \end{equation}
    for all such pairs $(t, x)$ and $(s, y)$.

    Let $\{\ell_{n+1}^{\ast}\}_{n=0}^{N_h-1}$ denote the local truncation errors of SE with the exact scale sequence $\{\sigma_n^{\ast}\}_{n=0}^{N_h-1}$, where
    \begin{align*}
        \sigma_n^{\ast}
        &:=
        \Sigma\bigl(t_n,u(t_n);t_{n+1},u(t_{n+1})\bigr),\\
        \ell_{n+1}^{\ast}
        &:=
        \frac{u(t_{n+1})-u(t_n)}{h}
        -
        \mathcal{C}_{\sigma_n^{\ast}}
        \left(u'(t_{n+1}),u'(t_n)\right),
    \end{align*}
    for $n=0,1,\ldots,N_h-1$.
    Since $\sigma_n^\ast\geq\underline{\sigma}$, the exact scales $\{\sigma_n^\ast\}_{n=0}^{N_h-1}$ satisfy hypothesis~\ref{H2} with $H=\varepsilon_0$ and $\eta=\underline{\sigma}^2$.
    Therefore, by \cref{thm:lte}~\ref{thm:lte-3} and \eqref{eq:curvature_defect_relation-2}, there exist constants $C^{\ast}>0$ and $h^{\ast}>0$, independent of $h$ and $n$, such that
    \begin{equation}
        \label{ineq:SE4-3}
        |\ell_{n+1}^{\ast}|
        \leq
        \frac{h^2}{24} \left| \mathcal{D}_{\sigma_n^\ast}\bigl(t_n,u(t_n)\bigr) + \mathcal{D}_{\sigma_n^\ast}\bigl(t_{n+1},u(t_{n+1})\bigr) \right| + C^{\ast}h^4
        =
        C^{\ast}h^4
    \end{equation}
    whenever $0<h\leq\min\{h^\ast,\varepsilon_0\}$ and $n=0,1,\ldots,N_h-1$.
    The last equality follows from \eqref{eq:SE4-2}.
    
    Since $F$ is continuous on the compact set $U_{\rho}(u)$, there exists $M_F > 0$ such that $|F(t, x)| \leq M_F$ for every $(t,x)\in U_\rho(u)$.
    Define
    \begin{displaymath}
        Q
        :=
        \frac{M_1M_2^2L_\Sigma}{2\underline{\sigma}^3}
        \qquad\text{and}\qquad
        \Lambda
        :=
        2L+\frac{2M_F^3L_\Sigma}{\underline{\sigma}^3}.
    \end{displaymath}
    For every $n$, define $\Sigma_n(t,x;s,y) := \Sigma(t,x;s,y)$, and let $\mathcal{N}_n$ and $\mathcal{E}_n$ be the corresponding functions in \cref{lem:lte_to_global_error_SE}.

    Fix $h\in(0,\min\{h^\ast,\varepsilon_0\}]$ and $n\in\{0,1,\ldots,N_h-1\}$ for which $\{(t_j,u_j)\}_{j=0}^{n} \subseteq U_\rho(u)$.
    For every $(t_{n+1}, v), (t_{n+1}, w)\in U_{\rho}(u)$, \cref{lem:Lipschitz_C_sigma,lem:scale_sensitivity_C_sigma}, hypothesis~\ref{H1}, and \eqref{ineq:SE4-1} give
    \begin{align*}
        &\left| \, \mathcal{N}_n(t_n,u_n;t_{n+1},v) - \mathcal{N}_n(t_n,u_n;t_{n+1},w) \right| \\
        \leq &
        \, 2L|v-w|  + \frac{2M_F^3}{\underline{\sigma}^3} \left| \Sigma_n(t_n, u_n; t_{n+1}, v) - \Sigma_n(t_n, u_n; t_{n+1}, w) \right| \\
        \leq &
        \, \Lambda|v-w|.
    \end{align*}
    Thus, \eqref{ineq:lte_to_global_error_SE-0} is satisfied.
    
    For every $(t_{n+1},v)\in U_\rho(u)$, \eqref{ineq:SE4-1} gives
    \begin{equation}    \label{ineq:SE4-7}
        \left| \Sigma_n(t_n, u_n; t_{n+1}, v) - \sigma_n^\ast \right|
        \leq 
        L_\Sigma \left( |e_n|+|v-u(t_{n+1})| \right). 
    \end{equation}
    Moreover, $\Sigma_n(t_n,u_n;t_{n+1},v),\sigma_n^\ast\geq\underline{\sigma}$ by \eqref{ineq:SE4-1}.
    Applying \cref{lem:scale_sensitivity_along_solution} with $t=t_{n+1}$, $s=t_n$, $\sigma=\underline{\sigma}$, $\mu=\Sigma_n(t_n,u_n;t_{n+1},v)$, and $\nu=\sigma_n^\ast$, together with \eqref{ineq:SE4-7}, gives
    \begin{align}
        &\left| \mathcal{C}_{\Sigma_n(t_n,u_n;t_{n+1},v)} \left(u'(t_{n+1}),u'(t_n)\right) - \mathcal{C}_{\sigma_n^\ast} \left(u'(t_{n+1}),u'(t_n)\right) \right| \nonumber\\
        \leq &
        \, \frac{M_1M_2^2}{2\underline{\sigma}^3} h^2 \left| \Sigma_n(t_n, u_n; t_{n+1}, v) - \sigma_n^\ast \right| \nonumber\\
        \leq &    
        \, Q h^2 \left( |e_n| + |v-u(t_{n+1})| \right).
        \label{ineq:SE4-8}
    \end{align}
    By the definitions of $\ell_{n+1}^\ast$ and $\mathcal{E}_n$,
    \begin{align*}
        & \frac{u(t_{n+1})-u(t_n)}{h} - \mathcal{E}_n(t_n, u_n; t_{n+1}, v)
        \\
        =& 
        \, \ell_{n+1}^\ast 
        + \mathcal{C}_{\sigma_n^\ast} \left(u'(t_{n+1}),u'(t_n)\right) 
        - \mathcal{C}_{\Sigma_n(t_n,u_n;t_{n+1},v)} \left(u'(t_{n+1}), u'(t_n)\right).
    \end{align*}
    Combining the preceding identity, \eqref{ineq:SE4-3}, and \eqref{ineq:SE4-8} gives
    \begin{displaymath}
        \left| \frac{u(t_{n+1}) - u(t_n)}{h} - \mathcal{E}_n(t_n, u_n; t_{n+1}, v) \right|
        \leq
        C^\ast h^4 + Q h^2 \left( |e_n| + |v - u(t_{n+1})| \right).
    \end{displaymath}
    Thus, \eqref{ineq:lte_to_global_error_SE-1} is satisfied with $p = 4$, $P=C^\ast$, and $\varepsilon=\min\{h^\ast, \varepsilon_0\}$.
    
    Therefore, \cref{lem:lte_to_global_error_SE} gives a unique sequence satisfying \eqref{eq:lte_to_global_error_SE-5} and \eqref{ineq:SE4}.
    By the definitions of $\mathcal{N}_n$ and $\Sigma_n$, \eqref{eq:lte_to_global_error_SE-5} is \eqref{mth:SE} with \eqref{eq:SE4}.
    Moreover, \eqref{eq:SE4-2} shows that $(u_{n+1},\sigma_n)$ satisfies the second equation in \eqref{ODS:SE_4th_order}.
\end{proof}

We next illustrate the scale-selection procedure.

\begin{example}[Nonlinear pendulum equation] \label{ex:pendulum_defect_balance_scale}
    Consider the equation
    \begin{displaymath}
        \ddot q + \sin q=0,
    \end{displaymath}
    where $q:\mathbb{R}\to\mathbb{R}$ denotes the angular displacement and the dot denotes differentiation with respect to time.
    Fix a turning angle $\theta \in\left(0, \frac{\pi}{2} \right]$ and consider the corresponding periodic orbit with turning points $q=\pm\theta$.
    On its positive-velocity branch, introduce the normalized phase variables
    \begin{displaymath}
        q = \theta \xi 
        \qquad\text{and}\qquad
        \dot q = \theta u(\xi),
    \end{displaymath}
    where $\xi \in (-1, 1)$ and $u(\xi) > 0$.
    The corresponding normalized phase equation is
    \begin{equation}
        \label{ODE:normalized_pendulum_phase}
        F(\xi,u)
        =
        -\frac{\sin(\theta\xi)}{\theta u},
        \qquad
        (\xi,u)\in(-1,1)\times(0,\infty).
    \end{equation}
    The exact positive-velocity phase branch is given by
    \begin{displaymath}
        u(\xi)
        =
        \frac{\sqrt{2}}{\theta}
        \sqrt{\cos(\theta\xi)-\cos\theta},
        \qquad
        \xi\in(-1,1),
    \end{displaymath}
    where the initial value is chosen as $u_0 := u(0)$.
    Its restriction to $\xi\in[0,1)$ is the solution of \eqref{ODE:IVP} with $t=\xi$ and $t_0=0$.

    Although the original orbit collapses to the equilibrium $(0, 0)$, the normalized positive-velocity phase branch converges to the upper unit semicircle as $\theta \searrow 0$.
    Indeed, for every fixed $\xi \in(-1,1)$ and $v > 0$, 
    \begin{displaymath}
        \lim_{\theta \searrow 0} F(\xi, v) 
        =
        -\frac{\xi}{v}
        \qquad\text{and}\qquad
        \lim_{\theta \searrow 0} u(\xi)
        =
        \sqrt{1-\xi^2}.
    \end{displaymath}
    In other words, the normalized pendulum problem in \eqref{ODE:normalized_pendulum_phase} reduces to the unit-circle case in \cref{ex:SE_trace} as $\theta \searrow 0$.
    Thus, we expect SE to be particularly accurate for small $\theta$.

    We next determine which cases of \cref{thm:sigma_selection_4th_order} apply to pairs of exact points on the normalized positive-velocity phase branch.
    A direct calculation with
    $\mathcal{L}_F=\partial_\xi+F\partial_u$ gives
    \begin{equation}    \label{eq:normalized_pendulum-1}
        \mathcal{L}_F F = -\frac{\mathsf{P}}{\theta^2u^3}
        \qquad\text{and}\qquad
        \mathcal{L}_F^2 F = -\frac{\sin(\theta\xi)\mathsf{Q}}{\theta^3u^5},
    \end{equation}
    where
    \begin{align*}
        \mathsf{P}(\xi,u)
        &:=
        \theta^2\cos(\theta\xi)u^2 + \sin^2(\theta\xi),
        \\
        \mathsf{Q}(\xi,u)
        &:=
        -\theta^4u^4 + 3\theta^2 u^2 \cos(\theta\xi) + 3\sin^2(\theta\xi), \\
        \mathsf{R}(\xi,u)
        &:=
        \theta^2u^2 \left( 3\cos^2(\theta\xi) + \sin^2(\theta\xi) \right) + 3\cos(\theta\xi)\sin^2(\theta\xi).
    \end{align*}
    These quantities satisfy
    \begin{equation}
        \label{eq:pendulum_PQR_identity}
        3\mathsf{P}^2
        =
        \sin^2(\theta\xi)\mathsf{Q}
        +
        \theta^2u^2\mathsf{R}.
    \end{equation}

    If $u$ is a solution of \eqref{ODE:normalized_pendulum_phase}, then for $\xi \in [0, 1)$, we have
    \begin{displaymath}
        u(\xi) > 0, \qquad
        \mathsf{Q}\bigl(\xi,u(\xi)\bigr)
        \geq
        2(1-\cos\theta)(1+2\cos\theta)
        >
        0, \qquad
        \mathsf{R}\bigl(\xi,u(\xi)\bigr) > 0.
    \end{displaymath}

    Fix $\xi, \eta \in [0, 1)$  with $\xi < \eta$.
    Let $a, b, c, d, e, f, A, B, C, D$ be the corresponding coefficients in \cref{thm:sigma_selection_4th_order}.
    By \cref{eq:normalized_pendulum-1,eq:pendulum_PQR_identity},
    \begin{align*}
        a
        &=
        -\frac{\sin(\theta\xi)\mathsf{Q}(\xi, u(\xi))}{\theta^3 u(\xi)^5},
        &
        c
        &=
        \frac{\sin^2(\theta\xi)}{\theta^2 u(\xi)^2},
        &
        ac - b
        &=
        \frac{\sin(\theta\xi)\mathsf{R}(\xi, u(\xi))}
        {\theta^3 u(\xi)^5},
        \\
        d
        &=
        -\frac{\sin(\theta\eta)\mathsf{Q}(\eta, u(\eta))}{\theta^3 u(\eta)^5},
        &
        f
        &=
        \frac{\sin^2(\theta\eta)}{\theta^2 u(\eta)^2},
        &
        df-e
        &=
        \frac{\sin(\theta\eta)\mathsf{R}(\eta, u(\eta))}
        {\theta^3 u(\eta)^5}.
    \end{align*}
    If $\xi=0$, then $a=b=c=0$, and hence
    \begin{displaymath}
        A=d<0,
        \qquad
        B=df-e>0,
        \qquad
        C=0,
        \qquad
        D=B^2>0.
    \end{displaymath}
    Thus, case~\ref{E5a} applies.
    If $\xi>0$, then
    \begin{displaymath}
        A=a+d<0, \qquad
        C=f(ac-b)+c(df-e)>0, \qquad
        D > 0,
    \end{displaymath}
    and hence case~\ref{E5b} applies.
    In either case, the unique optimal scale is given by \eqref{eq:defect_balance_unique_cancelling_scale}.
    The corresponding numerical results are presented in \cref{fig:pendulum_order_comparison}.
\end{example}

\subsection{Scales vanishing with the step size}

We next examine the two alternatives in case~\ref{E6b}, neither of which yields a separate convergence result within the framework of this paper.

Suppose that case~\ref{E6b}\textnormal{(i)} holds for the exact points $(t, u(t))$ and $(t + h, u(t + h))$ whenever $t$ and $t + h$ belong to an interval and $h>0$ is sufficiently small.
Letting $h\searrow0$ in the condition $C=0$ gives
\begin{displaymath}
    3\bigl(\mathcal{L}_F F(t,u(t))\bigr)^2
    =
    F(t, u(t))\mathcal{L}_F^2 F(t, u(t)).
\end{displaymath}
Along the exact solution, this equality becomes $3\bigl(u''(t)\bigr)^2 = u'(t)u'''(t)$.
Since $cf>0$, we have $u'(t)\neq0$ throughout the interval, and hence
\begin{displaymath}
    \frac{d^2}{dt^2}\left(u'(t)^{-2}\right)
    =
    0.
\end{displaymath}
Consequently, on this interval, the exact solution is either a nonconstant affine function or a non-affine function of the form
\begin{displaymath}
    u(t)
    =
    C_1+C_2\sqrt{C_3t+C_4},
\end{displaymath}
for some constants $C_1,C_2,C_3,C_4\in\mathbb R$.
The affine family corresponds to case~\ref{E1}; case~\ref{E6b}\textnormal{(i)} therefore leaves only the non-affine square-root family.
Up to translations and scalings, this family is represented by the inverse equation in \cref{ex:autonomous_small_sigma} and has limited practical scope.
This reduction is consistent with \cref{rmk:3rd_4th_order_classification_matchings}, which identifies case~\ref{E6b}\textnormal{(i)} with the equality case of~\ref{G4} after setting $(s,y)=(t,x)$.

Case~\ref{E6b}\textnormal{(ii)} requires $cf=0$ and $c+f>0$, and therefore exactly one of $F(t,x)$ and $F(s,y)$ vanishes.
This condition cannot hold for every sufficiently close exact pair on an interval.
Indeed, if $F\bigl(t, u(t)\bigr) \neq 0$ at some $t$, then applying $cf=0$ to $(t,u(t))$ and $(t+h,u(t+h))$ forces $F\bigl(t+h,u(t+h)\bigr)=0$ for every sufficiently small $h>0$, contradicting the continuity of the mapping $t\mapsto F(t, u(t))$ as $h\searrow0$.
On the other hand, if $F\bigl(t,u(t)\bigr)=0$ throughout the interval, then $c+f=0$, which also contradicts case~\ref{E6b}\textnormal{(ii)}.
Thus, case~\ref{E6b}\textnormal{(ii)} cannot provide a positive scale-selection function satisfying the uniform bounds in \eqref{ineq:SE4-1} together with the defect-balance condition \eqref{eq:SE4-2}.
A convergence result would require a different, step-dependent construction of $\sigma_n$, which is not pursued here.

\section{Numerical results}
\label{sec:numerical}

This section presents numerical illustrations of exact tracing and of the scale-selection results developed above.
In the numerical comparisons, we use the following names.

\begin{definition}
    We call SE with $\sigma_n=1$ for every $n$ the \emph{second-order specular ellipse method} (SE2).
    We call SE with scales $\sigma_n$ satisfying either \eqref{eq:SE3} or \eqref{ineq:G4_vanishing_scale_rate} the \emph{third-order specular ellipse method} (SE3).
    We call SE with scales $\sigma_n$ satisfying \eqref{eq:SE4} the \emph{fourth-order specular ellipse method} (SE4).
\end{definition}

We abbreviate the Crank--Nicolson method as CN and the third- and fourth-order Runge--Kutta methods as RK3 and RK4, respectively.
All experiments were performed using Python 3.14 and the Python package \texttt{specular-differentiation} \cite{2026s_Jung_SIAM}.

\subsection{Exact tracing}

We compare SE and CN for the equation in \cref{ex:SE_trace} by setting $p=2.25$ and $q=0$ and considering the solution \eqref{eq:ellipse_sol} on the upper elliptic branch.
Both experiments use the prescribed step size $h=0.3$ on $[0,4.5]$, with the initial value taken from the exact solution at the left mesh endpoint.
In \cref{fig:ellipse_exactness}, SE uses $\sigma_n=\frac{b}{a}$ in both panels.
The left panel has $a = b = 2.26$, for which SE is SE2, while the right panel has $(a, b) = (2.26, 1.5)$.
The SE approximations coincide with the exact solution at every mesh point in both cases, whereas the CN approximations do not.

\begin{figure}[htbp]
    \centering
    \includegraphics[width=1\textwidth]{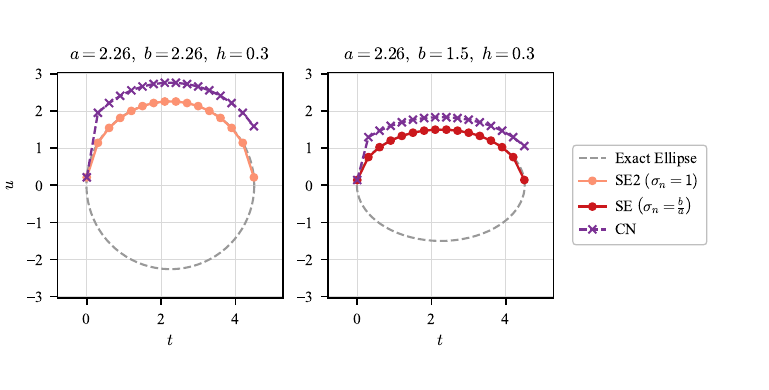}
    \par\vspace{-10pt}
    \caption{Numerical trajectories for the elliptic solutions of the problem in \cref{ex:SE_trace}.}
    \label{fig:ellipse_exactness}
\end{figure}

\subsection{Higher-order convergence with vanishing scales}

For the problem in \cref{ex:autonomous_small_sigma}, we set $u_0=1$ and apply SE2, SE3, CN, RK3, and RK4 on the interval $[0,1]$.
In \cref{fig:inverse_order}, we compare SE2 and the five SE3 choices $\sigma_n=h^p$, $p\in\left\{\frac{1}{2},1,2,3,4\right\}$, with CN, RK3, and RK4 using $17$ approximately logarithmically spaced values of $h$ from $10^{-1}$ to $10^{-3}$.
SE2 exhibits second-order convergence, while these five SE3 choices exhibit observed orders $3$, $4$, $6$, $8$, and $10$, respectively.
CN, RK3, and RK4 exhibit second-, third-, and fourth-order convergence, respectively.

\Cref{thm:G4_vanishing_scale_SE} gives at least third-order convergence for the SE3 scales in this experiment, while \eqref{ineq:autonomous_small_sigma-1} gives the bound $O\bigl(h^{2p + 2}\bigr)$ for $\sigma_n=h^p$ in this example.
Hence, the observed orders $3$, $4$, $6$, $8$, and $10$ follow from the direct estimate for the present equation.
By contrast, in \cref{fig:SE_vs_CN}, the scale supplied by the third-order construction produces third-order convergence, and no higher order is observed.

For each of the SE choices $\sigma_n=h^2$, $\sigma_n=h^3$, and $\sigma_n=h^4$, and for RK4, the error eventually reaches an accuracy floor.
For the SE3 curves, this behavior is consistent with double-precision roundoff and nonlinear-solver tolerances, while the RK4 error stops decreasing at a level consistent with double-precision roundoff.
A precise analysis of this behavior is beyond the scope of the present work.

\begin{figure}[htbp]
    \centering
    \includegraphics[width=1\textwidth]{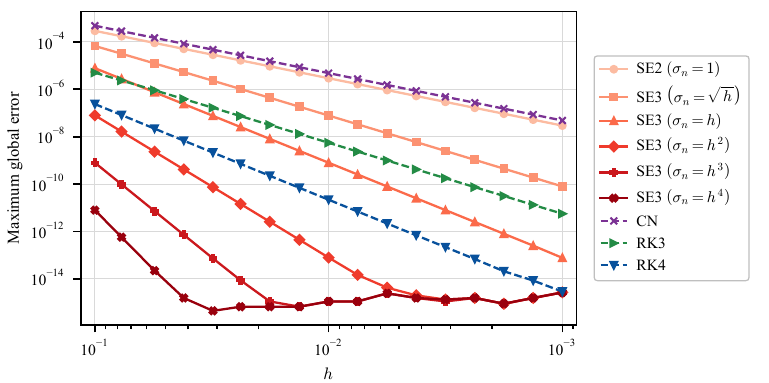}
    \par\vspace{-10pt}
    \caption{Maximum global errors for the problem in \cref{ex:autonomous_small_sigma}.}
    \label{fig:inverse_order}
\end{figure}

\subsection{Fourth-order convergence with optimal scales}

We compare SE2, SE3, SE4, and RK4 for the normalized pendulum family \eqref{ODE:normalized_pendulum_phase} considered in \cref{ex:pendulum_defect_balance_scale} on $\xi\in[0,0.8]$, with the corresponding initial value $u(0)$ for each $\theta=1,0.25,0.1,0.01$.
The computation concerns the normalized positive-velocity phase equation, not the original pendulum equation in time.
The turning points $\xi=\pm1$, where $u=0$ and the phase equation is singular, are excluded.
For each value of $\theta$, all four methods use the same $13$ step sizes, ranging from $10^{-1}$ to $10^{-4}$.

For SE3, case~\ref{G1} applies at $\xi_0=0$, where every positive scale is optimal; we therefore set $\sigma_0=1$ and use the scale in \eqref{eq:SE3} for $n=1,2,\ldots,N_h-1$.
For SE4, the numerical value and scale are determined together as in \cref{thm:SE4}.
For every trial value $v$ in the SE4 implicit solve, the coefficients $A$, $B$, $C$, and $D$ are evaluated at $(\xi_n,u_n)$ and
$(\xi_{n+1},v)$, and the conditions in cases~\ref{E5a} and~\ref{E5b} are checked.
Whenever either case applies, the scale in \eqref{eq:defect_balance_unique_cancelling_scale} is recomputed as $v$ changes.
The scale equation is enforced through this formula, while the state equation is solved numerically; the accepted pair satisfies both equations in \eqref{ODS:SE_4th_order} to the prescribed solver tolerances.

For $h=10^{-2}$, the selected scales and cancellation residuals computed from the accepted numerical pairs are summarized in \cref{tbl:pendulum_scale}.
For every tested value of $\theta$, the pair corresponding to $n=0$ belongs to case~\ref{E5a}, while those corresponding to $n=1,2,\ldots,79$ belong to case~\ref{E5b}.

\begin{table}[htbp]
    \centering
    \caption{
        Ranges of the selected scales and maximum cancellation residuals for the normalized pendulum problem in \cref{ex:pendulum_defect_balance_scale} with $h=10^{-2}$.
    }
    \label{tbl:pendulum_scale}
    \setlength{\tabcolsep}{10pt}
    \begin{tabular}{cccc}
        \toprule
        $\theta$
        & $\min_n \sigma_n$
        & $\max_n \sigma_n$
        & $\max_n \left| \mathcal{D}_{\sigma_n}(\xi_n,u_n) + \mathcal{D}_{\sigma_n}(\xi_{n+1},u_{n+1}) \right|$
        \\
        \midrule
        $1$
        & $0.90105$
        & $1.20073$
        & $7.105\times10^{-15}$
        \\
        $0.25$
        & $0.99394$
        & $1.01052$
        & $5.329\times10^{-15}$
        \\
        $0.1$
        & $0.99903$
        & $1.00167$
        & $3.553\times10^{-15}$
        \\
        $0.01$
        & $0.99999$
        & $1.00002$
        & $4.441\times10^{-15}$
        \\
        \bottomrule
    \end{tabular}
\end{table}

As shown in \cref{fig:pendulum_order_comparison}, the SE4 errors decay with observed order approximately four for every tested turning angle before reaching the accuracy floor.
At $h=10^{-2}$, RK4 is about $2.46$ times more accurate at $\theta=1$, whereas SE4 is about $6.37$, $39.8$, and $3.93\times10^3$ times more accurate at $\theta=0.25$, $0.1$, and $0.01$, respectively.
These comparisons concern accuracy at the same step size rather than at the same computational cost.
The preceding pendulum quantities concern the accepted numerical pairs and do not verify the uniform neighborhood assumption in \cref{thm:SE4}.

\begin{figure}[htbp]
    \centering
    \includegraphics[width=1\textwidth]{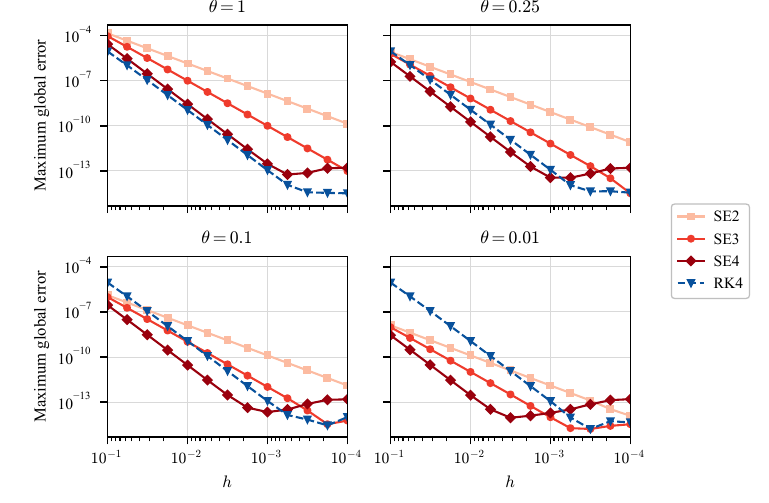}
    \par\vspace{-10pt}
    \caption{Maximum global errors for the normalized pendulum problem in \cref{ex:pendulum_defect_balance_scale}.}
    \label{fig:pendulum_order_comparison}
\end{figure}

\subsection{Defect cancellation and convergence with optimal scales}

We numerically solve the problem in \cref{ex:SE_vs_CN} on the interval $[0, 1]$.
The defect formula \eqref{eq:quadratic_decay_defect} and the corresponding unique optimal minimizer are derived in \cref{ex:SE_vs_CN}.
For two positive values $x$ and $y$, a direct calculation gives $A=-6(x^4+y^4)$, $B=6(x^4-y^4)^2$, and $C=6x^4y^4(x^4+y^4)$.
Hence, $D>0$ and $AC<0$, and case~\ref{E5b} applies.
Take $h=0.3$ and $t_n=nh$ for $n=0,1,2$.
For SE3 and SE4, respectively, define
\begin{displaymath}
    \sigma_n^{(3)} := u(t_n)^2 
    \qquad\text{and}\qquad
    \sigma_n^{(4)} := \Sigma\bigl(t_n, u(t_n); t_{n+1}, u(t_{n+1})\bigr).
\end{displaymath}
The selected scales satisfy
\begin{align*}
    \sigma_0^{(3)}
    &=1,
    &
    \sigma_1^{(3)}
    &=
    \frac{100}{169},
    &
    \sigma_2^{(3)}
    &=
    \frac{25}{64},
    \\
    \sigma_0^{(4)}
    &\approx
    0.877,
    &
    \sigma_1^{(4)}
    &\approx
    0.523,
    &
    \sigma_2^{(4)}
    &\approx
    0.349.
\end{align*}
For $n = 0, 1, 2$, define 
\begin{align*}
    \mathsf{D}^{(3)}_n(\sigma)
    &:= |\mathcal{D}_\sigma(t_n,u(t_n))|,    \\
    \mathsf{D}^{(4)}_n(\sigma)
    &:= \left| \mathcal{D}_\sigma(t_n,u(t_n)) + \mathcal{D}_\sigma(t_{n+1},u(t_{n+1})) \right|.
\end{align*}
Then, for all $n = 0, 1, 2$,
\begin{displaymath}
    \mathsf{D}^{(3)}_n \left( \sigma_n^{(3)} \right) = 0 
    \qquad\text{and}\qquad
    \mathsf{D}^{(4)}_n \left( \sigma_n^{(4)} \right) = 0.
\end{displaymath}
The left panel of \cref{fig:quadratic_decay_defect_cancellation} shows $\mathsf{D}^{(3)}_n(\sigma)$, while the right panel shows $\mathsf{D}^{(4)}_n(\sigma)$ for $n=0,1,2$.
The scale variable ranges over $10^{-2} \leq \sigma \leq 10^2$.
The blue dashed lines in the left and right panels mark $\sigma_n^{(3)}$ and $\sigma_n^{(4)}$, respectively, at which $\mathsf{D}_n^{(3)}$ and $\mathsf{D}_n^{(4)}$ vanish.
Both quantities are evaluated at exact solution values rather than at the numerical iterates.
Therefore, the variation of the cancelling scales with $n$ is not caused by accumulated numerical error.
Convergence is examined separately in \cref{fig:SE_vs_CN}.

\begin{figure}[htbp]
    \centering
    \includegraphics[width=1\textwidth]{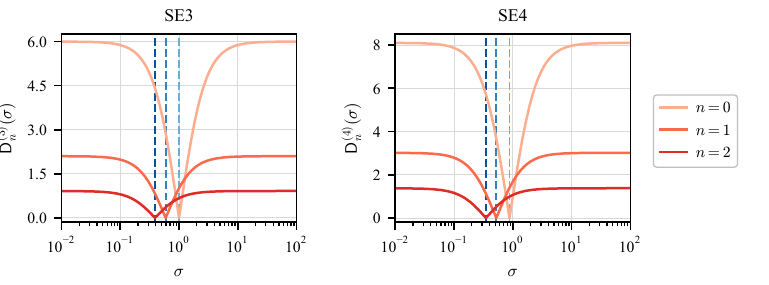}
    \par\vspace{-10pt}
    \caption{Defect cancellation by SE3 (left) and SE4 (right) for the problem in \cref{ex:SE_vs_CN} with $h=0.3$.}
    \label{fig:quadratic_decay_defect_cancellation}
\end{figure}

For the convergence experiment, SE3 sets $\sigma_n = u_n^2$ and keeps this scale fixed during the implicit solve.
At each trial value in the implicit solve, SE4 selects the unique optimal scale in case~\ref{E5b}.
Since $u(t)\geq\frac{1}{2}$ on $[0,1]$, the assumptions of \cref{thm:SE4} hold on $U_\rho(u)$ for every $\rho\in(0,\frac{1}{2})$.
We compare CN, SE2, SE3, SE4, RK3, and RK4 using the step sizes $h=10^{-1}2^{-k}$ for $k=0,1,\ldots,5$.
\Cref{fig:SE_vs_CN} shows second-order convergence for CN and SE2, third-order convergence for SE3 and RK3, and fourth-order convergence for SE4 and RK4. The observed orders of SE2, SE3, and SE4 agree with \cref{cor:convergence_orders_SE,thm:third_order_SE,thm:SE4}, respectively.
RK3 and RK4 are included as third- and fourth-order references, respectively.

\begin{figure}[htbp]
    \centering
    \includegraphics[width=1\textwidth]{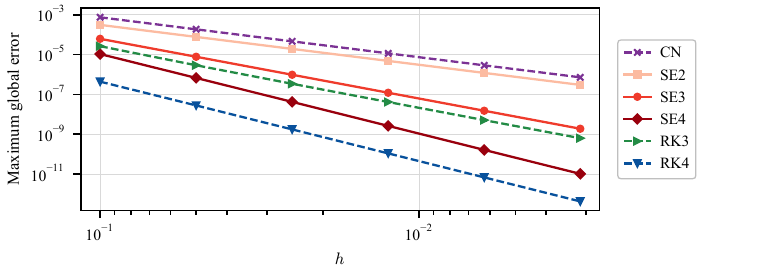}
    \par\vspace{-10pt}
    \caption{Maximum global errors for the problem in \cref{ex:SE_vs_CN}.}
    \label{fig:SE_vs_CN}
\end{figure}

\subsection{One-step behavior for diverging scales}

By \eqref{eq:C_infty}, the SE relation approaches the CN relation as $\sigma\to\infty$.
The following experiment compares the corresponding one-step errors.

We revisit the equation in \cref{ex:autonomous_most_cases}.
For each $h$, we choose consecutive exact points $(t_n,u(t_n))$ and $(t_{n+1},u(t_{n+1}))$, with $t_{n+1}-t_n=h$, such that
\begin{displaymath}
    \mathcal{L}_F^2F(t_n,u(t_n))
    +
    \mathcal{L}_F^2F(t_{n+1},u(t_{n+1}))
    =
    0.
\end{displaymath}
For these pairs, $A=0$, $B<0$, $C<0$, and $cf>0$, and hence case~\ref{E3b1} applies.

\Cref{fig:autonomous_large_scale} compares the one-step errors of SE2, SE with the fixed scales $\sigma=3$ and $10$, SE with $\sigma=h^{-1}$, and CN.
We use the step sizes $h=2^{-k}$ for $k=1, 2, \ldots, 7$.
The exact pair is selected separately for each $h$; thus, this is a one-step experiment rather than a global convergence experiment.
The one-step errors of the three fixed-scale choices decay with observed order three, whereas those of SE with $\sigma=h^{-1}$ and CN decay with observed order five.
In this experiment, SE with $\sigma=h^{-1}$ has a slightly smaller error than CN.

\begin{figure}[htbp]
    \centering
    \includegraphics[width=1\textwidth]{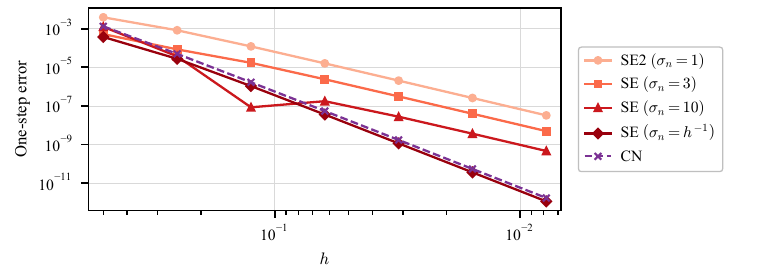}
    \par\vspace{-10pt}
    \caption{One-step errors in case~\ref{E3b1} for the problem in \cref{ex:autonomous_most_cases}.}
    \label{fig:autonomous_large_scale}
\end{figure}

\subsection{Numerical considerations}

In the present implementation \cite{2026s_Jung_SIAM}, the classical Runge--Kutta methods were generally faster in the tested examples.
A systematic comparison of accuracy and computational cost, accounting for derivative evaluations and parallel implementation, is left for future work.

Although $\mathcal{B}=\mathcal{C}$ and the representation \eqref{eq:repr_B} is exact, direct floating-point evaluation of the trigonometric representation $\mathcal{B}$ can be unstable when the averaged angle is close to $\pm\frac{\pi}{2}$, since the tangent amplifies errors in the angle.
For this reason, the experiments use the algebraic representation $\mathcal{C}$ and its scaled form $\mathcal{C}_\sigma$.

\section{Conclusion}

For a given problem \eqref{ODE:IVP} with a sufficiently smooth right-hand side $F$, the values of $F$ and its derivatives determine the applicable case and, under the corresponding regularity and uniformity assumptions, whether the specular ellipse method is exact, converges with order two, converges with order three in case~\ref{G3} or the equality case of~\ref{G4}, or converges with order four in cases~\ref{E5a} and~\ref{E5b}.

\appendix 
\crefalias{section}{appendix}

\section{An auxiliary Taylor estimate}
\label{apx:C_sigma}

We collect properties of the function $\mathcal{C}_{\sigma}$; see \cite[App.~A]{2026a_Jung} for properties of the function $\mathcal{C}$.
First, observe that 
\begin{equation}    \label{eq:algebraic_C_sigma}
    \mathcal{C}_{\sigma}(\alpha, \beta) = \frac{\beta \sqrt{\sigma^2 + \alpha^2} + \alpha \sqrt{\sigma^2 + \beta^2}}{\sqrt{\sigma^2 + \alpha^2} + \sqrt{\sigma^2 + \beta^2}}
\end{equation}
for every $\sigma>0$ and $(\alpha,\beta)\in\mathbb{R}^2$.

\begin{lemma}
    Let $(\alpha, \beta) \in \mathbb{R}^2$ be fixed. 
    Then
    \begin{equation}    \label{eq:C_0}
        \lim_{\sigma \searrow 0}
        \mathcal{C}_{\sigma}(\alpha,\beta)
        =
        \begin{cases}
            \displaystyle
            \frac{2\alpha\beta}{\alpha+\beta},
            & \text{if $\alpha\beta>0$,}\\[2mm]
            0,
            & \text{if $\alpha\beta\leq0$,}
        \end{cases}
    \end{equation}
    and 
    \begin{equation}    \label{eq:C_infty}
        \lim_{\sigma \to \infty} \mathcal{C}_{\sigma} (\alpha,\beta) = \frac{\alpha + \beta}{2}.
    \end{equation}
\end{lemma}

\begin{proof}
    The results follow from direct calculations.
\end{proof}

\begin{lemma} \label{lem:Taylor_C_sigma}
    Let $\sigma > 0$.
    The function $\mathcal{C}_{\sigma}$ is smooth in $\mathbb{R}^2$.
    The second-order Taylor expansion of $\mathcal{C}_{\sigma}$ about $(\alpha,\beta)\in\mathbb{R}^2$ is given by
    \begin{align*}
        \mathcal{C}_{\sigma} (\alpha + h_1, \beta+h_2)
        =&
        \, \mathcal{C}_{\sigma}(\alpha,\beta)
        + \left[ \sigma^2+ \mathcal{C}_{\sigma}(\alpha,\beta)^2 \right] \left[ \frac{h_1}{2(\sigma^2+\alpha^2)} + \frac{h_2}{2(\sigma^2+\beta^2)} \right.
        \\
        &\, + \frac{\mathcal{C}_{\sigma}(\alpha,\beta)-2\alpha}{4(\sigma^2+\alpha^2)^2} h_1^2
            + \frac{\mathcal{C}_{\sigma}(\alpha,\beta)}{2(\sigma^2+\alpha^2)(\sigma^2+\beta^2)} h_1h_2
        \\
        &\, +\left. \frac{\mathcal{C}_{\sigma}(\alpha,\beta)-2\beta}{4(\sigma^2+\beta^2)^2} h_2^2 \right]
        + R_{\sigma}(h_1,h_2),
    \end{align*}
    where
    \begin{displaymath}
        \left|R_{\sigma}(h_1,h_2)\right|
        \leq
        \lambda \left(|h_1|^3+|h_2|^3\right)
    \end{displaymath}
    for all sufficiently small $h_1,h_2$, where $\lambda\equiv\lambda(\alpha,\beta,\sigma)>0$ is a constant.
\end{lemma}

\begin{proof}
    For each fixed $\sigma>0$, the smoothness of $\mathcal{C}_\sigma$ follows directly from that of $\mathcal{C}$.
    Set 
    \begin{displaymath}
        \theta_{\sigma}(\alpha,\beta) := \frac{1}{2}\arctan\frac{\alpha}{\sigma} + \frac{1}{2}\arctan\frac{\beta}{\sigma}.
    \end{displaymath}
    Then $\mathcal{C}_{\sigma}(\alpha,\beta) = \sigma\tan\theta_{\sigma}(\alpha,\beta)$, and hence $\sigma^2 \sec^2\theta_{\sigma}(\alpha,\beta) = \sigma^2 + \mathcal{C}_{\sigma}(\alpha,\beta)^2$.
    Direct computations give
    \begin{align}
        \partial_{\alpha} \mathcal{C}_{\sigma}(\alpha,\beta)
        &=
        \sec^2\theta_{\sigma}(\alpha,\beta) \frac{\sigma^2}{2(\sigma^2+\alpha^2)}
        =
        \frac{\sigma^2 + \mathcal{C}_{\sigma}(\alpha, \beta)^2}{2(\sigma^2 + \alpha^2)},  \label{lem:Taylor_C_sigma-a}
        \\
        \partial_{\beta} \mathcal{C}_{\sigma}(\alpha,\beta)
        &=
        \sec^2\theta_{\sigma}(\alpha,\beta) \frac{\sigma^2}{2(\sigma^2+\beta^2)}
        =
        \frac{\sigma^2 + \mathcal{C}_{\sigma}(\alpha, \beta)^2}{2(\sigma^2 + \beta^2)},   \label{lem:Taylor_C_sigma-b}
        \\
        \partial_{\alpha\alpha} \mathcal{C}_{\sigma}(\alpha,\beta)
        &=
        \sec^2\theta_{\sigma}(\alpha,\beta) \frac{\sigma^2 \bigl( \mathcal{C}_{\sigma}(\alpha,\beta)-2\alpha \bigr)}{2(\sigma^2+\alpha^2)^2},   \nonumber
        \\
        \partial_{\beta\beta} \mathcal{C}_{\sigma}(\alpha,\beta)
        &=
        \sec^2\theta_{\sigma}(\alpha,\beta) \frac{ \sigma^2 \bigl( \mathcal{C}_{\sigma}(\alpha,\beta)-2\beta \bigr)}{2(\sigma^2+\beta^2)^2
        },  \nonumber
        \\
        \partial_{\alpha\beta} \mathcal{C}_{\sigma}(\alpha,\beta)
        &=
        \sec^2\theta_{\sigma}(\alpha,\beta) \frac{ \sigma^2\mathcal{C}_{\sigma}(\alpha,\beta)}{2(\sigma^2+\alpha^2)(\sigma^2+\beta^2)} =
        \partial_{\beta\alpha} \mathcal{C}_{\sigma}(\alpha,\beta).  \nonumber
    \end{align}
    The stated expansion now follows from Taylor's theorem.
    Since $\mathcal{C}_\sigma$ is smooth, its third-order partial derivatives are bounded on a sufficiently small compact neighborhood of $(\alpha,\beta)$, which gives the desired remainder estimate.
\end{proof}

\begin{lemma}
    \label{lem:uniform_diagonal_C_sigma}
    Let $K\subset\mathbb{R}$ be nonempty and compact, and let $\eta > 0$.
    Then there exist constants $r \equiv r(K,\eta)>0$ and $\lambda\equiv\lambda(K, \eta)>0$ such that, for every $\sigma > 0$, every $\alpha\in K$ satisfying $\sigma^2 + \alpha^2 \geq \eta$ and every $\delta \in \mathbb{R}$ with $|\delta| \leq r$,
    \begin{align*}
        \mathcal{C}_\sigma(\alpha,\alpha+\delta)
        &=
        \alpha + \frac{1}{2}\delta - \frac{\alpha}{4(\sigma^2+\alpha^2)}\delta^2 + \frac{\alpha^2-\sigma^2}{8(\sigma^2+\alpha^2)^2}\delta^3 + R_\sigma(\alpha,\delta),
    \end{align*}
    where the estimate
    \begin{displaymath}
        |R_\sigma(\alpha, \delta)|
        \leq
        \lambda|\delta|^4
    \end{displaymath}
    holds uniformly for all such choices of $\alpha$, $\sigma$, and $\delta$.
\end{lemma}

\begin{proof}
    Define 
    \begin{equation}    \label{def:uniform_diagonal_radius}
        M :=\max_{\alpha\in K}|\alpha|
        \qquad\text{and}\qquad        
        r := \left( M^2+\frac{\eta}{2} \right)^{\frac{1}{2}} - M. 
    \end{equation}
    Then $r>0$ and $4Mr + 2r^2=\eta$.
    Fix $\alpha\in K$ satisfying $\sigma^2+\alpha^2\geq\eta$, and let $\delta\in\mathbb{R}$ satisfy $|\delta|\leq r$.
    Let $w := (2\alpha\delta+\delta^2)(\sigma^2+\alpha^2)^{-1}$.
    Then $|w| \leq \frac{1}{2}$.
    Also, \eqref{eq:algebraic_C_sigma} yields
    \begin{equation}    \label{eq:exact_diagonal_C_sigma}
        \mathcal{C}_{\sigma}(\alpha,\alpha + \delta)
        = 
        \alpha + \frac{\delta}{1+\sqrt{1 + w}} .
    \end{equation}
    Taylor's theorem at $x=0$ applied to the mapping $x \mapsto \left( 1+\sqrt{1+x} \right)^{-1}$ gives
    \begin{displaymath}
        \frac{1}{1+\sqrt{1+w}}
        =
        \frac{1}{2} - \frac{w}{8} + \frac{w^2}{16} + O(|w|^3)
    \end{displaymath}
    uniformly for $|w|\leq\frac{1}{2}$ (the uniformity follows from the boundedness of the third derivative on $[-\frac{1}{2},\frac{1}{2}]$).
    Substituting this expansion into \eqref{eq:exact_diagonal_C_sigma}, we obtain
    \begin{align*}
        \mathcal{C}_\sigma(\alpha,\alpha+\delta)
        &=
        \alpha + \frac{\delta}{2} - \frac{\delta}{8} \cdot \frac{2\alpha\delta+\delta^2}{\sigma^2+\alpha^2} + \frac{\delta}{16}
        \left( \frac{2\alpha\delta+\delta^2}{\sigma^2+\alpha^2} \right)^2 + \delta \, O\left( \left\vert \frac{2\alpha\delta+\delta^2}{\sigma^2+\alpha^2} \right\vert^3  \right)
        \\
        &=
        \alpha + \frac{\delta}{2} - \frac{\alpha}{4(\sigma^2+\alpha^2)} \delta^2 + \frac{\alpha^2-\sigma^2}{8(\sigma^2+\alpha^2)^2} \delta^3 + O(|\delta|^4).
    \end{align*}
    Since $|\alpha|\leq M$, $|\delta|\leq r$, and $\sigma^2+\alpha^2\geq\eta$, all terms absorbed into $O(|\delta|^4)$ are bounded uniformly.
    Since $M$ and $r$ depend only on $K$ and $\eta$, the stated remainder estimate follows by choosing a sufficiently large constant $\lambda>0$ depending only on $K$ and $\eta$.
\end{proof}

The following estimate provides an error-propagation bound that is uniform over all positive scales.

\begin{lemma} \label{lem:Lipschitz_C_sigma}
    For every $\sigma>0$ and every $(\alpha_1,\beta_1),(\alpha_2,\beta_2)\in\mathbb{R}^2$,
    \begin{displaymath}
        \left| \, \mathcal{C}_{\sigma}(\alpha_1,\beta_1) - \mathcal{C}_{\sigma}(\alpha_2,\beta_2) \right|
        \leq
        2 \left(|\alpha_1-\alpha_2| + |\beta_1-\beta_2| \right).
    \end{displaymath}
\end{lemma}

\begin{proof}
    From \eqref{lem:Taylor_C_sigma-a} and \eqref{lem:Taylor_C_sigma-b}, we have 
    \begin{displaymath}
        \partial_{\alpha}\mathcal{C}_{\sigma}(\alpha,\beta)
        =
        \frac{\cos^2\theta}{2\cos^2m} 
        \qquad\text{and}\qquad
        \partial_{\beta}\mathcal{C}_{\sigma}(\alpha,\beta)
        =
        \frac{\cos^2\phi}{2\cos^2m},
    \end{displaymath}
    where $\theta:=\arctan\frac{\alpha}{\sigma}$, $\phi:=\arctan\frac{\beta}{\sigma}$, and $m:=\frac{\theta+\phi}{2}$.
    Since $\theta,\phi\in(-\frac{\pi}{2},\frac{\pi}{2})$, all three cosines are positive.
    Moreover,
    \begin{displaymath}
        \cos\theta
        \leq
        \cos\theta+\cos\phi
        =
        2\cos m\cos\frac{\theta-\phi}{2}
        \leq
        2\cos m,
    \end{displaymath}
    and the analogous estimate holds for $\cos\phi$.
    Hence $0 < \partial_{\alpha}\mathcal{C}_{\sigma} \leq 2$ and $0 < \partial_{\beta}\mathcal{C}_{\sigma} \leq 2$.
    Applying the Mean Value Theorem to each term and using the preceding derivative bounds, we obtain
    \begin{align*}
        \left| \mathcal{C}_{\sigma}(\alpha_1,\beta_1) - \mathcal{C}_{\sigma}(\alpha_2,\beta_2) \right| 
        &\leq
        \left| \mathcal{C}_{\sigma}(\alpha_1,\beta_1) - \mathcal{C}_{\sigma}(\alpha_2,\beta_1) \right| + \left| \mathcal{C}_{\sigma}(\alpha_2,\beta_1) - \mathcal{C}_{\sigma}(\alpha_2,\beta_2) \right|\\
        &\leq 
        \left| \partial_{\alpha}\mathcal{C}_{\sigma} (c_1, \beta_1)\right| |\alpha_1 - \alpha_2| + \left| \partial_{\beta}\mathcal{C}_{\sigma} (\alpha_2, c_2) \right| |\beta_1 - \beta_2| \\
        &\leq 2|\alpha_1 - \alpha_2| + 2|\beta_1 - \beta_2|,
    \end{align*}
    where $c_1\in \left[\min\{\alpha_1,\alpha_2\},\max\{\alpha_1,\alpha_2\}\right]$ and $c_2\in \left[\min\{\beta_1,\beta_2\},\max\{\beta_1,\beta_2\}\right]$.
\end{proof}

\begin{lemma}
    \label{lem:scale_sensitivity_C_sigma}
    Fix $\sigma>0$.
    Then, for every $\mu, \nu \geq \sigma$ and every $\alpha, \beta \in \mathbb{R}$,
    \begin{equation}
        \label{ineq:scale_sensitivity_C_sigma}
        \left| \mathcal{C}_{\mu}(\alpha, \beta) - \mathcal{C}_{\nu}(\alpha, \beta) \right|
        \leq
        \frac{|\alpha + \beta|\,|\alpha - \beta|^2}{4 \sigma^3}|\mu - \nu|.
    \end{equation}
\end{lemma}

\begin{proof}
    If $\mu=\nu$, then \eqref{ineq:scale_sensitivity_C_sigma} is immediate.
    Without loss of generality, assume that $\nu > \mu$.
    Define the function $\varphi : (0, 1] \to \mathbb{R}$ by 
    \begin{displaymath}
        \varphi(\lambda) := \frac{1}{\lambda} \mathcal{C}_{\sigma}(\lambda \alpha, \lambda \beta).
    \end{displaymath}
    Using \eqref{eq:algebraic_C_sigma}, \eqref{lem:Taylor_C_sigma-a}, and \eqref{lem:Taylor_C_sigma-b}, we obtain
    \begin{displaymath}
        \varphi'(\lambda)
        = 
        -\frac{\sigma^2\lambda(\alpha+\beta)(\alpha-\beta)^2}{\sqrt{\sigma^2+\lambda^2\alpha^2} \sqrt{\sigma^2+\lambda^2\beta^2}\bigl(\sqrt{\sigma^2+\lambda^2\alpha^2} + \sqrt{\sigma^2+\lambda^2\beta^2} \bigr)^2}
    \end{displaymath}
    for all $\lambda \in (0, 1)$.
    Since both square roots in the denominator are bounded below by $\sigma$, we have
    \begin{equation}   \label{ineq:scale_sensitivity_C_sigma-1}
        \left\vert \varphi'(\lambda) \right\vert
        \leq
        \frac{\lambda |\alpha+\beta| \, |\alpha-\beta|^2}{4 \sigma^2}.
    \end{equation}
    Applying the Mean Value Theorem to $\varphi$, there exists $c \in \bigl(\frac{\sigma}{\nu}, \frac{\sigma}{\mu} \bigr)$ such that
    \begin{equation}   \label{eq:scale_sensitivity_C_sigma-2}
        \varphi\left( \frac{\sigma}{\mu} \right) - \varphi\left( \frac{\sigma}{\nu} \right)
        =
        \sigma \, \varphi'(c) \left( \frac{1}{\mu} - \frac{1}{\nu} \right).
    \end{equation}
    Combining \eqref{eq:scale_sensitivity_C_sigma-2} and \eqref{ineq:scale_sensitivity_C_sigma-1} yields 
    \begin{displaymath}
        \left\vert \mathcal{C}_{\mu}(\alpha, \beta) - \mathcal{C}_{\nu}(\alpha, \beta) \right\vert 
        = \sigma \, | \varphi'(c) | \left\vert \frac{1}{\mu} - \frac{1}{\nu} \right\vert
        \leq \frac{c \, |\alpha+\beta| \, |\alpha-\beta|^2}{4 \sigma} \cdot \frac{|\mu - \nu|}{\sigma^2}.
    \end{displaymath}
    This implies \eqref{ineq:scale_sensitivity_C_sigma} since $c \in (0, 1)$.
\end{proof}

\section{Auxiliary results for the convergence analysis}
\label{apx:auxiliary_results}

The following lemma converts an error-dependent local residual bound into a global convergence estimate for SE.

\begin{lemma}
    \label{lem:lte_to_global_error_SE}
    Suppose that hypothesis~\ref{H1} holds.
    Let $p\geq2$, $P,\varepsilon>0$, and $Q,\Lambda\geq0$ be fixed.
    Define
    \begin{equation}
    \label{eq:W_rho_epsilon}
    W_{\rho,\varepsilon}(u)
    :=
    \left\{
        \bigl((t,x),(s,y)\bigr)\in U_\rho(u)^2
        \,\middle|\,
        0\leq s-t\leq\varepsilon
    \right\}.
    \end{equation}
    For each fixed $h\in(0,\varepsilon]$ and every $n=0,1,\ldots,N_h-1$, let $\Sigma_n$ be a fixed function from $W_{\rho,\varepsilon}(u)$ to $(0,\infty)$, possibly depending on $h$, and define the real-valued functions $\mathcal{N}_n$ and $\mathcal{E}_n$ on this set by
    \begin{align*}
        \mathcal{N}_n(t, x; s, y)
        &:=
        \mathcal{C}_{\Sigma_n(t, x; s, y)} \left(F(s, y), F(t,x)\right),    \\
        \mathcal{E}_n(t, x; s, y)
        &:=
        \mathcal{C}_{\Sigma_n(t, x; s, y)} \left(u'(s), u'(t)\right).
    \end{align*}
    Suppose that whenever $0<h\leq \varepsilon$, $0\leq n<N_h$, and $\{(t_j,u_j)\}_{j=0}^{n}\subseteq U_\rho(u)$, the following inequalities hold:
    \begin{align}
        &\left| \, \mathcal{N}_n(t_n, u_n; t_{n+1}, v) - \mathcal{N}_n(t_n, u_n; t_{n+1}, w) \right|
        \leq
        \Lambda|v-w|
        \label{ineq:lte_to_global_error_SE-0}
    \end{align}
    for every $(t_{n+1},v),(t_{n+1},w)\in U_\rho(u)$ and
    \begin{equation}    \label{ineq:lte_to_global_error_SE-1}
        \left| \frac{u(t_{n+1}) - u(t_n)}{h} - \mathcal{E}_n(t_n, u_n; t_{n+1}, v) \right|
        \leq
        P h^p + Q h^2 \left( |e_n| + |v-u(t_{n+1})| \right)
    \end{equation}
    for every $(t_{n+1},v)\in U_\rho(u)$.
    Then there exist constants
    \begin{align*}
        H
        &\equiv H(p,P,Q,\Lambda,\varepsilon,\rho,L,t_0,T) > 0,   \\
        C_g
        &\equiv C_g(P,Q,L,t_0,T) > 0
    \end{align*}
    such that, for every $0<h\leq H$, there exists a unique sequence $\{u_n\}_{n=0}^{N_h}$ with $u_0=u(t_0)$ satisfying
    \begin{equation} \label{eq:lte_to_global_error_SE-5}
        u_{n+1}
        =
        u_n + h \, \mathcal{N}_n(t_n, u_n; t_{n+1}, u_{n+1}),
        \qquad
        n=0,1,\ldots,N_h-1,
    \end{equation}
    with $\{(t_n,u_n)\}_{n=0}^{N_h}\subseteq U_\rho(u)$ and
    \begin{displaymath}
        \max_{0\leq n\leq N_h}|e_n|
        \leq
        C_g h^p.
    \end{displaymath}
\end{lemma}

\begin{proof}
    Define
    \begin{equation}    \label{def:C_g}
        C_g
        :=
        2P(T-t_0)
        \exp\bigl((8L+4Q)(T-t_0)\bigr).
    \end{equation}
    Choose $H\in(0,\min\{1,\varepsilon\}]$ such that
    \begin{align}
        &\Lambda H\leq\frac12,
        \qquad
        2LH+QH^3\leq\frac12,
        \qquad
        C_gH^p\leq\frac{\rho}{2},    \label{ineq:lte_to_global_error_SE-2}    \\
        &3\rho LH+\frac{3Q\rho}{2}H^3+PH^{p+1}
        \leq
        \frac{\rho}{2}.     \label{ineq:lte_to_global_error_SE-6} 
    \end{align}

    Fix $h\in(0,H]$.
    We prove inductively that, for every $n=0,1,\ldots,N_h$, the value $u_n$ is uniquely determined and satisfies
    \begin{equation} \label{ineq:lte_to_global_error_SE-3}
        (t_n,u_n)\in U_\rho(u)
        \qquad\text{and}\qquad
        |e_n|\leq C_gh^p.
    \end{equation}
    Since $u_0 = u(t_0)$, the assertion holds for $n=0$.

    Suppose that the assertion holds through an index $n<N_h$.
    Define the closed interval
    \begin{displaymath}
        V_{n+1}
        :=
        \left\{
            v\in\mathbb{R}
            \,\middle|\,
            |v-u(t_{n+1})|\leq\rho
        \right\}
    \end{displaymath}
    and the function $\Phi_n:V_{n+1}\to\mathbb{R}$ by
    \begin{displaymath}
        \Phi_n(v)
        :=
        u_n+h \, \mathcal{N}_n(t_n,u_n;t_{n+1},v).
    \end{displaymath}
    For every $v,w\in V_{n+1}$, the definition of $V_{n+1}$, the induction hypothesis, and $h\leq H\leq \varepsilon$ imply that
    \begin{displaymath}
        \bigl((t_n,u_n),(t_{n+1},v)\bigr),
        \quad
        \bigl((t_n,u_n),(t_{n+1},w)\bigr)
        \in
        W_{\rho,\varepsilon}(u).
    \end{displaymath}
    By \eqref{ineq:lte_to_global_error_SE-0} and \eqref{ineq:lte_to_global_error_SE-2},
    \begin{align*}
        |\Phi_n(v) -\Phi_n(w)|
        &\leq
        \Lambda h|v-w|
        \leq
        \frac{1}{2}|v-w|.
    \end{align*}
    Thus, $\Phi_n$ is a contraction.
    Furthermore, for every $v\in V_{n+1}$, \cref{lem:Lipschitz_C_sigma}, hypothesis~\ref{H1}, \eqref{ineq:lte_to_global_error_SE-1}, \eqref{ineq:lte_to_global_error_SE-2}, \eqref{ineq:lte_to_global_error_SE-3}, and \eqref{ineq:lte_to_global_error_SE-6} give
    \begin{align*}
        |\Phi_n(v) - u(t_{n+1})|
        &\leq
        (1 + 2Lh+Qh^3)|e_n| + (2Lh+Qh^3)|v-u(t_{n+1})| + Ph^{p+1}\\
        &\leq
        \frac{\rho}{2} + 3L\rho h + \frac{3Q\rho}{2}h^3 + Ph^{p+1}\\
        &\leq
        \rho.
    \end{align*}
    Hence, $\Phi_n(V_{n+1})\subseteq V_{n+1}$.
    The Banach fixed-point theorem gives a unique $u_{n+1}\in V_{n+1}$ satisfying \eqref{eq:lte_to_global_error_SE-5}. 

    Since $\Phi_n(u_{n+1})=u_{n+1}$, taking $v=u_{n+1}$ in the first inequality of the preceding estimate gives
    \begin{equation}
        \label{ineq:lte_to_global_error_SE-4}
        (1-2Lh-Qh^3)|e_{n+1}|
        \leq
        (1+2Lh+Qh^3)|e_n|
        +
        Ph^{p+1}.
    \end{equation}
    We apply \cref{lem:discrete_error_accumulation} with $N=n+1$, $\{E_j\}_{j=0}^{n+1}=\{|e_j|\}_{j=0}^{n+1}$, $C_\ell=P$, $\alpha_h=\beta_h=2Lh+Qh^3$, and $K=4L+2Q$.
    Indeed, $E_0=0$, and the second inequality in \eqref{ineq:lte_to_global_error_SE-2} gives $\alpha_h\leq\frac{1}{2}$.
    Moreover, since $h\leq 1$,
    \begin{displaymath}
        \alpha_h+\beta_h
        =
        4Lh+2Qh^3
        \leq
        (4L+2Q)h
        =
        Kh.
    \end{displaymath}
    Since $n < N_h$, we also have
    \begin{displaymath}
        N h
        =
        (n+1)h
        \leq
        N_h h
        \leq
        T-t_0.
    \end{displaymath}
    Hence,
    \begin{displaymath}
        \max_{0\leq j\leq n+1}|e_j|
        \leq
        2P(T-t_0)
        \exp\bigl((8L+4Q)(T-t_0)\bigr)h^p
        =
        C_gh^p.
    \end{displaymath}
    Together with $u_{n+1}\in V_{n+1}$, this establishes the induction assertion for the index $n+1$ and closes the induction.
\end{proof}

The following lemma provides the scalar error accumulation estimate used in the convergence proofs.

\begin{lemma}[Discrete error accumulation]
    \label{lem:discrete_error_accumulation}
    Let $T>t_0$, $p>0$, $C_{\ell}>0$, and $K\geq0$.
    Let $h>0$ and $N\in\mathbb{N}$ satisfy $N h \leq T-t_0$.
    Suppose that $\{E_n\}_{n=0}^{N}$ is a nonnegative sequence with $E_0 = 0$ and that $\alpha_h,\beta_h\geq0$ satisfy $\alpha_h \leq \frac{1}{2}$ and $\alpha_h+\beta_h \leq Kh$.
    If
    \begin{equation}    \label{ineq:discrete_error_accumulation}
        (1-\alpha_h)E_{n+1}
        \leq
        (1+\beta_h)E_n
        +
        C_{\ell}h^{p+1}
    \end{equation}
    for every $n=0,1,\ldots,N-1$, then
    \begin{displaymath}
        \max_{0\leq n\leq N}E_n
        \leq
        2C_{\ell}(T-t_0)
        \exp\left(
            2K(T-t_0)
        \right)
        h^p.
    \end{displaymath}
\end{lemma}

\begin{proof}
    We have
    \begin{displaymath}
        \frac{1+\beta_h}{1-\alpha_h}
        =
        1 + \frac{\alpha_h+\beta_h}{1-\alpha_h}
        \leq 
        1 + 2 (\alpha_h + \beta_h)
        \leq
        \exp(2Kh).
    \end{displaymath}
    Thus, \eqref{ineq:discrete_error_accumulation} implies
    \begin{displaymath}
        E_{n+1}
        \leq
        \exp(2Kh) E_n + 2 C_{\ell} h^{p+1}.
    \end{displaymath}    
    Iterating this inequality from $E_0=0$ gives, for every $n=1,2,\ldots,N$,
    \begin{align*}
        E_n
        &\leq
        2C_{\ell}h^{p+1} \sum_{j=0}^{n-1} \exp(2Kjh) 
        \leq
        2C_{\ell}nh^{p+1} \exp(2Knh)\\
        &\leq
        2C_{\ell}(T-t_0) \exp\left( 2K(T-t_0) \right) h^p.
    \end{align*}
    Taking the maximum over $n = 0, 1, \ldots, N$ completes the proof.
\end{proof}

The following lemma gives a scale-sensitivity estimate for $\mathcal{C}_\sigma$ evaluated at $u'(t)$ and $u'(s)$.

\begin{lemma}
    \label{lem:scale_sensitivity_along_solution}
    Let $u\in C^3([t_0,T])$, and fix $\sigma>0$.
    Then, for every $t, s \in[t_0, T]$ and every $\mu,\nu\geq\sigma$,
    \begin{displaymath}
        \left| \mathcal{C}_\mu\bigl(u'(t),u'(s)\bigr) - \mathcal{C}_\nu\bigl(u'(t),u'(s)\bigr) \right|
        \leq
        \frac{M_1M_2^2}{2\sigma^3} |t-s|^2|\mu-\nu|.
    \end{displaymath}
\end{lemma}

\begin{proof}
    We have $|u'(t)+u'(s)| \leq 2 M_1$ and $|u'(t)-u'(s)|\leq M_2|t-s|$.
    Hence, \cref{lem:scale_sensitivity_C_sigma} gives the desired inequality.
\end{proof}

\bibliographystyle{siamplain}
\bibliography{reference}{}

@Book{2016_Butcher_BOOK,
  author    = {Butcher, J. C.},
  publisher = {John Wiley \& Sons, Ltd., Chichester},
  title     = {Numerical methods for ordinary differential equations},
  year      = {2016},
  edition   = {3rd},
  isbn      = {978-1-119-12150-3},
  doi       = {10.1002/9781119121534},
  mrclass   = {65-02 (65Lxx)},
  mrnumber  = {3559553},
  pages     = {xxiii+513},
}

@Book{2026a_Jung,
  author    = {Jung, Kiyuob},
  publisher = {preprint, arXiv:2601.09900 [math.CA]},
  title     = {Specular differentiation in one dimension: a quasi-mean value theorem, regularity, and discontinuities},
  year      = {2026},
  url       = {https://arxiv.org/abs/2601.09900},
}

@Article{1987_Evans,
  author    = {Evans, David John and Sanugi, Bahrom},
  journal   = {Int. J. Comput. Math.},
  title     = {A comparison of numerical ODE solvers based on arithmetic and geometric means},
  year      = {1987},
  number    = {1},
  pages     = {37--62},
  volume    = {23},
  doi       = {10.1080/00207168708803607},
  fjournal  = {International Journal of Computer Mathematics},
  publisher = {Taylor \& Francis},
}

@Article{1991_Evans,
  author    = {Evans, David John and Sanugi, Bahrom},
  journal   = {Int. J. Comput. Math.},
  title     = {A comparison of nonlinear trapezoidal formulae for solving initial value problems},
  year      = {1991},
  number    = {1-2},
  pages     = {65--79},
  volume    = {41},
  doi       = {10.1080/00207169108804027},
  fjournal  = {International Journal of Computer Mathematics},
  publisher = {Taylor \& Francis},
}

@Article{1996_Sivakumar,
  author   = {Sivakumar, T. R. and Savithri, S.},
  journal  = {J. Comput. Appl. Math.},
  title    = {A new nonlinear integration formula for {ODE}s},
  year     = {1996},
  issn     = {0377-0427,1879-1778},
  number   = {2},
  pages    = {291--299},
  volume   = {67},
  doi      = {10.1016/0377-0427(94)00128-6},
  fjournal = {Journal of Computational and Applied Mathematics},
  mrclass  = {65L20},
  mrnumber = {1390186},
}

@Article{1989a_Evans,
  author   = {Evans, D. J.},
  journal  = {Appl. Math. Lett.},
  title    = {New {R}unge-{K}utta methods for initial value problems},
  year     = {1989},
  issn     = {0893-9659,1873-5452},
  number   = {1},
  pages    = {25--28},
  volume   = {2},
  doi      = {10.1016/0893-9659(89)90109-2},
  fjournal = {Applied Mathematics Letters. An International Journal of Rapid Publication},
  mrclass  = {65L05},
  mrnumber = {989853},
}

@Article{1994_Sanugi,
  author     = {Sanugi, B. B. and Evans, D. J.},
  journal    = {Int. J. Comput. Math.},
  title      = {A new fourth order runge-kutta formula based on the harmonic mean},
  year       = {1994},
  number     = {1-2},
  pages      = {113--118},
  volume     = {50},
  doi        = {10.1080/00207169408804246},
  fjournal   = {International Journal of Computer Mathematics},
  mrclass    = {65L05},
  mrnumber   = {937565},
  mrreviewer = {G.\ K.\ Gupta},
}

@Article{1995_Evans,
  author    = {D. J. Evans and A. R. Yaakub},
  journal   = {Int. J. Comput. Math.},
  title     = {A new fourth order runge-kutta formula based on the contra-harmonic {$C_0M$} mean},
  year      = {1995},
  number    = {3-4},
  pages     = {249--256},
  volume    = {57},
  doi       = {10.1080/00207169508804428},
  fjournal  = {International Journal of Computer Mathematics},
  publisher = {Taylor \& Francis},
}

@Article{2008_Villatoro,
  author     = {Villatoro, Francisco R.},
  journal    = {J. Difference Equ. Appl.},
  title      = {Exact finite-difference methods based on nonlinear means},
  year       = {2008},
  issn       = {1023-6198,1563-5120},
  number     = {7},
  pages      = {681--691},
  volume     = {14},
  doi        = {10.1080/10236190701736708},
  fjournal   = {Journal of Difference Equations and Applications},
  mrclass    = {65L12},
  mrnumber   = {2431869},
  mrreviewer = {Antonia\ Vecchio},
}

@Article{2009_Villatoro,
  author     = {Villatoro, Francisco R.},
  journal    = {J. Comput. Appl. Math.},
  title      = {Given a one-step numerical scheme, on which ordinary differential equations is it exact?},
  year       = {2009},
  issn       = {0377-0427,1879-1778},
  number     = {2},
  pages      = {1058--1065},
  volume     = {223},
  doi        = {10.1016/j.cam.2008.03.038},
  fjournal   = {Journal of Computational and Applied Mathematics},
  mrclass    = {65L05 (65L06)},
  mrnumber   = {2478901},
  mrreviewer = {Philip\ W.\ Sharp},
}

@Article{2010b_Villatoro,
  author   = {Villatoro, Francisco R.},
  journal  = {Int. J. Comput. Math.},
  title    = {Local error analysis of {E}vans-{S}anugi, nonlinear one-step methods based on {$\theta$}-means},
  year     = {2010},
  issn     = {0020-7160,1029-0265},
  number   = {5},
  pages    = {1009--1022},
  volume   = {87},
  doi      = {10.1080/00207160802255730},
  fjournal = {International Journal of Computer Mathematics},
  mrclass  = {65L05 (26E60 34L30 65L99)},
  mrnumber = {2665710},
}

@Article{2010a_Villatoro,
  author   = {Villatoro, Francisco R.},
  journal  = {Int. J. Comput. Math.},
  title    = {Stability by order stars for non-linear theta-methods based on means},
  year     = {2010},
  issn     = {0020-7160,1029-0265},
  number   = {1-3},
  pages    = {226--242},
  volume   = {87},
  doi      = {10.1080/00207160802036858},
  fjournal = {International Journal of Computer Mathematics},
  mrclass  = {65L20},
  mrnumber = {2598738},
}

@Article{2010_Kosmas,
  author   = {Kosmas, O. T. and Vlachos, D. S.},
  journal  = {Comput. Phys. Comm.},
  title    = {Phase-fitted discrete {L}agrangian integrators},
  year     = {2010},
  issn     = {0010-4655},
  number   = {3},
  pages    = {562--568},
  volume   = {181},
  doi      = {10.1016/j.cpc.2009.11.005},
  fjournal = {Computer Physics Communications. An International Journal and Program Library for Computational Physics and Physical Chemistry},
  mrclass  = {65P10 (65L05)},
  mrnumber = {2578165},
}

@Article{2019_Kosmas,
  author     = {Kosmas, Odysseas and Leyendecker, Sigrid},
  journal    = {Adv. Comput. Math.},
  title      = {Variational integrators for orbital problems using frequency estimation},
  year       = {2019},
  issn       = {1019-7168,1572-9044},
  number     = {1},
  pages      = {1--21},
  volume     = {45},
  doi        = {10.1007/s10444-018-9603-y},
  fjournal   = {Advances in Computational Mathematics},
  mrclass    = {65L06},
  mrnumber   = {3915000},
  mrreviewer = {Philip\ W.\ Sharp},
}

@Article{2002_Ixaru,
  author     = {Ixaru, L. Gr. and Vanden Berghe, G. and De Meyer, H.},
  journal    = {J. Comput. Appl. Math.},
  title      = {Frequency evaluation in exponential fitting multistep algorithms for {ODE}s},
  year       = {2002},
  issn       = {0377-0427,1879-1778},
  number     = {1-2},
  pages      = {423--434},
  volume     = {140},
  doi        = {10.1016/S0377-0427(01)00474-5},
  fjournal   = {Journal of Computational and Applied Mathematics},
  mrclass    = {65L05},
  mrnumber   = {1934453},
  mrreviewer = {T.\ E.\ Simos},
}

@Article{2012_DAmbrosio,
  author   = {D'Ambrosio, R. and Esposito, E. and Paternoster, B.},
  journal  = {Appl. Math. Comput.},
  title    = {Exponentially fitted two-step {R}unge-{K}utta methods: construction and parameter selection},
  year     = {2012},
  issn     = {0096-3003,1873-5649},
  number   = {14},
  pages    = {7468--7480},
  volume   = {218},
  doi      = {10.1016/j.amc.2012.01.014},
  fjournal = {Applied Mathematics and Computation},
  mrclass  = {65L06},
  mrnumber = {2892715},
}

@Article{2015_Ramos,
  author   = {Ramos, Higinio and Singh, Gurjinder and Kanwar, V. and Bhatia, Saurabh},
  journal  = {Appl. Math. Comput.},
  title    = {Solving first-order initial-value problems by using an explicit non-standard {$A$}-stable one-step method in variable step-size formulation},
  year     = {2015},
  issn     = {0096-3003,1873-5649},
  pages    = {796--805},
  volume   = {268},
  doi      = {10.1016/j.amc.2015.06.119},
  fjournal = {Applied Mathematics and Computation},
  mrclass  = {65L05 (65L06 65L20)},
  mrnumber = {3399464},
}

@Book{2026s_Jung_SIAM,
  author    = {Jung, Kiyuob},
  publisher = {v1.3.0, Zenodo},
  title     = {{specular-differentiation}},
  year      = {2026},
  doi       = {10.5281/zenodo.18246734},
  url       = {https://github.com/kyjung2357/specular-differentiation},
}

\end{document}